\documentclass[11pt,oneside,english,article]{amsart}
\usepackage[T1]{fontenc}
\usepackage[latin9]{inputenc}
\usepackage{geometry}
\usepackage{mathrsfs}
\usepackage{amsthm}
\usepackage{amstext}
\usepackage{amssymb}
\usepackage{esint}
\usepackage{color}
\allowdisplaybreaks[1]

\def\Ma{{\mathcal M}}
\def\Pa{{\mathcal P}}
\makeatletter
\numberwithin{equation}{section}
\numberwithin{figure}{section}
\theoremstyle{plain}
\newtheorem{thm}{Theorem}[section]
  \theoremstyle{definition}
  
  \theoremstyle{plain}
  \newtheorem{prop}[thm]{Proposition}
  \theoremstyle{plain}
  \newtheorem{cor}[thm]{Corollary}
  \theoremstyle{plain}
  \newtheorem{lem}[thm]{Lemma}
   \theoremstyle{plain}
  \newtheorem{hypo}[thm]{Hypothesis}
     \theoremstyle{plain}
  \newtheorem{rem}[thm]{Remark}
  
  \def\deg{{{\rm deg}}}
\def\ra{{\rightarrow}}
\makeatother

\usepackage{babel}

\newcommand{\R}{\mathbb{R}}
\def\Tr{{\rm Tr}}
\newcommand{\XXi}{{\mbox{\boldmath$\Xi$}}}

\newcommand{\zzeta}{{\mbox{\boldmath$\zeta$}}}
\newcommand{\YY}{{\mbox{{\bf Y}}}}
\newcommand{\yy}{{\mbox{{\bf y}}}}
\newcommand{\zz}{{\mbox{{\bf z}}}}
\newcommand{\operator}[1]{\mathsf{#1}}
\def\D{{\mathcal D}}
\def\W{{\mathcal W}}
\def\Tr{{\rm Tr}}
\def\tr{{\rm Tr}}

\def\ot{{\otimes}}
\def\ra{{\rightarrow}}
\def\La{{\mathcal L}}
\def\Ta{{\mathcal T}}
\def\LL{\mathscr{L}}
\def\AA{\mathscr{A}}
\def\BB{\mathscr{B}}

\def\a{\alpha}
\def\one{{\bf 1}}
\def\RR{\mathcal{R}}

 \newcommand{\negint}{{\int\negthickspace\negthickspace\negthickspace\negthickspace-}}

 \newcommand{\nnegint}{{\int\negthickspace\negthickspace\negthickspace-}}

\begin{document}

\title[Universality  in several-matrix models via approximate transport maps]{Universality in several-matrix models\\ via approximate transport maps}

\author{A. Figalli$^\dagger$, A. Guionnet$^{\ddagger}$ }

\thanks{
$^\dagger$ The University of Texas at Austin,
Mathemtics Dept. RLM 8.100,
2515 Speedway Stop C1200,
Austin, Texas 78712-1202, USA. email: \texttt{figalli@math.utexas.edu}.\\
$^{\ddagger}$ Department of Mathematics, Massachusetts Institute of Technology, 77 Massachusetts Ave, Cambridge, MA 02139-4307 USA. email:
\texttt{guionnet@math.mit.edu}.}

\begin{abstract}
We construct approximate transport maps  for perturbative  several-matrix models.
As a consequence, we deduce that local statistics have the same asymptotic 
 as in the case of independent GUE or GOE matrices (i.e., they are given by the sine-kernel in the bulk and the Tracy-Widom distribution at the edge),
 and we show averaged energy universality (i.e., universality for averages of $m$-points correlation functions around some energy level $E$ in the bulk).
As a corollary, these results yield universality for self-adjoint polynomials in several independent GUE or GOE matrices which are close to the identity.
\end{abstract}
\maketitle

\tableofcontents

\section{Introduction.}
Large random matrices appear  in many different fields, including quantum mechanics,  quantum chaos, telecommunications, finance, and   statistics. 
As such, understanding how the asymptotic properties of the spectrum depend on the fine details of the model, in particular on the  distribution of the entries,
  soon appeared as a central question. 

An important model is the one of Wigner matrices, that is Hermitian matrices with independent and identically distributed real or complex entries. We will denote by $N$ the dimension of
the matrix, and assume that the entries are renormalized to have covariance $N^{-1}$.  It was
shown by Wigner 
\cite{Wig62}  that the macroscopic distribution of the spectrum converges, under very mild assumptions, to
the so-called semi-circle law. However, because the spectrum is a complicated function of the entries,
its local
properties took much longer to be revealed.
The first approach to the study of local fluctuations of the spectrum was 
based on exact models, namely the Gaussian models,   where the joint law of the eigenvalues has a simple description as a Coulomb Gas law \cite{ME,TW,TW2,Forr,Deiftbook}.
There, it was shown that the largest eigenvalue fluctuates around the boundary of the support of the semi-circle law
in the scale $N^{-2/3}$, and that the limit distribution of these fluctuations were given by the so-called Tracy-Widom law \cite{TW,TW2}. 
On the other hand, inside the bulk the distance between two consecutive eigenvalues is of order $N^{-1}$ and the fluctuations at this scale can be described by the
sine-Kernel distribution. 
Although this precise description was first obtained only for the Gaussian models, it
was already envisioned by Wigner that these fluctuations should be universal, i.e., 
independent of the precise distribution of the entries.

Recently, a series of remarkable breakthroughs 
\cite{Las,EP, ESYY,EYY,LY,TV1,TV2,TV3,T1}  proved that, under rather general assumptions, the local statistics of a Wigner matrix
 are  independent of the precise distribution of the entries, provided they have enough finite moments, are centered and with the same variance. These results were  extended
 to  the case where distribution of the entries depend
 on the indices, still assuming that their variance is uniformly bounded below \cite{erdy}. The study of band-matrices is still a challenge when the width of the band approaches the critical order of $\sqrt{N}$, see related works \cite{ShT, erdy2}.
Such universality results
were also extended to non-normal square matrices
 with independent entries \cite{taocirc3}.

 A related question is to study universality for local fluctuations for the so-called $\beta$-models, that are laws of particles in interaction according to a Coulomb-gas potential to the power $\beta$ and submitted to a potential $V$. When $\beta=1,2,4$
 and $V$ is quadratic, these laws correspond to the joint law of the eigenvalues of Gaussian matrices with real, complex, or symplectic entries. Universality was proven for very general potentials  in the case $\beta=2$ \cite{LL,Lub}. In the case $\beta=1,4$, universality was proved in  \cite{DeGibulk} in the bulk, and \cite{DeGiedge} at the edge, for monomials $V$ (see \cite{DeGi2} for a review).    For general one-cut potentials, the first proof of universality was given in \cite{Shch} in the case $\beta=1$, whereas \cite{KrSh} treated the case $\beta=4$.
The local fluctuations of more general $\beta$-ensembles were only derived recently \cite{VV,RRV} in the Gaussian case. 
Universality in the $\beta$-ensembles was first addressed in \cite{BEY1} (in the bulk, $\beta> 0$, $V\in C^4$), then in  \cite{BEY2} (at the edge, $\beta\ge 1$, $V\in C^4$), 
\cite{KRV} (at the edge, $\beta>0$, $V$ convex polynomial), and  finally in  \cite{Shch} (in the bulk, $\beta>0$, $V$ analytic, multi-cut case included) and in \cite{BFG}  (in the bulk and the edge, $V$ smooth enough). The universality at the edge in the several-cut case is treated in \cite{Bek}. The case where the interaction is more general than a Coulomb gas, but given by a mean-field interaction $\prod_{i< j} \varphi(x_i-x_j)$ where $\varphi(t)$ behaves as $|t|^\beta$ in a neighborhood of the origin
and $\log |x|^{-\beta}\varphi(x)$ is real analytic as well as the potential, was considered in \cite{GoVe} ($\beta=2$, universality in the bulk), \cite{Ven} ($\beta>0$, universality in the bulk), and \cite{KrVe} ($\beta=2$, universality at the edge).\\

Despite all these new developments, up to now nothing was known about the universality of the fluctuations of the eigenvalues in several-matrix models, except in very particular situations. The aim of this paper is to provide new universality  results for  general
perturbative several matrix models, giving a firm mathematical ground  to the widely spread belief coming from physics that universality of local fluctuations should hold, at least until some phase transition occurs.

An important application of our results  is given by polynomials in Gaussian Wigner matrices and deterministic matrices. More precisely, let $X^{N}_1,\ldots,X^{N}_d$ be $N\times N$ independent GUE matrices,
that is $N\times N$ Hermitian matrices with independent complex Gaussian entries with covariance $1/N$, and let   $B^N_1,\ldots, B^N_m$ 
be  $N\times N$ Hermitian deterministic matrices.  Assume that
for any choices of $i_1,\ldots,i_k\in\{1,\ldots,m\}$ and $k\in\mathbb N$,
\begin{equation}\label{toto1}\frac{1}{N}\tr(B^N_{i_1}\cdots B^N_{i_k})\end{equation}
converges to some limit $\tau(b_{i_1}\cdots b_{i_k})$, where $\tau$ is a linear form on the set of polynomials in the variables $\{b_\ell\}_{1\le \ell\le m}$
that inherits properties of the trace (such as positivity, mass
 one, and traciality, see \eqref{trace}),
and it is called a ``tracial state'' or a ``non-commutative distribution'' in free probability.

A key result due to Voiculescu \cite{Voi91} shows the existence of a non-commutative distribution  $\sigma$
 such
that for any polynomial $p$ 
 in $d+m$ self-adjoint non-commutative variables
$$\lim_{N\ra\infty}\frac{1}{N}\tr\bigl(p(X^{N}_1,\ldots,X^{N}_d, B^N_1,\ldots, B^N_m)\bigr)=\sigma\bigl(p(S_1,\ldots,S_d, b_1,\ldots, b_m)\bigr)\qquad \mbox{a.s.}$$
where, under $\sigma$, $S_1,\ldots,S_d$ are $d$ free semicircular variables, free from $b_1,\ldots,b_m$ with law $\tau$.  
More recently, Haagerup and Thorbj{\o}rnsen \cite{HT05}  (when the matrices $\{B^N_i\}_{1\le i\le m}$ vanish) and then Male \cite{male} (when the spectral radius of
polynomials  $p(B^N_1,\ldots, B^N_m)$ in $\{B^N_i\}_{1\le i\le m}$ converge to the norm of their limit $p(b_1,\ldots,b_m)$) showed that this convergence is also true for the operator norms, namely the following convergence holds almost surely:
$$\lim_{N\ra\infty}\|p(X^{N}_1,\ldots,X^{N}_d, B^N_1,\ldots,B^N_m)\|_\infty=\| p(S_1,\ldots,S_d,b_1,\ldots,b_m)\|_\infty\,,$$
where
$$ \| p(S_1,\ldots,S_d,b_1,\ldots,b_m)\|_\infty=\lim_{r\ra\infty} \sigma\Bigl( \bigl(p(S_1,\ldots,S_d,b_1,\ldots, b_m)p(S_1,\ldots,S_d,b_1,\ldots,b_m)^*\bigr)^{r}\Bigr)^{\frac{1}{2r}}
\,.$$
However, it was not known in general how the eigenvalues of such a polynomial  fluctuate locally.

In this paper we show that if
$p$ is a perturbation of $x_1$ then, under some weak additional assumptions on the deterministic matrices $B^N_1,\ldots,B^N_m$, the eigenvalues of $p(X^{N}_1,\ldots,X^{N}_d,B^N_1,\ldots,B^N_m)$ fluctuate as the eigenvalues of $X_1^N$. 
In particular, if
 $p(X_1,\ldots,X_d)=X_1+\epsilon\, Q(X_1,\ldots,X_d)$ with $\epsilon$ small enough and $Q$ self-adjoint, then we can show that,
once properly renormalized,  the  fluctuations of the eigenvalues of $p(X^{N}_1,\ldots,X^{N}_d)$ 
 follow the sine-kernel inside the
bulk and the Tracy-Widom law at the edges. In addition, this universality result holds also for (averages with respect to $E$ of) $m$-points correlation functions around some energy level $E$ in the bulk.
Furthermore, all these results extend to the case of GOE matrices.
\smallskip

Although we shall not investigate this here,  our results 
should extend to non-Gaussian entries at least when the entries have the same first four moments 
as the Gaussian. 
This would however  be a non-trivial generalization, as it would involve fine analysis such as the local law and rigidity.

To our knowledge this type of result is completely new except in the case of the very specific polynomial $p(S,b)=b+S$, which was recently treated 
in non-perturbative situations  \cite{CP,LSSY} or $p$ is a product of  non-normal random matrices \cite{LW,AS}.
Notice that although our results hold only in a perturbative setting, it is clear that some assumptions
on $p$ are needed and universality cannot hold for any polynomial.
Indeed, even if one considers
only one matrix, if $p$ is not strictly increasing then the largest eigenvalue of $p(X_1^N)$ could be the image by $p$ of an eigenvalue of $X^N_1$  inside the bulk, hence it would follow the sine-kernel law instead of the Tracy-Widom
law.\\

Our approach to universality for polynomials in several matrices goes through the universality for
unitarily invariant matrices interacting via a potential. Indeed, as shown in Section \ref{lawofpol}, the law of the eigenvalues
of such polynomials is a special case of the latter models, that we describe now.

Let $V$ be a polynomial in non-commutative variables, $W_1,\ldots,W_d:\R \to \R$ be  smooth functions, and
 consider the following probability measure on the space of $d$-uple of $N\times N$ Hermitian or symmetric  matrices (see also Section \ref{sect:statement} for more details):
$$d{\mathbb P}^{N,V}_\beta(dX_1,\ldots,dX_d) =\frac{1}{Z^{N,V}_\beta} e^{ N\,\Tr V(X_1,\ldots,X_d,B_1,\ldots,B_m)} e^{-N\sum_{k=1}^d \Tr W_k(X_k)}\prod_{i=1}^d \one_{\|X_i\|_\infty \le M} \,dX,$$
where $dX=dX_1\ldots dX_d$ is the Lebesgue measure on the set of  $d$-uple of $N\times N$ Hermitian or symmetric matrices (from now on, to simplify the notation, we remove the superscript $N$ on $X_i$ and $B_i$).  $M>0$ is a cut-off which ensures that 
$$
Z^{N,V}_\beta:=\int e^{ N\,\Tr V(X_1,\ldots,X_d,B_1,\ldots,B_m)} e^{-N\sum_k \Tr W_k(X_k)}\prod_{i=1}^d \one_{\|X_i\|_\infty \le M}\, dX
$$
is finite despite the fact that $V$ is a polynomial which could go to infinity
faster than the $W_k$'s.  We assume that $V$ is self-adjoint in the sense that $V(X_1,\ldots,X_d,B_1,\ldots,B_m)$ is Hermitian (resp. symmetric)  for any $N\times N$ Hermitian (resp. symmetric)
 matrices $X_1,\ldots,X_d,B_1,\ldots,B_m$. As a consequence,
$\mathbb P^{N,V}_\beta$ has a real non-negative density. Since we shall later need to assume $V$ small, we shall not try to get the best assumptions on the $W_k$'s,
and we shall assume that they are uniformly convex. As discussed in Remark \ref{rmk:Wk} below, this could be relaxed.

Such multi-matrix models appear in physics, in connection with the  enumeration of colored maps \cite{BIPZ, MehIs,kos,eybo},
and in  planar algebras and the Potts model on random graphs \cite{GJS,GJSZ}. However, 
despite the introduction of biorthogonal 
polynomials \cite{Ber} to compute precisely observables in these models, the local properties of the spectrum in these models could not be studied so far, except in very specific 
situations \cite{ABK}.   Our proof  shows that the limiting spectral measure of the matrix models has a connected support and behaves as a square root at the boundary when $a$ is small enough and the $W_k$  are uniformly convex, see Lemma \ref{lem:support}. This in particular shows that in great generality  the $n$-th moments for the related models, which can be identified with  generating functions for planar maps,  grow like $C^n n^{-3/2}$, as for the semi-circle law and rooted trees. 
More interesting  exponents could be found at criticality, a case that we can hardly study in this article  since we need $a$ to be small.  The transport maps between the limiting measures could themselves provide valuable combinatorial information, as a way to analyze the limiting spectral measures,  but they would also need to be extended to criticality too. Yet, the extension of our techniques to the non-commutative setting yields interesting isomorphisms of  related algebras \cite{GS, Nel}. \smallskip

In 
\cite{GMS1,GMS2} it was shown that there exists $M_0<\infty$ such that the following holds: for $M>M_0$
there exists $a_0>0$ so that,
for $a\in [-a_0,a_0]$, there is a non-commutative distribution $\tau^{aV}$ satisfying
$$
\lim_{N\ra\infty}\mathbb P^{N,aV}_\beta\left(\frac{1}{N}\Tr \bigl(p( X_1,\ldots,X_d)\bigr)\right)
=\tau^{aV}(p)
$$
for any polynomials $p$ in $d$ non-commutative letters.
In particular, if $(\lambda_i^k)_{1\le i\le N}$ denote the eigenvalues of $X_k$,
 the spectral measure $L^{N}_k:=\frac{1}{N}\sum_i \delta_{\lambda_i^k}$ converges weakly and in moments towards the probability measure $\mu_k^{aV}$ defined by
\begin{equation}\label{limit}\mu_k^{aV}(x^\ell):=\tau^{aV}\left((X_k)^\ell\right)\qquad \forall\, \ell\in\mathbb N.\,.\end{equation}
Moreover, one can bound these moments  to see that $\mu_k^{aV}$ is compactly supported and hence defined by the family of its moments. In addition, it can be proved that $\mu_k^{aV}$ does not depend on the cutoff $M$.
Furthermore, a central limit theorem for this problem was studied in \cite{GMS2} where it was proved that, for any polynomial~$p$,
$$\Tr\bigl(p(X_1,\ldots,X_d)\bigr)-N\,\tau^{aV}(p)$$
converges in law towards a Gaussian variable.  Higher order expansion (the ``topological expansion'') were derived in \cite{edouard}.
\smallskip

In this article we show that, if $a$ is small enough,
 the local  fluctuations of the eigenvalues of each matrix under $\mathbb P^{N,aV}_\beta$ are the
same as when $a=0$ and  the $W_k$ are just quadratic; in other words, up to rescaling,  they follow the sine-kernel  distribution inside the bulk and the Tracy-Widom law  at the edges of the corresponding ensemble
(see Corollaries \ref{thm:univ} and \ref{thm:univ2}).
In addition, averaged energy universality of the correlation functions  holds in our multi-matrix setting (see Corollary  \ref{thm:energy}).

The idea to prove these results consists in finding a map from the law of the eigenvalues of independent GUE or GOE matrices 
to a probability measure that approximates our  matrix models 
(see Theorem \ref{thm:transport} and Corollary \ref{thm:univ2}).
This approach is inspired by the method introduced in \cite{BFG} to study one-matrix models.
However, not only the arguments here are much more involved,
but we also improve the results in \cite{BFG}. Indeed,
the estimates on the approximate transport map obtained in \cite{BFG} 
 allowed one to obtain universality results only with bounded test functions, and could not be used to show averaged energy universality even in the single-matrix setting.
Here, we are able to show stronger estimates that allow us to deal also with functions 
that grow polynomially in $N$ (see Equation \eqref{aptr}),
and we exploit this to prove averaged energy universality in multi-matrix models (see Corollary \ref{thm:energy}).


A second key (and highly nontrivial) step in our proof consists in showing a large $N$-expansion for integrals over the unitary and orthogonal group (see
Section \ref{matrixintegrals}).  Such integrals arise when one seeks for the joint law of the eigenvalues by simply performing a change of variables and integrating over the eigenvectors. The  expansion of such integrals  was only know up to  the first order \cite{CGM} in the orthogonal case,  and was derived for linear statistics in the case $\beta=2$ in \cite{GN}.  However, to be able to study the law of the eigenvalues of polynomials in several matrices we need to treat quadratic statistics. Moreover, we need to prove that the expansions are smooth functions of the empirical measures of the matrices.  Indeed, such an expansion allows us to express the joint law of the eigenvalues of our matrix models
as the distribution of  mean field interaction models (more precisely, as the distribution of $d$ $\beta$-ensembles interacting via a mean field smooth interaction), and from this representation 
we are able to apply to this setting the approximate transport argument mentioned above, 
and prove our universality results.
\smallskip

In the next section we describe in detail our results.

\section{Statement of the results}
\label{sect:statement}

We are interested in the joint law of the eigenvalues under $\mathbb P^{N,V}_\beta$.
 We shall in fact consider a slightly more general model, where the interaction potential may
 not be linear in the trace, but rather some tensor power of the trace. This is necessary to deal with
 the law of a polynomial in several matrices.
Hence, we consider  the probability measure  {
 $$d {\mathbb P}_\beta^{N,V}(X_1,\ldots,X_d):=\frac{1}{Z_\beta^{N,V}} e^{ N^{2-r} \tr^{\otimes r}  V(X_1,\ldots,X_d, B_1,\ldots, B_m)}\prod_{k=1}^d dR_{\beta,M }^{N,W_k}
 (X_k)$$
 with
 $$dR_{\beta,M}^{N,W}(X):= \frac{1}{Z^{N,W}_{\beta,M}} e^{-N \,\tr(W(X))} \one_{\|X\|_\infty\le M} dX\,,$$
 where $\one_E$ denotes the indicator function of a set $E$, 
 and $Z_\beta^{N,V},Z^{N,W}_{\beta,M}$ are normalizing constants. 
Here:
\begin{itemize}
\item[-] $\beta=2$ (resp. $\beta=1$) corresponds to integration over the Hermitian (resp. symmetric) set $\mathcal H^N_\beta$ of $N\times N$ matrices
with complex (resp. real) entries. In particular $dX=\prod_{1\le j \leq \ell\le N}dX_{\ell j}$ if $\beta=1$, whereas $dX=\prod_{1\le j \leq \ell\le N}d\Re(X_{\ell j})\prod_{1\le j<\ell\le N} d\Im (X_{\ell j})$ if $\beta=2$.
\item[-] $\Tr$ denotes the trace over $N\times N$ matrices, that is, $\Tr A=\sum_jA_{jj}$.
\item[-] $W_k:\mathbb R\to \mathbb R$ are uniformly convex functions, that is
$$
W''_k(x)\ge c_0>0\qquad\forall \,x \in\mathbb R,
$$
and given a function $W:\R \to \R$ and a $N\times N$ Hermitian matrix $X$, we define $W(X)$ as
$$
W(X):=U W(D) U^*,
$$
where $U$ is a unitary matrix which diagonalize $X$ as $X=UDU^*$, and 
$W(D)$ is the diagonal matrix  with entries $\bigl(W(D_{11}),\ldots, W(D_{NN})\bigr)$.
\item[-]
$B_1,\ldots, B_m$ are Hermitian (resp. symmetric) matrices if $\beta =2$ (resp. $\beta=1$).
\item[-] $\mathbb C\langle x_1,\ldots,x_d, b_1,\ldots, b_m\rangle^{\otimes r}$ denote the space of $r$-th tensor product of polynomials in $d$ non-commutative
variables 
with complex (resp. real) coefficients when $\beta=2$ (resp. $\beta=1$). 
For $p\in \mathbb C\langle x_1,\ldots,x_d, b_1,\cdots, b_m\rangle^{\otimes r}$ we denote by $$p=\sum \langle p,q_1\otimes q_2 \cdots\otimes q_r \rangle \,q_1\otimes q_2 \cdots\otimes q_r$$ its decomposition on the monomial basis, and
 let $p^*$ denote its adjoint given by
$$p^*:=\sum \overline{\langle p,q_1\otimes q_2 \cdots\otimes q_r \rangle}\, q_1^*\otimes q_2^* \cdots\otimes q_r^*,$$
where $*$ denotes the involution given by
$$(Y_{i_1}\cdots Y_{i_\ell})^*=Y_{i_\ell}\cdots Y_{i_1}\qquad \forall\, i_1,\ldots,i_\ell\in\{1,\ldots,d+m\},$$
where $\{Y_i=X_i\}_{1\le i\le d}$ and $\{Y_{j+d}=B_j\}_{1\le j\le m}$.
We take $V$ to belong to the closure  of  $\mathbb C\langle x_1,\ldots,x_d, b_1,\ldots, b_m\rangle^{\otimes r}$ for the norm given, for $\xi>1$ and $\zeta\ge 1$,
by
\begin{equation}\label{tyu}
\|p\|_{\xi,\zeta}:=\sum |\langle p,q_1\otimes q_2 \cdots\otimes q_r \rangle| \xi^{\sum_{i=1}^r \deg_X(q_i)}\zeta^{\sum_{i=1}^r\deg_{B}(q_i)}\end{equation}
where $\deg_X(q)$ (resp. $\deg_{B}(q)$) denotes the number of letters $\{X_i\}_{1\le i\le d}$ (resp. $\{B_i\}_{1\le i\le m}$)  contained in $q$. If $p$ only depends on the $X_i$ (resp. the $B_i$), its norm does not depend on $\zeta$ (resp. $\xi$)
and we simply denote it $\|p\|_\xi$ (resp. $\|p\|_\zeta$). 
We also assume that $V$ is self-adjoint, that is $V(X_1,\ldots,X_d, B_1,\ldots, B_m)^*=V(X_1,\ldots,X_d, B_1,\ldots, B_m)$.
\item[-] We use $\|\cdot\|_\infty $ to denote the spectral radius norm. 
\end{itemize}
Performing the change of variables $X_k\mapsto U_k D(\lambda^k) U_k^*$,  with $U_k$ unitary and $D(\lambda^k)$ the diagonal matrix with entries $\lambda^k:=(\lambda_1^k,\ldots\lambda_N^k)$, we find that the joint law of the eigenvalues is given by
\begin{equation}\label{poo}
dP^{N,V}_\beta(\lambda^1,\ldots,\lambda^d)=\frac{1}{\tilde Z^{N,V}_{\beta}} I^{N,V}_\beta(\lambda^1,\ldots, \lambda^d)  \prod_{k=1}^ddR^{N,W_k}_{\beta,M}(\lambda^k)
\end{equation}
where $$I^{N,V}_\beta(\lambda^1,\ldots, \lambda^d) :=\int e^{N^{2-r} \tr^{\otimes r} V\left(U_1D(\lambda^1)U_1^*,\ldots, U_d D(\lambda^d) U_d^*,B_1,\ldots, B_m\right)} dU_1\ldots dU_d,$$
$dU$ being the Haar measure on the unitary group when $\beta=2$ (resp. the orthogonal group when $\beta=1$), $\tilde Z^{N,V}_\beta>0$ is a normalization constant, and $R^{N,W}_{\beta,M}$ is the probability measure on $\mathbb R^N$ given by 
\begin{equation}
\label{eq:RNW}
dR^{N,W}_{\beta,M}(\lambda):=\frac{1}{Z^{N,W}_{\beta,M}}\prod_{i<j}|\lambda_i-\lambda_j|^\beta e^{-N\sum_{i=1}^N W(\lambda_i)}\prod_{i=1}^N 1_{|\lambda_i|\le M} d\lambda_i\,,\qquad \lambda=(\lambda_1,\ldots,\lambda_N).
\end{equation}
As we shall prove in Section \ref{eqmeas}, if $W_k$ are uniformly convex and $V$ is sufficiently small, 
for all $k\in \{1,\ldots,d\}$ the empirical measure $L_k^N$ of the eigenvalues of $X_k$
converges to a compactly supported probability measure $\mu_k^V$. In particular, if the cut-off $M$ is chosen sufficiently large
so that $[-M,M]\supset \supset {\rm supp}(\mu_k^0)$, for $V$ sufficiently small 
$[-M,M]\supset \supset {\rm supp}(\mu_k^V)$ and the limiting measures $\mu_k^V$ will be independent of $M$.
Hence, we shall assume that $M$ is a universally large constant (i.e., the largeness depends only on the potentials $W_k$).
More precisely, throughout the whole paper we will suppose that the following holds: 
\begin{hypo}\label{hypo}Assume that:
\begin{itemize}
\item $W_k:\mathbb R\ra\mathbb R$ is uniformly convex for any $k\in\{1,\ldots, d\}$, that is,
$W_k''(x)\ge c_0>0$ for all $x\in\mathbb R$. Moreover, $W_k\in C^{\sigma}(\mathbb R)$
for some $\sigma \geq 36.$
\item $M> 1$ is a large universal constant.
\item $V$ is self-adjoint and  $\|V\|_{M\xi,\zeta}<\infty$ for some $\xi$ large enough
(the largeness being universal, see Lemma \ref{lem:smooth}) and $\zeta\ge 1$.
\item The spectral radius of the Hermitian matrices $B_1,\ldots,B_m$ is bounded by $1$.
\end{itemize}
\end{hypo}
\begin{rem}\label{rmk:Wk}
{\rm 
The convexity assumption on the potentials $W_k$ could be relaxed. Indeed, the main reasons for this assumption are:\\
- To ensure that the equilibrium measures, obtained as limits of the empirical measure of the eigenvalues,  enjoy the properties described in Section \ref{eqmeas}.\\
- To guarantee that the operator $\XXi_t$ appearing in Proposition \ref{prop:Wk} is invertible.\\
- To prove the concentration inequalities in Section \ref{sect:rest}.\\
- To have rigidity estimates on the eigenvalues, needed in the universality proofs in Section \ref{sect:extensions}.\\
As shown in the papers \cite{BEYbulk,BGK,BFG}, the properties above hold under weaker assumptions on the $W_k$'s.
However, because the proofs of our results are already very delicate, we decided to introduce this convexity assumptions in order to avoid additional technicality that 
would obscure the main ideas in the paper.
}
\end{rem}

In order to be able to apply the approximate transport strategy introduced in \cite{BFG},
a key result we will prove is the following large dimensional expansion of $I^{N,V}_\beta$. 
\begin{thm}\label{mainexp} Under Hypothesis \ref{hypo},  there exists $a_0>0$
so that for  $a\in [-a_0,a_0]$ 
\begin{equation}\label{po}I^{N,aV}_{\beta}(\lambda^1,\ldots, \lambda^k) =\biggl(1+O\Bigl(\frac{1}{N}\Bigr)\biggr)e^{ \sum_{l=0}^2 N^{2-l}  F_l^{a}(L^N_1,\cdots, L^{N}_d,\tau^N_B)  },\end{equation}
where  $L^{N}_k$ are the spectral measures
$$L^{N}_k:=\frac{1}{N}\sum_{i=1}^N\delta_{\lambda^k_i}\,,$$
$O(\frac{1}{N})$ depends only on $M$, $\tau^N_B$ denotes the non-commutative distribution
of the $B_i$'s given  by the collection of complex numbers 
\begin{equation}
\label{eq:tauBN}
\tau_B^N(p):=\frac{1}{N}\tr\bigl(p(B_1,\ldots, B_m)\bigr),\qquad p\in\mathbb C\langle b_1,\ldots, b_m\rangle\,,
\end{equation}
and $\{F_l^{a}(\mu_1,\ldots,\mu_d,\tau)\}_{0\le l\le 2}$ are smooth functions of $(\mu_1,\ldots,\mu_d,\tau)$  for the weak topology generated  on 
the space of probability measures $\mathcal P([-M,+M])$ by 
$\|\mu\|_{\zeta M}:=\max_{k\ge 1}(M\zeta)^{-k}|\mu(x^k)|$
and the norm $\sup_{\|p \|_{\zeta}\leq 1}|\tau(p)|$  on linear forms $\tau$ on $\mathbb C\langle b_1,\ldots,b_m\rangle$.\end{thm}
This result is proved in Section \ref{matrixintegrals}.  We notice that it was already 
partially proved in \cite{GN} in the unitary case. However, only the case where $r=1$ was considered there, and the expansion was shown to
hold only in terms of the joint non-commutative distribution of the diagonal matrices $\{D(\lambda^k)\}_{1\le k\le d}$ rather than the spectral measure of each of them.

From the latter expansion of the density of $P^{N,aV}_\beta$ we
can deduce the convergence of
the spectral measures by standard large deviation techniques.
\begin{cor} \label{cor:muk}
Assume that, for any polynomial $p\in\mathbb C\langle b_1,\ldots, b_m\rangle$,
\begin{equation}
\label{eq:convtauBN}
\lim_{N\ra\infty}\tau^N_B(p)=\tau_B(p).
\end{equation}
Then, under Hypothesis \ref{hypo}, there exists $a_0>0$
such that, for  $a\in [-a_0,a_0]$, the empirical measures $\{L^{N}_k\}_{1\le k\le d}$ converge almost surely under $ P^{N,aV}_\beta$
towards  probability measures $\{\mu^{aV}_{k}\}_{1\le k\le d}$ on the real line.
\end{cor}
In the case $r=1$ this result is already  a  consequence 
of \cite{GMS1} and \cite{CGM}. The existence and study of the equilibrium measures is performed in Section
\ref{eqmeas}.

Starting from the representation of the density given in Theorem \ref{mainexp} (see Section \ref{transport}), we are able to prove the following existence results on approximate transport maps:
\begin{thm}
\label{thm:transport}
Under Hypothesis \ref{hypo} with $\zeta>1$, 
suppose additionally that 
\begin{equation}\label{expB}
\tau_B^N(p)=\tau_B^0(p)+\frac{1}{N}\tau_B^1(p)+\frac{1}{N^2}\tau_B^2(p)+O \biggl(\frac{1}{N^3}\biggr)\end{equation}
where the error is uniform on balls for $\|\cdot\|_{\zeta}$.
Then there exists a constant $\alpha>0$ such that, provided $|a|\leq \alpha$,
we can construct a map
$$T^N=\left( (T^{N})_1^1,\ldots,(T^{N})_N^1,\ldots,(T^{N})_1^d,\ldots, (T^{N})_N^d\right):\mathbb R^{dN}\ra\mathbb R^{dN}$$
satisfying the following property:
Let  $\chi:\mathbb R^{dN}\to \mathbb R^+$ be a nonnegative measurable function
such that $\|\chi\|_\infty \leq N^k$ for some $k \geq 0$. Then, for any $\eta>0$, we have
\begin{equation}\label{aptr}\left|\log\biggr(1+\int \chi\circ T^N\, d P^{N,0}_\beta\biggr)-
\log\biggl(1+\int \chi\, d P^{N,aV}_\beta\biggr)\right|\le C_{k,\eta}\,N^{\eta-1}
\end{equation}
for some constant $C_{\eta,k}$ independent of $N$.
Moreover $T^N$ 
 has the form
$$(T^{N})_i^k(\hat\lambda)= T_0^k(\lambda_i^k) +\frac{1}{N} (T_1^{N})_i^k(\hat\lambda)\qquad \forall\,i=1,\ldots,N,\,k=1,\ldots,d,\qquad \hat\lambda:=(\lambda_1^1,\ldots,\lambda_N^d),$$
where $T_0^k:\R\to \R$ and $T_1^{N}:\R^{dN}\to \R^{dN}$ are of class $C^{\sigma-3}$ and satisfy uniform (in $N$) regularity estimates. More precisely,
we have the decomposition  $T^{N}_1=X^{N}_{1,1}+
\frac{1}{N} X_{2,1}^N$ where
\begin{equation}\label{boi1int}
\max_{1\leq k\leq d,\,1\leq i\leq N}\|(X_{1,1}^{N})_i^k\|_{L^4( P_\beta^{N,0})} \leq C\,\log N,\qquad
\max_{1\leq k\leq d,\,1\leq i\leq N} \|(X_{2,1}^{N})_i^k\|_{L^2(P_\beta^{N,0})} \leq C\,(\log N)^2,
\end{equation}for some constant $C>0$ independent of $N$.
In addition, with $ P^{N,0}_\beta$-probability greater than $1- e^{-c(\log N)^2}$,
$$
\max_{i,k}\bigl|(X_{1,1}^N)_i^k\bigr|\leq C\,\log N \,N^{1/(\sigma-14)},
\qquad 
\max_{i,k}\bigl|(X_{2,1}^N)_i^k\bigr|\leq C\,(\log N)^2 \,N^{2/(\sigma-15)},
$$
$$
\max_{1\leq i,i'\leq N}\left|(X_{1,1}^{N})_i^k(\hat \lambda)-(X_{1,1}^{N})_{i'}^k(\hat\lambda)\right|\le C\,\log N\,N^{1/(\sigma-15)} \,|\lambda_i^k-\lambda_{i'}^k|\qquad \forall\,k=1\ldots,d,
$$
$$
\max_{i,i'}\bigl|(X_{2,1}^N)_i^k(\hat\lambda)-(X_{2,1}^N)_{i'}^k(\hat \lambda)\bigr|\le C\,(\log N)^2\,N^{2/(\sigma-17)}|\lambda_i^k-\lambda_{i'}^{k}|\qquad \forall\,k=1,\dots,d,
$$
$$
\max_{1\le i,j\le N}
\left|\partial_{\lambda_j^\ell}(X_{1,1}^N)_i^k\right|(\hat\lambda)\le C\,\log N\,N^{1/(\sigma-15)} \qquad \forall\,k,\ell=1,\dots,d.
$$
\end{thm}
As explained in Section \ref{sect:extensions},
the existence of an approximate transport map satisfying regularity properties
as above allows us to show universality properties for the local fluctuations of the spectrum.
For instance, we can prove the following result:
\begin{cor}
\label{thm:univ}
Under the hypotheses of Theorem \ref{thm:transport} the following holds:
Let $T_0^k$ be as in Theorem \ref{thm:transport} and denote by 
 $\tilde P^{N,aV}_{\beta}$  the distribution of
the increasingly ordered  eigenvalues $(\{\lambda_i^k\}_{1\leq i \leq N}, 1\le k\le d)$ under  the law $P^{N,aV}_\beta$.
Also, let $\mu_k^0,\mu_k^{aV}$ be as in Corollary \ref{cor:muk}, and $\alpha$ as in Theorem \ref{thm:transport}.
Then, for any $\theta\in (0,1/6)$ there exists a constant $\hat C>0$, independent of $N$, such that the following two facts hold true provided $|a|\leq \alpha$:
\begin{itemize}
\item[(1)] Let $\{i_k\}_{1\leq k \leq d} \subset [\varepsilon N,(1-\varepsilon)N]$ for some $\varepsilon >0$. Then, choosing $\gamma_{i_k/N}^k \in \R$ such that
$\mu^0_{k}((-\infty,\gamma_{i_k/N}^k))=i_k/N$,
if $m \leq N^{2/3-\theta}$ then,
 for any bounded Lipschitz function $f:\mathbb R^{dm}\to\mathbb R$,
 \begin{multline*}
\bigg|\int f\Bigl(\bigl(N(\lambda_{i_{k}+1}^k-\lambda_{i_{k}}^k),\ldots,N(\lambda_{i_k+m}^k-\lambda_{i_k}^k)\bigr)_{1\le k\le d}\Bigr)\, d\tilde P^{N,aV}_{\beta}\\
\qquad -\int f\Bigl(\bigl((T_{0}^k)'(\gamma_{i_k/N}^k)\,N(\lambda_{i_k+1}^k-\lambda_{i_k}^k),\ldots,(T_{0}^k)'(\gamma_{i_k/N}^k)\,N(\lambda_{i_k+m}^k-\lambda_{i_k}^k)\bigr)_{1\le k\le d})\Bigr) \,d\tilde P^{N,0}_{\beta}\bigg|\\
 \le \hat C \,N^{\theta-1} \,
\|f\|_\infty + \hat C\,m^{3/2}\,N^{\theta-1}\,
\|\nabla f\|_{\infty}.
\end{multline*}
\item[(2)]
Let $a_k^0$ (resp. $a_{k}^{aV}$) denote the smallest point in the
support of $\mu_{k}^{0}$ (resp.  $\mu_{k}^{aV}$), so that ${\rm supp}(\mu_{k}^{0})\subset [a_{k}^{0},\infty)$
(resp. ${\rm supp}(\mu_{k}^{aV})\subset [a_{k}^{aV},\infty)$).
If $m \leq N^{4/7}$ then, for any bounded Lipschitz function $f:\mathbb R^{dm}\to\mathbb R$,
\begin{multline*}
\bigg|
\int f\Bigl(\bigl(N^{2/3}(\lambda_1^k-a_{k}^{aV}), \ldots,N^{2/3}(\lambda_m^k-a_{k}^{aV})\bigr)_{1\le k\le d}\Bigr) \,d\tilde P^{N,aV}_{\beta}\\
-\int f\Bigl(\bigl((T_{0}^k)'(a_{k}^{0})\,N^{2/3}(\lambda_1^k-a_{k}^{0}), \ldots,(T_{0}^k)'(a_{k}^{0})\,N^{2/3}(\lambda_m^k-a_{k}^{0})\bigr)_{1\le k\le d}\Bigr)\, d\tilde P^{N,0}_{\beta}\bigg|\\
\le \hat C \,N^{\theta-1}\,
\|f\|_\infty+\hat C\,\bigl(m^{1/2}\,N^{\theta-1/3}+ m^{7/6}\,N^{-2/3}\bigr)\,
\|\nabla f\|_{\infty} .
\end{multline*}
The same bound holds around the largest point in the support of $\mu_{k}^{aV}$.
\end{itemize}
\end{cor}
Similar results could be derived with functions of both statistics in the bulk and at the edge.
Let us remark that for $a=0$ the eigenvalues of the different matrices are uncorrelated and
$P^{N,0}_\beta$ 
becomes a product: $dP^{N,0}_\beta=
\prod_{k=1}^ddR^{N,W_k}_{\beta,M}$. Universality under the latter $\beta$-models was already proved in \cite{BEY1,BEY2,Shch,BFG}.
Moreover, by the results in \cite{BFG} we can find approximate transport maps $S_k^N:\R^N\to \R^N$ from
the law $P^N_{{\rm GVE},\beta}$ (this is the law of GUE matrices 
when $\beta=2$ and GOE matrices when $\beta=1$) to
$R^{N,W_k}_{\beta,M}$ for any $k=1,\ldots,d$.
Hence $(S^N_1,\ldots,S^N_d):\R^{dN}\to \R^{dN}$ is an approximate transport from $(P^N_{{\rm GVE},\beta})^{\ot d}$
(i.e., the law of $d$ independent GUE matrices 
when $\beta=2$ and GOE matrices when $\beta=1$) to
$P^{N,0}_\beta$, and this allows us to deduce that the local statistics are in the same universality class as GUE  (resp. GOE) matrices. 

More precisely, as already observed in \cite{BFG}, the leading orders in the transport can be restated in terms of the equilibrium densities:
denoting by
\begin{equation}
\label{eq:semicircle}
\rho_{ {\rm sc}}(x):=\frac1{2\pi} \sqrt{(4-x^2)_+}
\end{equation}
the
density of the semicircle 
distribution and by $\rho^{0}_{k}$ the density of $\mu_k^0$, then the leading order term of $S^N_k$ is given by $(S_0^k)^{\otimes N}$, where $S_0^k:\R \to \R$ 
is the monotone transport from $\rho_{ {\rm sc}}\,dx$ to $\rho^{0}_{k}\,dx$ that can be found solving the ODE
\begin{equation}
\label{eq:Sk}
(S_0^k)'(x)=\frac{\rho_{{\rm sc}}}{\rho^{0}_{k}(S_0^k)}(x),\qquad
S_0^k(-2)=a_k^{0}.
\end{equation}
Also, the transport $T_0^k:\R \to \R$ appearing in  Corollary \ref{thm:univ} solves
\begin{equation}
\label{eq:Tk0}
(T_0^k)'(x)=\frac{\rho^{0}_{k}}{\rho^{aV}_{k}(T_0^k)}(x),\qquad
T_0^{k}(a_k^0)=a_k^{aV}.
\end{equation}
Set 
\begin{equation}
\label{eq:ck}
c_{k}^{aV}:=\lim_{x\to -2^+}\frac{\rho_{\rm sc}}{\rho_{k}^{aV}(T_0^k\circ S_0^k)}(x).
\end{equation}
Thanks to these observations, we can easily prove the following result:
\begin{cor}
\label{thm:univ2}
Let $m \in \mathbb N$.
Under the hypotheses of Theorem \ref{thm:transport} the following holds:
 Denote by
 $\tilde P^{N,aV}_{\beta}$ (resp. $(\tilde P_{{\rm GVE},\beta}^N)^{\otimes d}$) the distribution of
the increasingly ordered  eigenvalues $(\{\lambda_i^k\}_{1\leq i \leq N}, 1\le k\le d)$ under  the law $P^{N,aV}_\beta$ (resp. $(P_{{\rm GVE},\beta}^N)^{\otimes d}$). 
Also, let $\alpha$ be as in Theorem \ref{thm:transport}.
Then, for any $\theta\in (0,1/6)$ and $C_0>0$
there exists a constant $\hat C>0$, independent of $N$, such that the following two facts hold true provided $|a|\leq \alpha$:
\begin{itemize}
\item[(1)] Given $\{\sigma_k\}_{1\leq k \leq d} \subset (0,1)$, let $\gamma_{\sigma_k} \in \R$ be such that $\mu_{\rm sc}((-\infty,\gamma_{\sigma_k}))=\sigma_k$,
and $\gamma_{\sigma_k,k}$ such that $\mu_{k}^{aV}((-\infty,\gamma_{\sigma_k,k}))=\sigma_k$.
Then, if $|i_k/N - \sigma_k| \leq C_0/N$ and $m \leq N^{2/3-\theta}$,
 for any bounded Lipschitz function $f:\mathbb R^{dm}\to \mathbb R$ we have
\begin{multline*}
\bigg|\int f\Bigl(\bigl(N(\lambda_{i_k+1}^k-\lambda_{i_k}^k),\ldots,N(\lambda_{i_k+m}^k-\lambda_{i_k}^k)\bigr)_{1\le k\le d} \Bigr)\, d\tilde P^{N,aV}_{\beta}\\
\qquad -\int f\biggl (\Big(\frac{\rho_{\rm sc}(\gamma_{\sigma_k})}{\rho_{k}^{aV}(\gamma_{\sigma_k,k})} N(\lambda_{i_k+1}^k-\lambda_{i_k}^k),\ldots,\frac{\rho_{\rm sc}(\gamma_{\sigma_k})}{\rho_{k}^{aV}(\gamma_{\sigma_k,k})} N(\lambda_{i_k+m}^k-\lambda_{i_k}^k)\Big)_{1\le k\le d}\biggr) \,d (\tilde P_{{\rm GVE},\beta}^N)^{\otimes d}\bigg|\\
 \le  \hat C
  \,N^{\theta-1} \,
\|f\|_\infty + \hat C\,
\,m^{3/2}\,N^{\theta-1}\,
\|\nabla f\|_{\infty}.
 \end{multline*}
\item [(2)]
Let $c_{k}^{aV}$ be as in \eqref{eq:ck}.
If $m \leq N^{4/7}$ then,
for any bounded Lipschitz function $f:\mathbb R^m\to \mathbb R$,
we have
\begin{multline*}
\bigg|
\int f\Bigl(\bigl(N^{2/3}(\lambda_1^k-a_{k}^{aV}), \ldots,N^{2/3}(\lambda_m^k-a_{k}^{aV})\bigr)_{1\le k\le d}\Bigr) \,d\tilde P^{N,aV}_{\beta}\\
-\int f\Bigl(c_{k}^{aV} \,N^{2/3}\bigl(\lambda_1^k+2\bigr), \ldots,c_{k}^{aV}\,N^{2/3}\bigl(\lambda^k_m+2\bigr)\bigr)_{1\le k\le d}\Bigr)\, d (\tilde P_{{\rm GVE},\beta}^N)^{\otimes d}
\bigg|\\
\le   \hat C \,N^{\theta-1}
\|f\|_\infty+\hat C\,\bigl(m^{1/2}\,N^{\theta-1/3}+ m^{7/6}\,N^{-2/3}\bigr)\,
\|\nabla f\|_{\infty} .\end{multline*}
The same bound holds around the largest point in the support of $\mu_{k}^{aV}$.
\end{itemize}
\end{cor}

While the previous results deal only with bounded test function,
in the next theorem we take full advantage of the estimate \eqref{aptr}
to show averaged energy universality in our multi-matrix setting.
Note that, to show this result, we need to consider as test functions averages (with respect to $E$) of $m$-points correlation functions of the form
$\sum_{i_1\neq \ldots\neq i_m} f\bigl(N(\lambda_{i_1}^k-E),\ldots,N(\lambda_{i_m}^k-E)\bigr)$
where $E$ belongs to the bulk of the spectrum. In particular, these 
test functions have $L^\infty$ norm of size $N^m$.
Actually, as in Corollaries  \ref{thm:univ} and \ref{thm:univ2}, we can deal with 
test functions depending at the same time on the eigenvalues of the different matrices.

Here and in the following, we use $\nnegint_I$ to denote the averaged integral over an interval $I\subset \R$, namely
$\nnegint_I=\frac{1}{|I|}\int_I$.

\begin{cor}
\label{thm:energy}
Fix $m \in \mathbb N$ and $\zeta \in (0,1)$, and let $\alpha$ be as in Theorem \ref{thm:transport}.
Also, let $T_0^k$ and $S_0^k$ be as in \eqref{eq:Tk0} and \eqref{eq:Sk},
and define $R_k:=T_0^k\circ S_0^k$.
Then, given $\{E_k\}_{1\leq k \leq d} \subset (-2,2)$, $\theta \in (0,\min\{\zeta,1-\zeta\})$,
and $f:\R^{dm}\to \R^+$ a nonnegative Lipschitz function with compact support, there exists a constant $\hat C>0$, independent of $N$, such that the following holds true provided $|a|\leq \alpha$:
\begin{align*}
&\bigg|\int\bigg[ 
\negint_{R_1(E_1)-N^{-\zeta}\,R_1'(E_1)}^{R_1(E_1)+N^{-\zeta}\,R_1'(E_1)}d\tilde E_1\ldots\negint_{R_d(E_d)-N^{-\zeta}\,R_d'(E_d)}^{R_d(E_d)+N^{-\zeta}\,R_d'(E_d)}d\tilde E_d\\
&\qquad\qquad\qquad\qquad\sum_{i_{k,1}\neq \ldots\neq i_{k,m}} 
f\Bigl(\bigl(N(\lambda_{i_{k,1}}^k-\tilde E_k),\ldots,N(\lambda_{i_{k,m}}^k-\tilde E_k)\bigr)_{1\leq k \leq d}\Bigr)\bigg]\,dP^{N,aV}_{\beta}\\
&\qquad -\int\bigg[ 
\negint_{E_1-N^{-\zeta}}^{E_1+N^{-\zeta}}d\tilde E_1\ldots\negint_{E_d-N^{-\zeta}}^{E_d+N^{-\zeta}}d\tilde E_d\\
&\qquad\qquad\qquad \sum_{i_{k,1}\neq \ldots\neq i_{k,m}} 
f\Bigl(\bigl(R_k'(E_k)\,\,N(\lambda_{i_{k,1}}^k-\tilde E_k),\ldots,R_k'(E_k)\,N(\lambda_{i_{k,m}}^k-\tilde E_k)\bigr)_{1\leq k \leq d}\Bigr)\bigg] \,d P^N_{{\rm GVE}}\bigg|\\
&\qquad\qquad\qquad\qquad\qquad\qquad\qquad\qquad\qquad\qquad \le  \hat C\,\Bigl(N^{\theta+\zeta-1}+
N^{\theta-\zeta} \Bigr).
 \end{align*}
\end{cor}

It is worth mentioning that, in the single-matrix case, Bourgade, Erdos, Yau, and Yin \cite{BEYYfixed}
have recently been able to remove the average with respect to $E$
and prove the Wigner-Dyson-Mehta conjecture at fixed energy in the bulk of the spectrum for generalized
symmetric and Hermitian Wigner matrices.
We believe that combining their techniques with ours one should be able to remove the average with respect to $E$ in the previous theorem. However, this would go beyond the scope of this paper and we shall not investigate this here.\smallskip

Another consequence of our transportation approach is 
the universality of other observables, such as
the minimum spacing in the bulk.
The next result is restricted to the case $\beta=2$ since we rely on \cite[Theorem 1.4]{BAB} which is proved in the case
$\beta=2$ and is currently unknown for $\beta=1$.

\begin{cor}\label{univmax} Let $\beta=2$, fix $k \in \{1,\ldots,d\}$, let 
  $I_k$ be a compact subset of $(-a_k^{aV},b_k^{aV})$ with non-empty interior, and denote the renormalized gaps by
$$\Delta^{k}_i:= \frac{\lambda_{i+1}^k-\lambda_i^k}{(T_0^k \circ S_0^k)'(\gamma_{i/N})} \,, \qquad\lambda_i^k \in I_k,$$
where $\gamma_{i/N} \in \R$ is such that $\mu_{\rm sc}((-\infty,\gamma_{i/N}))={i/N}$.
 Also, denote by
 $\tilde P^{N,aV}_{\beta,k}$ the distribution of
the increasingly ordered  eigenvalues $\{\lambda_i^k\}_{1\leq i \leq N}$ under $P^{N,aV}_{\beta,k}$, the law
of the eigenvalues of the $k$-{th} matrix under $P^{N,aV}_\beta$.
Then, under the hypotheses of Theorem \ref{thm:transport}, it holds:
\begin{itemize}\item {\rm {Smallest gaps.}}  
Let $\tilde t^1_{N,k}<\tilde t^2_{N,k}\cdots<\tilde t^p_{N,k}$ denote the $p$ smallest renormalized spacings $\Delta_i^k$ of the eigenvalues
of the $k$-th matrix lying in $I$, and set 
$$\tilde\tau^p_{N,k}:=\left(\frac{1}{144\pi^2}\int_{(T_0^k \circ S_0^k)^{-1}(I)}(4-x^2)^2\,dx\right)^{1/3}\tilde t^p_{N,k}.$$
Then, as $N \to \infty,$ $N^{4/3}\tilde\tau^p_{N,k}$ converges in law towards $\tau^p$ whose density is given by
$$\frac{3}{(p-1)!}x^{3p-1} e^{-x^3} dx\,.$$

\item {\rm {Largest gaps.}}  Let $\ell_{N,k}^1(I)>\ell_{N,k}^2(I)>\ldots$ be the largest gaps of the form $\Delta^k_i$ 
with $\lambda_i^k\in I_k$. Let  $\{r_N\}_{N \in \mathbb N}$ be a family of positive integers such that
$$
\frac{\log r_N}{\log{N}} \to 0\qquad \text{as $N\to \infty$}.
$$
Then, as $N \to \infty,$
$$\frac{N}{\sqrt{32\log N}}\ell_{N,k}^{r_N} \to 1 \qquad \text{ in $L^q(\tilde P_{\beta,k}^{N,aV})$}$$
for any $q<\infty$.
\end{itemize}
\end{cor}
All the above corollaries 
are proved in Section \ref{sect:extensions}.

As an important application of our results, we consider the law of the eigenvalues
of a self-adjoint polynomials in several GUE or GOE matrices. Indeed, if $\epsilon$ is sufficiently small
and $X_1,\ldots,X_d$ are independent GUE or GOE matrices,
 a change of variable formula
shows that the law of the eigenvalues of the $d$ random matrices given by
$$Y_i=X_i+\epsilon \,P_i(X_1,\ldots,X_d),\qquad 1\le i\le d,$$
follows a distribution of the form $P^{N,aV}_\beta$ with $r=2$ and $V$ a convergent series, see  Section \ref{lawofpol}. 
Hence we have:
\begin{cor}
\label{cor:P} Let $P_1,\ldots,P_d\in \mathbb C\langle x_1,\ldots,x_d,b_1,\ldots,b_m\rangle$ be self-adjoint polynomials.
There exists $\epsilon_0>0$ such that the following holds:
Let $X_i$ be independent GUE or GOE matrices and set 
$$Y_i:=X_i+\epsilon\, P_i(X_1,\ldots,X_d).$$
Then, for $\epsilon\in [-\epsilon_0,\epsilon_0]$,
the eigenvalues of the matrices $\{Y_i\}_{1\le i\le d}$ fluctuate in the bulk or at the edge as when $\epsilon=0$,
up to rescaling. 
The same result holds for
$$Y_i=X_i+\epsilon \,P_i(X_1,\ldots,X_d,B_1,\ldots, B_m)$$
provided $\tau^N_B$ satisfies \eqref{expB}. Namely, in both models, the law  $\tilde P^{N,\epsilon P}_{\beta}$ of  the ordered eigenvalues of the matrices $Y_k$ satisfies the same conclusions as
$\tilde P^{N,aV}_{\beta}$ in Corollaries \ref{thm:univ2} and \ref{univmax}.

 \end{cor}
 
 \begin{rem}{\rm 
 Recall that, as already stated at the beginning of Section \ref{sect:statement}, when $\beta=1$ the matrices $B_i$ are assumed to be real as well as the coefficients of $P$. In particular, in the statement above, if $X_i$ are GOE then the matrices $Y_i$ must be orthogonal. The reason for that is that 
 we need  the map $(X_1,\ldots, X_d) \mapsto (Y_1,\ldots, Y_d)$ to be an isomorphism close to identity at least for uniformly bounded matrices. Our result should generalize  to mixed polynomials in GOE and GUE which satisfy this property, but it does not include  the case of the perturbation of a GOE matrix by a small  GUE matrix which is Hermitian but not orthogonal. 
 }
 \end{rem}
 
\medskip

\textit{Acknowledgments:} 
AF was partially supported by NSF Grant DMS-1262411 and NSF Grant DMS-1361122. AG was partially supported by the Simons Foundation and by NSF Grant DMS-1307704. The authors would like to thank an anonymous referee for his challenging questions.

\section{Study of the equilibrium measure}\label{eqmeas}
In this section we study the macroscopic behavior of the eigenvalues, that is the convergence of their empirical measures
and the properties of their limits.
Note here that  we are restricting ourselves to measures supported on $[-M,M]$ so that the weak topology is equivalent to the
topology of moments induced by the norm $\|\nu\|_{\zeta M}:=\max_{k\ge 1}(\zeta M)^{-k}|\nu(x^k)|$.
As a consequence, a large deviation principle for the law $\Pi^{N,aV}_\beta$ of $(L^{N}_1,\ldots,L^{N}_d)$ under $ P^{N,aV}_{\beta}$
can be proved:
\begin{lem}\label{ldplem}
Assume that $M>1$ is sufficiently large and that $\tau_B^N$ converges towards $\tau_B$ (see \eqref{eq:tauBN} and \eqref{eq:convtauBN}).
Then the measures $(\Pi^{N,aV}_\beta)_{N\ge 0}$ on ${\mathcal P}([-M,M])^d$ equipped with the weak topology satisfy
a large deviation principle in the scale $N^2$ with good rate function
$$I^{a}(\mu_1,\ldots,\mu_d):= J^a(\mu_1,\ldots,\mu_d)-\inf_{\nu_k\in\mathcal P([-M,M])}J^a(\nu_1,\ldots,\nu_d),$$
where
$$J^a(\mu_1,\ldots,\mu_d):=
\frac{1}{2}\sum_{k=1}^d  \biggl(\iint \bigl[ W_k(x)+W_k(y) -\beta\log|x-y|\bigr]\,d\mu_k(x)\,d\mu_k(y)\biggr)
-F_0^{a}(\mu_1,\ldots,\mu_d,\tau_B)\,. $$
\end{lem}
\begin{proof} The proof is given in \cite{BAG,AGZ} in the case $F_0^{a}=0$, while the general case follows from Laplace method (known also as Varadhan lemma)
since $F_0^{a}$ is continuous for the $\|\cdot\|_{\zeta M} $ topology (and therefore for the usual weak topology, which is stronger).\end{proof}
It follows by the result above that  $\{L^{N}_k\}_{1\le k\le d}$ converge to the minimizers of $I^a$.
We next prove that, for $a$ small enough, $I^a$ admits a unique minimizer, and show some of its properties.
This is an extended and  refined version of \eqref{limit} which shall be useful later on.
\begin{lem}
\label{lem:support} Let Hypothesis \ref{hypo} hold. There  exists $a_0>0$ such that, for $a\in [-a_0,a_0]$,
$I^a$ admits a unique minimizer ($\mu_1^{aV},\ldots,\mu_d^{aV})$. Moreover the support of each $\mu_k^{aV}$ is connected
and strictly contained inside $[-M,M]$, and each $\mu_k^{aV}$ has a density which is smooth and strictly positive inside its support
except at the two boundary points, where it goes to zero as a square root.
\end{lem}
\begin{proof}
We first notice that if $I^a(\mu_1,\ldots,\mu_k)$ is finite, so is $-\int \log|x-y|\,d\mu_k(x)\,d\mu_k(y)$. In particular the minimizers $\{\mu^{aV}_i\}_{1\le i\le d}$ of $I^a$
have no atoms.  We then consider the small perturbation $I^a(\mu_1^{aV}+\varepsilon\nu_1,\ldots,\mu_d^{aV}+\varepsilon\nu_d)$  for centered measures $(\nu_1,\ldots,\nu_d)$ (that is, $\int d\nu_k=0$)
such that $\nu_k \geq 0$ outside the support of $\mu_k^{aV}$ and $\mu_k^{aV}+\varepsilon \nu _k \geq 0$
for $|\varepsilon| {\ll} 1$.
Hence, by differentiating $I^a(\mu_1^{aV}+\varepsilon\nu_1,\ldots,\mu_d^{aV}+\varepsilon\nu_d)$ with respect to $\varepsilon$ and setting $\varepsilon=0$,
we deduce that
\begin{equation}
\label{eq:nuk}
0=\int F_k(x)\,d\nu_k(x),
\end{equation}
where 
$$
F_k(x):=W_k(x)-D_kF_0^{a}(\mu_1^{aV},\ldots,\mu_d^{aV},\tau_B)[\delta_x]-\beta\int\log|x-y|\,d\mu_k^{aV}(y)
$$
and
$x\mapsto D_k F_0^{a}(\mu_1,\ldots,\mu_d,\tau_B) [\delta_x]$ denotes the function such that, for any measure $\nu$,
\begin{multline}
\label{eq:DF}
\frac{d}{d\varepsilon}|_{\varepsilon=0} F_0^{a}(\mu_1^{aV},\ldots,\mu_{k-1}^{aV},\mu_k^{aV}+\varepsilon \nu,\mu_{k+1}^{aV},\ldots,\mu_d^{aV},\tau_B)\\
=\int D_k F_0^{a}(\mu_1^{aV},\ldots,\mu_d^{aV},\tau_B) [\delta_x] \,d\nu(x)\,.
\end{multline}
It is shown in Lemma \ref{lem:smooth} that this function is smooth and of size $a$ (as well as its derivatives). 
Since $\nu_k$ is centered and $\nu_k \geq 0$ outside the support of $\mu_k$,
it follows from \eqref{eq:nuk} that there exists a constant $C_k \in \R$ such that
$$
F_k \left\{
\begin{array}{ll}
=C_k & \text{on ${\rm supp}(\mu_k^{aV})$},\\
\geq C_k & \text{on $\R\setminus {\rm supp}(\mu_k^{aV})$},\\
\end{array}
\right.
$$
Since   $\partial_x^2 \bigl(D_kF_0^a(\mu_1^{aV},\ldots,\mu_k^{aV})[\delta_x]\bigr)$ is uniformly 
bounded by  $C(M)a$ for some finite constant $C(M)$
which only depends on $M$,  the effective potential
\begin{equation}
\label{eq:Veff}
W^{\rm eff}_{k}(x):=W_k(x)-D_kF_0^a(\mu_1^{aV},\ldots,\mu_k^{aV},\tau_B)[\delta_x]
\end{equation}
is uniformly convex for $a<c_0/C(M)$ thanks to Hypothesis \ref{hypo}.
In addition $x\mapsto -\int\log|x-y|d\mu_k^{aV}(y)$ is convex for $x \in \R\setminus {\rm supp}(\mu_k^{aV})$.
This implies that the nonnegative  function $F_k-C_k$ is uniformly convex on $ \R\setminus {\rm supp}(\mu_k^{aV})$ and vanishes at the boundary of the support of $\mu_k$, hence 
$\mu_k^{aV}$ has  necessarily a connected support, that we denote by $[a_k^{aV},b_k^{aV}]$.

We now consider the measures $\mu^\varepsilon_k:=({\rm Id}+\varepsilon f_k)_\#\mu_k^{aV}$, where $f_k:\R\to\R$ is a smooth function. Then, since $I^a(\mu_1^\varepsilon,\ldots,\mu_d^\varepsilon)\geq I^a(\mu_1^{aV},\ldots,\mu_d^{aV})$, we deduce by comparing the terms linear in  $\varepsilon$ that
\begin{equation}
\label{eq:muk*}
\int (W_k^{\rm eff})'(x)f(x)\,d\mu_k^{aV}(x)= \iint \frac{f(x)-f(y)}{x-y}\,d\mu_k^{aV}(x)\,d\mu_k^{aV}(y)\qquad \forall\,k=1,\ldots,d,\,\,\forall\,f.
\end{equation}
In particular, choosing $f(x):=(z-x)^{-1}$ with $z \in \R\setminus [a_k^{aV},b_k^{aV}]$ 
we obtain that $G_k(z):=\int (z-x)^{-1}\,d\mu_k^{aV}(x)$ satisfies the equation
$$G_k(z)^2=(W_k^{\rm eff})' (z)G_k(z)+H_k(z), \qquad H_k(z):=\int \frac{(W_k^{\rm eff})' (x)-(W_k^{\rm eff})' (z)}{z-x}\,d\mu_k^{aV}(x).$$
Solving this quadratic equation so that $G(z) \to 0$ as $|z| \to \infty$ yields
$$G_k(z)=\frac{1}{2}\Bigl((W_k^{\rm eff})' (z)-\sqrt{(W_k^{\rm eff})' (z)^2+4H_k(z)}\Bigr)$$
from which it follows (by smoothness of $H_k$, see also \cite[Proof of Lemma 3.2]{BFG}) that
$$\frac{d\mu_k^{aV}(x)}{dx}=d_k(x)\sqrt{(x-a_k^{aV})(b_k^{aV}-x)}$$
where
 $$d_k(x)^2(x-a_k^{aV})(b_k^{aV}-x)=-(W_k^{\rm eff})' (x)^2-4H_k(x)=:g_k(x)\qquad
 \text{for $x\in [a_k^{aV},b_k^{aV}]$.}$$
 Note that $g_k$ is a smooth function. In the case where $a=0$, it is well known that the strict convexity of $W_k$ implies that
 $g_k$ has simple zeroes in $a_k^{aV},b_k^{aV}$, and that $d_k$ does not vanish in an open neighborhood of $[a_k^{aV},b_k^{aV}]$. On the other hand we also know (see e.g. Lemma \ref{expina}) that  the measures $\mu_k^{aV}$'s depends continuously on the parameter $a$ (the set of probability measures being 
equipped with the weak topology) as they are compactly supported measures with moments depending analytically on $a$. As a consequence, $g_k$ and $g_k'$ are smooth functions of $a$, uniformly in the variable $x$. This implies that, for $a$ small enough,
$g_k$ can only vanish in a small neighborhood of $a_k^{aV}$ and $b_k^{aV}$ where its derivative does not vanish. Hence $g_k$ can only have one simple zero in a small neighborhood of $a_k^{aV}$ (resp. $b_k^{aV}$), and $d_k$ cannot vanish in an open neighborhood of $[a_k^{aV},b_k^{aV}]$. Also, notice that $d_k$ is smooth as so are $W_k^{\rm eff}$ and $H_k$.
In addition, if one chooses $M>\max\{|a_k^0|,|b_k^0|\}$ for all $k=1,\ldots,d$, then by continuity we deduce that $[a_k^{aV},b_k^{aV}]\subset (-M,M)$ for any $a \in [-a_0,a_0]$.

We finally deduce uniqueness: Assume there are two minimizers $(\mu_1,\ldots,\mu_d)$ and $(\mu_1',\ldots,\mu_d')$.
By the previous considerations, both $\mu_i$ and $\mu_i'$ have smooth densities with respect to the Lebesgue measure on $\R$ and we can therefore
consider the unique monotone nondecreasing maps $T_i:\R\to\R$ such that that $\mu_i'=(T_i)_\#\mu_i$.
We then consider $$j^a(\tau):=J^a\bigl((\tau {\rm Id}+(1-\tau) T_1)_\#\mu_1,\ldots, (\tau{\rm Id}+(1-\tau)T_d)_\#\mu_d\bigr).$$
By concavity of the logarithm and  uniform convexity of $W_k-D_kF_0^a(\nu_1,\ldots,\nu_d,\tau_B)[\delta_x]$ (uniform with respect to $\nu_\ell\in \mathcal P([-M,M])$),
we conclude that $j^a$ is uniformly convex on $[0,1]$, which contradicts the minimality of $\mu_i$ and $\mu_i'$.
\end{proof}

We next show that, since the support of each $\mu_k^{aV}$ is strictly contained inside $[-M,M]$,
the eigenvalues will not touch $\R\setminus [-M,M]$ with large probability.
\begin{lem}\label{confinement}
Let Hypothesis \ref{hypo} hold.  There exists $a_0>0$ such that the following holds for $a\in [-a_0,a_0]$: if $[a_k^{aV},b_k^{aV}]$ denotes the support of $\mu_k^{aV}$ (see Lemma \ref{lem:support}), then
for any $\varepsilon>0$ there exists $c(\varepsilon)>0$ such that, for $N$ large enough,
$$ P^{N,aV}_\beta\bigl(\exists \,i\in \{1,\ldots,N\},\, \exists\, k\in\{1,\ldots,d\} : \lambda_i^k\in [a_k^{aV}-\varepsilon,b_k^{aV}+\varepsilon]^c\bigr)
\le e^{-c(\varepsilon)N}$$
\end{lem}
\begin{proof} By \cite[Lemma 3.1]{BGK} (see also \cite{BG1,BG2}) we can prove that for any closed sets $F_k$
$$\limsup_{N\ra\infty}\frac{1}{N}\log P^{N,aV}_\beta\left(\exists\, i,k:\lambda_i^k\in F_k \right)\le -\inf_{F_1\times \cdots \times F_d} {\mathcal I}$$
where ${\mathcal I}$ is the good rate function
$${\mathcal I}(x_1,\ldots,x_i):= {\mathcal J}(x_1,\ldots,x_k)-\inf _{y_1,\ldots,y_k\in [-M,M]^d} \mathcal J(y_1,\ldots,y_k)$$
with
$${\mathcal J}(x_1,\ldots,x_d):=
\sum_{k=1}^d\Bigl[ W^{\rm eff }_k(x_k) -\beta\int\log|x_k-y|\,d\mu_k^{aV}(y)\Bigr]$$
where $W^{\rm eff}$ is defined in \eqref{eq:Veff}.
As in the proof of Lemma \ref{lem:support} one sees that, for $|a|$ sufficiently small,
${\mathcal J}$ is uniformly convex outside the support of the measure,
whereas it is constant on each support. Hence it is strictly greater than its minimal value at positive distance of
this support, from which the conclusion follows.
\end{proof}

\section{Construction of approximate transport maps: proof of Theorem \ref{thm:transport}}\label{transport}
As explained in the introduction,
one of the drawbacks of the results in \cite{BFG} is that it only allows one to deal with bounded test functions.
To avoid this, we shall prove a multiplicative
closeness result (see \eqref{aptr}).\
\subsection{Simplification of the measures and strategy of the proof}
\label{sect:simplify}
We begin from the measure $P^{N,V}_\beta$ as in \eqref{poo}.
Because of Theorem \ref{mainexp}, it makes sense to introduce the
probability measures 
$$
dP^{N,aV}_{ t,\beta}(\lambda^1,\ldots,\lambda^d):=\frac{1}{\tilde Z^{N,aV}_{t,\beta}} e^{N^2 t F_0^a(L^{N}_1,\ldots,L^{N}_d,\tau_B^N)+ NtF_1^a(L^{N}_1,\ldots,L^{N}_d,\tau_B^N)+tF_2^a(L^{N}_1,\ldots,L^{N}_d,\tau_B^N)} 
\prod_{k=1}^d dR^{N,W_k}_{\beta,M} (\lambda^k)
$$
for $t\in [0,1]$, where $R^{N,W}_\beta$ is as in \eqref{eq:RNW}.
Then, it follows by \eqref{poo} and \eqref{po} that, for any nonnegative function $\chi:\R^N\to \R^+$,
$$
\frac{1+\int \chi \,dP^{N,aV}_{\beta}}{1+\int \chi\,dP^{N,aV}_{1,\beta}}=
\frac{\int (1+\chi) \,dP^{N,aV}_{\beta}}{\int (1+\chi)\,dP^{N,aV}_{1,\beta}}=1+O\biggl(\frac{1}{N}\biggr),
$$
therefore
\begin{equation}
\label{eq:close1}
\biggl|\log\biggl(1+\int \chi \,dP^{N,aV}_{\beta}\biggr)
- \log\biggl(1+\int \chi \,dP^{N,aV}_{1,\beta}\biggr)\biggr| \leq \frac{C}{N}.
\end{equation}
 Hereafter we do not stress the dependency in $\beta$, so $P^{N,aV}_{ t,\beta}=P^{N,aV}_{ t}$.

To remove the cutoff in $M$, let
$$
dQ^{N,aV}_{t}(\lambda^1,\ldots,\lambda^d):=\frac{1}{Z^{N,aV}_{t}} e^{\sum_{l=0}^2 N^{2-l} t F_l^a(\phi^M_\#L^{N}_1,\ldots,\phi^M_\#L^{N}_d,\tau_B^N)} \prod_{k=1}^d dR^{N,W_k}_{\beta,\infty} (\lambda^k),
$$
where 
\begin{equation}
\label{eq:ZN}
Z^{N,aV}_{t}:=\int e^{\sum_{l=0}^2 N^{2-l} t F_l^a(\phi^M_\#L^{N}_1,\ldots,\phi^M_\#L^{N}_d,\tau_B^N)} \prod_{k=1}^d dR^{N,W_k}_{\beta,\infty} (\lambda^k)
\end{equation}
and
$\phi^M:\R \to \R$ is a smooth function equal to $x$ on a neighborhood of the supports $[a_k^{aV},b_k^{aV}]$, vanishing outside of $[-2M,2M]$, 
and bounded by $2M$ everywhere.
Then Lemma \ref{confinement}  (as well as similar considerations for $Q^{N,aV}_{t}$) implies that, for some $\delta>0$,
\begin{equation}
\label{eq:close2}
\|Q^{N,aV}_1-P^{N,aV}_1\|_{TV}\le e^{-\delta N}\,.
\end{equation}
Notice that $Q^{N,aV}_{0}=Q^{N,0}_1=P^{N,0}_\beta$ so, if we can construct an approximate transport map
from $Q^{N,aV}_{0}$ to $Q^{N,aV}_1$ as in the statement of Theorem \ref{thm:transport}, by \eqref{eq:close1} and \eqref{eq:close2} the same map will be an approximate transport
from $P_\beta^{N,0}$ to $P_\beta^{N,aV}$.
Thus it suffices to prove Theorem \ref{thm:transport}
with $Q^{N,aV}_{0}$ and $Q^{N,aV}_1$ in place of $P^{N,0}_\beta$ and $P_\beta^{N,aV}$.

For this, we improve the strategy developed in \cite{BFG}: we construct a one parameter family of maps $T^N_{t}:\R^{dN}
\to \R^{dN}$ that approximately sends $Q^{N,aV}_{0}$ onto $Q^{N,aV}_{t}$
by solving
$$\partial_t T^N_{t}=\YY^N_t(T^N_{t}), \qquad T_0^N=\operatorname{Id},$$
where $\YY^N_{t}=\bigl((\YY^N_{t})_1^1,\ldots,(\YY^N_{t})_N^d\bigr):\R^{dN}\to \R^{dN}$ is constructed so that the following quantity is small in $L^q(Q^{N,aV}_{t})$ for any $q<\infty$:
\begin{equation}
\label{eq:RNt}
\begin{split}
{\mathcal R}^N_t(\YY^N)&:= c_t^N-\beta\sum_{k} \sum_{i<j}\frac{(\YY^N_t)_i^k-(\YY^N_t)_j^k}{\lambda_i^k-\lambda_j^k}
-\sum_{i,k} \partial_{\lambda_i^k}(\YY^N_t)_i^k \\
&-N^2 F_0^a(\phi^M_\#L^{N}_1,\ldots,\phi^M_\#L^{N}_d,\tau_B^N)-NF_1^a(\phi^M_\#L^{N}_1,\ldots,\phi^M_\#L^{N}_d,\tau_B^N)\\
&-F_2^a(\phi^M_\#L^{N}_1,\ldots,\phi^M_\#L^{N}_d,\tau_B^N)
+ \sum_{i,k} \partial_{\lambda_i^k} H_t(\hat\lambda)(\YY^N_t)_i^k,
\end{split}
\end{equation}
where $\hat\lambda:=(\lambda^1,\ldots,\lambda^d)=(\lambda_1^1,\ldots,\lambda_N^1,\ldots,\lambda_1^d,\ldots\lambda_N^d)$,
$c^N_t:= \partial_t \log Z^{N,aV}_{t}\,,$ $L^{N}_k:=\frac{1}N \sum_{i=1}^N \delta_{\lambda_i^k}$,
and
\begin{multline}
\label{eq:def H}
H_t(\hat\lambda):=N\sum_{i,k} W_k(\lambda_i^k)-tN^2 F_0^a(\phi^M_\#L^{N}_1,\ldots,\phi^M_\#L^{N}_d,\tau_B^N)\\
-tNF_1^a(\phi^M_\#L^{N}_1,\ldots,\phi^M_\#L^{N}_d,\tau_B^N)-tF_2^a(\phi^M_\#L^{N}_1,\ldots,\phi^M_\#L^{N}_d,\tau_B^N)\,.
\end{multline}
In  \cite{BFG} it is proved that  the flow of $\YY^N_t$ is an approximate transport map provided 
${\mathcal R}^N_t(\YY^N)$ is small: more precisely, if
$X_t^N$
solves the ODE
\begin{equation}
\label{eq:XtN}
 \dot X_t^N=\YY_t^N(X_t^N),\qquad X_0^N=\operatorname{Id},
\end{equation}
and we set
$T^N:=X_1^N$, then \cite[Lemma 2.2]{BFG} shows that
\begin{equation}
\label{aptr2}
\left|\int \chi\circ T^N\, d Q_0^{N,aV}-\int \chi\, d Q_1^{N,aV}\right|\le \|\chi\|_\infty\int_0^1 \|{\mathcal R}^N_t(\YY^N)\|_{L^1(Q^{N,aV}_{t})}\,dt
\end{equation}
for any bounded measurable function $\chi:\mathbb R^{dN}\to \mathbb R$.

Although this result is powerful enough if $\chi$ is a bounded test function, it becomes immediately useless if 
we would like to integrate a function that grows polynomially in $N$.
For this reason we prove here a new estimate that considerably improves \cite[Lemma 2.2]{BFG}.

\begin{lem}
\label{lem:key}
Assume that, for any $q<\infty$, there exists a constant $C_q$ such that
\begin{equation}\label{eq:toprove}
\|{\mathcal R}^N_t(\YY^N)\|_{L^q(Q^{N,aV}_{t})} \leq C_q \,\frac{(\log N)^3}{N} \qquad \forall \,t\in [0,1],
\end{equation}
define $X_t^N$ as in \eqref{eq:XtN},
and set $T^N:=X_1^N$.
Let $\chi:\R^N\to \R^+$ be a nonnegative measurable function satisfying $\|\chi\|_\infty \leq N^k$ for some $k \geq 0$.
Then, for any $\eta>0$ there exists a constant $C_{k,\eta}$, independent of $\chi$, such that
$$
\biggl| \log\biggl(1+\int \chi\, d Q_1^{N,aV}\biggr)
-  \log\biggl(1+\int \chi\circ T^N\, d Q_0^{N,aV}\biggr)\biggr| \leq C_{k,\eta} \,N^{\eta-1}.
$$
\end{lem}
Notice that this lemma proves the validity of \eqref{aptr} with $Q^{N,aV}_{0}$ and $Q^{N,aV}_1$ in place of $P^{N,0}_\beta$ and $P_\beta^{N,aV}$, provided we can show that \eqref{eq:toprove}
holds.

Here, we shall first prove Lemma \ref{lem:key} and then
we show the validity of \eqref{eq:toprove}.
More precisely, in Section \ref{sect:proof lemma key} we prove Lemma \ref{lem:key}.
Then in Sections \ref{sect:constr}-\ref{sect:rest}
we show that 
\begin{equation}
\label{eq:toprove1}
|\mathcal R^N_t(\YY^N)|\leq C\,\frac{(\log N)^3}N\quad\text{
on a set $G_t\subset \R^N$ satisfying $Q_t^{N,aV}(G_t)\geq 1-N^{-cN}$.}
\end{equation}
Since $\RR_t^N(\YY^N)$ is trivially bounded by $CN^2$ everywhere (being the sum of $O(N^2)$ bounded terms, see \eqref{eq:RNt}),
\eqref{eq:toprove1} implies that
$$
\|{\mathcal R}^N_t(\YY^N)\|_{L^q(Q^{N,aV}_{t})} 
\leq C\,\frac{(\log N)^3}{N} +C\,N^2 \Bigl(Q^{N,aV}_{t}(\R^N\setminus G_t)\Bigr)^{1/q} \leq  C\,\frac{(\log N)^3}{N},
$$
proving \eqref{eq:toprove}.

Finally, in Section \ref{sect:flow} we show that $T^N=X_1^N$ satisfies all the properties stated in Theorem \ref{thm:transport}.

\subsection{Proof of Lemma \ref{lem:key}}
\label{sect:proof lemma key}
Let $\rho_t$ denote the density of $Q_t^{N,aV}$ with respect to the Lebesgue measure $\mathcal L$.
Then, by a direct computation one can check that $\rho_t$, $\YY^N$, and
$\mathcal R_t^N=\RR_t^N(\YY^N)$ are related by
the following formula:
\begin{equation}
\label{eq:rho cont}
\partial_t \rho_t+{\rm div}(\YY_t^N\rho_t)=\RR_t^N\rho_t.
\end{equation}
Now, given a smooth function $\chi:\R^N\to\R^+$ satisfying $\|\chi\|_\infty \leq N^k$
we define 
\begin{equation}
\label{eq:chit}
\chi_t:=\chi\circ X_1^N \circ (X_t^N)^{-1}\qquad \forall\,t \in [0,1].
\end{equation}
Note that with this definition $\chi_1=\chi$.
Also, since $\chi_t\circ X_t^N$ is constant in time, differentiating with respect to $t$ we deduce that
$$
0=\frac{d}{dt}\bigl(\chi_t\circ X_t^N\bigr)=\Bigl(\partial_t \chi_t+\YY_t^N\cdot \nabla \chi_t\Bigr)\circ X_t^N,
$$
hence $\chi_t$
solves the transport equation
\begin{equation}
\label{eq:chi transp}
\partial_t \chi_t+\YY_t^N\cdot \nabla \chi_t=0,\qquad \chi_1=\chi.
\end{equation}
Combining \eqref{eq:rho cont} and \eqref{eq:chi transp}, we compute
\begin{align*}
\frac{d}{dt}\int \chi_{t}\, \rho_t \,d\mathcal L&=\int \partial_t \chi_t\,\rho_t \,d\mathcal L + \int \chi_t\,\partial_t\rho_t \,d\mathcal L\\
&=- \int \YY_t^N\cdot \nabla \chi_t\,\rho_t \,d\mathcal L -\int \chi_{t} \,{\rm div}(\YY_t^N\rho_t)\,d\mathcal L+ \int \chi_t\, \RR_t^N\,\rho_t \,d\mathcal L\\
&=\int \chi_t\, \RR_t^N\,\rho_t \,d\mathcal L.
\end{align*}
We want to control the last term.
To this aim we notice that, since $\|\chi\|_\infty \leq N^k$, it follows immediately from \eqref{eq:chit}
that $\|\chi_t\|_\infty \leq N^k$ for any $t \in [0,1]$.
Hence, using H\"older inequality and \eqref{eq:toprove}, for any $p>1$ we can bound
\begin{align*}
\biggl|\int \chi_t\, \RR_t^N\,\rho_t \,d\mathcal L\biggr| &\leq  \|\chi_t\|_{L^p(Q^{N,aV}_{t})} \|\RR_t^N\|_{L^q(Q^{N,aV}_{t})}
\leq \|\chi_t\|_{\infty}^{\frac{p-1}{p}}\|\chi_t\|_{L^1(Q^{N,aV}_{t})}^{1/p}  \|\RR_t^N\|_{L^q(Q^{N,aV}_{t})}\\
&\leq  N^{\frac{k(p-1)}{p}}\|\chi_t\|_{L^1(Q^{N,aV}_{t})}^{1/p}  \|\RR_t^N\|_{L^q(Q^{N,aV}_{t})}
\leq C_q \frac{N^{\frac{k(p-1)}{p}}(\log N)^3}{N}\|\chi_t\|_{L^1(Q^{N,aV}_{t})}^{1/p},
\end{align*}
where $q:=\frac{p}{p-1}$.
Hence, given $\eta>0$, we can choose $p:=1+\frac{\eta}{2k}$
to obtain
$$
\biggl|\int \chi_t\, \RR_t^N\,\rho_t \,d\mathcal L\biggr| \leq C_q \,N^{\eta-1}\|\chi_t\|_{L^1(Q^{N,aV}_{t})}^{1/p}
\leq C \,N^{\eta-1}\Bigl(1+\|\chi_t\|_{L^1(Q^{N,aV}_{t})}\Bigr),
$$
where $C$ depends only on $C_q$, $k$, and $\eta$.
Therefore, setting
$$
Z(t):=\int \chi_{t}\, \rho_t \,d\mathcal L=\|\chi_t\|_{L^1(Q^{N,aV}_{t})}
$$
(recall that $\chi_t\ge 0$), we proved that 
$$
|\dot Z(t)| \leq C \,N^{\eta-1}\bigl(1+Z(t)\bigr),
$$
which implies that
$$
\bigl|\log\bigl(1+Z(1)\bigr)-\log\bigl(1+Z(0)\bigr)\bigr|\leq C \,N^{\eta-1}.
$$
Recalling that $T^N=X_1^N$,
this proves the desired result when $\chi$ is smooth.
By approximation the result extends to all measurable functions $\chi:\R^N\to \R^+$ satisfying $\|\chi\|_\infty \leq N^k$, concluding the proof.
\qed

 \subsection{Construction of approximate transport maps}\label{sect:constr}
Define
$$
M^{N}_k:=\sum_{i=1}^N \delta_{\lambda_i^k} - N\mu_{k,t}^{*},
$$
where $\mu_{k,t}^{*}:=\mu_{k,t}^{aV}$ are the limiting measures for $L^N_k$ under $Q^{N,aV}_t$; their existence
and properties are derived exactly as in the case $t=1$, see Section \ref{eqmeas}. In analogy with
 \cite[Section 2.3]{BFG} we make the following ansatz: we
look for a vector field $\YY_{t}^N$ of the form
\begin{equation}
\label{eq:psi}
(\YY_t^N)_i^k(\hat\lambda)
=\yy_{k,t}^0 (\lambda_i^k)+ \frac1N \yy_{k,t}^1 (\lambda_i^k)
+
\frac{1}{N}\sum_{\ell=1}^d\zzeta_{k\ell,t}(\lambda_i^k,M^{N}_\ell)
\end{equation}
where $\yy_{k,t}^0:\R\to\R$, $\yy_{k,t}^1:\R\to\R$, 
$\zz_{k\ell,t}=\zz_{\ell k,t}:\R^2\to\R$, and $
\zzeta_{k\ell,t}(x,M^{N}_\ell):= \int \zz_{k\ell,t}(x,y)\,dM^{N}_\ell(y)$.
With this particular choice of $\YY_t^N$ we see that
\begin{multline*}
\sum_{i}\partial_{\lambda_i^k}(\YY_t^N)_i^k(\hat\lambda)
=N\int (\yy_{k,t}^0)' (x)\,dL^{N}_k(x)+ \int (\yy_{k,t}^1)' (x)\,dL^{N}_k(x)
\\
+\sum_\ell \int \partial_1 \zzeta_{k\ell,t}(x,M^{N}_\ell)\,dL^{N}_k(x)
+ \int \partial_{2} \zz_{kk,t}(x,x)\,dL^{N}_k(x).
\end{multline*}
We now expand $\{F_l^a\}_{l=0,1,2}$ around the stationary measures $\mu_{k,t}^{*}$ (recall that $F_l^a$ are smooth by Lemma \ref{lem:smooth},
and that $M_N$ has mass bounded by $2N$) and use that $\phi^M_\#\mu_{k,t}^{*}=\mu_{k,t}^{*}$ to get
\begin{align*}
F_l^a(\phi^M_\#L^{N}_1&,\ldots,\phi^M_\#L^{N}_d,\tau_B^N)
=F_l^a(\mu_{1,t}^{*},\ldots,\mu_{d,t}^{*},\tau_B^N) + \frac{1}N
\sum_{k} D_kF_l^a(\mu_{1,t}^{*},\ldots,\mu_{d,t}^{*},\tau_B^N)[\phi^M_\#M^{N}_k]
\\
&+ \frac{1}{N^2}
\sum_{k\ell} D_{k\ell}^2F_l^a(\mu_{1,t}^{*},\ldots,\mu_{d,t}^{*},\tau_B^N)[\phi^M_\#M^{N}_k,\phi^M_\#M^{N}_\ell] \\
&+ \frac{1}{N^3}
\sum_{k\ell m} D_{k\ell m}^3F_l^a(\mu_{1,t}^{*},\ldots,\mu_{d,t}^{*},\tau_B^N)[\phi^M_\#M^{N}_k,\phi^M_\#M^{N}_\ell,\phi^M_\#M_m^N] + O\biggl(\frac{|\phi^M_\#M^N|^4}{N^4}\biggr)
\end{align*}
where $O\left(\frac{|\phi^M_\#M^N|^p}{N^k}\right):=O\left(N^{-k}\Bigl\| \phi^M_\#M^N\Bigr\|^p_{M\zeta}\right)$, see Lemma \ref{lem:smooth}. 

We now use assumption \eqref{expB} and the smoothness of the functions $F^a_l$ (see Lemma \ref{lem:smooth} again) to expand $D_kF_l^a$, $D_{k\ell}^2F_l^a$, and
$D_{k\ell m}^3F_l^a$ with respect to $\tau$. To simplify notation, we define the following  functions: 
\begin{align*}
f_{k,l}(x)&:=D_kF_l^a(\mu_{1,t}^{*},\ldots,\mu_{d,t}^{*},\tau_B^0)[\delta_{\phi^M(x)}],\\
 f_{k\tau^1,l}(x)&:=D_{k,\tau}^2F_l^a(\mu_{1,t}^{*},\ldots,\mu_{d,t}^{*},\tau_B^0)[\delta_{\phi^M(x)},\tau_B^1],\\
 f_{k\tau^2,l}(x)&:=D_{k,\tau}^2F_l^a(\mu_{1,t}^{*},\ldots,\mu_{d,t}^{*},\tau_B^0)[\delta_{\phi^M(x)},\tau_B^2]\\
&\quad+\textstyle{\frac12}D_{k,\tau\tau}^3F_l^a(\mu_{1,t}^{*},\ldots,\mu_{d,t}^{*},\tau_B^0)[\delta_{\phi^M(x)},\tau_B^1,\tau_B^1],\\
 f_{k\ell,l}(x,y)&:=D_{k\ell}^2F_l^a(\mu_{1,t}^{*},\ldots,\mu_{d,t}^{*},\tau_B^0)[\delta_{\phi^M(x)},\delta_{\phi^M(y)}]\\
f_{k\ell \tau^1,l}(x,y)&:=
D_{k\ell,\tau}^3F_l^a(\mu_{1,t}^{*},\ldots,\mu_{d,t}^{*},\tau_B^0)[\delta_{\phi^M(x)},\delta_{\phi^M(y)},\tau_B^1],\\
 f_{k\ell m,l}(x,y)&:=D_{k\ell m}^3F_l^a(\mu_{1,t}^{*},\ldots,\mu_{d,t}^{*},\tau_B^0)[\delta_{\phi^M(x)},\delta_{\phi^M(y)},\delta_{\phi^M(z)}].
\end{align*}
We can assume without loss of generality that these functions are symmetric with respect to their arguments. 
Then we get the following formulas: 
\begin{align*}
&F_l^a(\phi^M_\#L^{N}_1,\ldots,\phi^M_\#L^{N}_d,\tau_B^N)=F_l^a(\mu_{1,t}^{*},\ldots,\mu_{d,t}^{*},\tau_B^0)+\frac1N \sum_k\int f_{k,l}(x)\,dM^{N}_k(x)\\
&+\frac1{N^2} \sum_k\int f_{k\tau^1,l}(x)\,dM^{N}_k(x)
+\frac{1}{N^2}\sum_{k\ell} \iint f_{k\ell,l}(x,y)\,dM^{N}_k(x)\,dM^{N}_\ell(y)\\
&+\frac1{N^3} \sum_k\int f_{k\tau^2,l}(x)\,dM^{N}_k(x)
+\frac{1}{N^3}\sum_{k\ell} \iint f_{k\ell \tau^1,l}(x,y)\,dM^{N}_k(x)\,dM^{N}_\ell(y)\\
&+
\frac{1}{N^3}\sum_{k\ell m} \iiint f_{k\ell m,l}(x,y,z)\,dM^{N}_k(x)\,dM^{N}_\ell(y)\,dM_{m}^N(z)
+ O\biggl(\frac{|\phi^M_\#M^N|^4}{N^4}\biggr),
\end{align*}
and
\begin{align*}
\partial_{\lambda_i^k}  F_l^a(\phi^M_\#L^{N}_1,\ldots,\phi^M_\#L^{N}_d,\tau_B^N)
&=\frac1N f_{k,l}'(\lambda_i^k)+\frac1{N^2} f_{k \tau^1,\alpha}'(\lambda_i^k)+
\frac{2}{N^2}\sum_{\ell} \int \partial_1f_{k\ell,l}(\lambda_i^k,y)\,dM^{N}_\ell(y)
\\
&+\frac1{N^3} f_{k \tau^2,\alpha}'(\lambda_i^k)
+\frac{2}{N^3}\sum_{\ell} \int \partial_1f_{k\ell\tau^1,\alpha}(\lambda_i^k,y)\,dM^{N}_\ell(y)\\
&+\frac{3}{N^3}\sum_{\ell m} \iint \partial_1f_{k\ell m,l}(\lambda_i^k,y,z)\,dM^{N}_\ell(y)\,dM_{m}^N(z)
+ O\biggl(\frac{|\phi^M_\#M^N|^3}{N^4}\biggr).
\end{align*}
This gives, for $H$ defined in \eqref{eq:def H}, 
\begin{align*}
\partial_{\lambda_i^k} H_t(\hat\lambda)&=N W_k'(\lambda_i^k)
-tN f_{k,0}'(\lambda_i^k)-t \Bigl[f_{k\tau^1,0}'(\lambda_i^k)  - f_{k,1}'(\lambda_i^k)\Bigr]
-2t\sum_{\ell} \int \partial_1f_{k\ell,0}(\lambda_i^k,y)\,dM^{N}_\ell(y) \\
&-\frac{t}N \Bigl[f_{k\tau^2,0}'(\lambda_i^k)+ f_{k\tau^1,1}'(\lambda_i^k)+f_{k,2}'(\lambda_i^k)\Bigr]-\frac{2t}{N}\sum_{\ell} \int \Bigl[\partial_1f_{k\ell \tau^1,0}(\lambda_i^k,y)+\partial_1f_{k\ell,1}(\lambda_i^k,y)\Bigr]\,dM^{N}_\ell(y)\\
&-\frac{3t}{N}\sum_{\ell m} \iint \partial_1f_{k\ell m,0}(\lambda_i^k,y,z)\,dM^{N}_\ell(y)\,dM_{m}^N(z)
+ O\biggl(\frac{|\phi^M_\#M^N|^2}{N^2}\biggr).
\end{align*}
Also, with this notation, the analogue of \eqref{eq:Veff} for $t \in [0,1]$ becomes
\begin{equation}
\label{eq:Veff2}
W^{\rm eff}_{k,t}(x):=W_k(x)-t f_{k,0}(x).
\end{equation}
Hence, with all this at hand, we can estimate the term ${\mathcal R}^N_t(\YY^N)$ defined in \eqref{eq:RNt}:
using the convention that when we integrate a function of the form $\frac{\psi(x)-\psi(y)}{x-y}$
with respect to $L_k^N\otimes L_k^N$ the diagonal terms give $\psi'(x)$, we get 
\begin{align*}
{\mathcal R}^N_t(\YY^N)&= c_t^N-\frac{\beta N^2}{2}
\sum_{k} \iint \frac{\yy_{k,t}^0(x) - \yy_{k,t}^0(y)}{x-y}\,dL^{N}_k(x)\,dL^{N}_k(y)
-N\biggl(1-\frac\beta2\biggr)\sum_{k}  \int (\yy_{k,t}^0)'
\,dL^{N}_k\\
&-\frac{\beta N}{2}
\sum_{k} \iint \frac{\yy_{k,t}^1(x) - \yy_{k,t}^1(y)}{x-y}\,dL^{N}_k(x)\,dL^{N}_k(y)
-\biggl(1-\frac\beta2\biggr)\sum_{k}  \int (\yy_{k,t}^1)'\,dL^{N}_k\\
&-\frac{\beta N}{2}
\sum_{k\ell} \iint \frac{\zzeta_{k\ell,t}(x,M^{N}_\ell) -\zzeta_{k\ell,t}(y,M^{N}_\ell)}{x-y}\,dL^{N}_k(x)\,dL^{N}_k(y)\\
&-\biggl(1-\frac\beta2\biggr)\sum_{k\ell}  \int \partial_1\zzeta_{k\ell,t}(x,M^{N}_\ell)\,dL^{N}_k
-\sum_k \int \partial_{2}\zz_{kk,t}(x,x)\,dL^{N}_k(x)\\
&-N^2 
F_0^a(\mu_{1,t}^*,\ldots,\mu_{d,t}^*,\tau_B^0)-N \sum_k\int f_{k,0}(x)\,dM^{N}_k(x)
-\sum_{k\ell} \iint f_{k\ell,0}(x,y)\,dM^{N}_k(x)\,dM^{N}_\ell(y)\\
&-N
F_1^a(\mu_{1,t}^*,\ldots,\mu_{d,t}^*,\tau_B^0)-\sum_k\int f_{k,1}(x)\,dM^{N}_k(x)
-F_2^a(\mu_{1,t}^*,\ldots,\mu_{d,t}^*,\tau_B^0)\\
&+N^2\sum_k \int (W^{\rm eff}_{k,t})'(x)\yy_{k,t}^0(x)\,dL^{N}_k(x)
+N\sum_k \int (W^{\rm eff}_{k,t})'(x)\yy_{k,t}^1\,dL^{N}_k(x)\\
&+N\sum_{k\ell} \int (W^{\rm eff}_{k,t})'(x)\,\zzeta_{k\ell,t}(x,M^{N}_\ell)\,dL^{N}_k(x)
-tN\sum_k \int \Bigl[f_{k\tau^1,0}'  - f_{k,1}'\Bigr](x)\yy_{k,t}^0(x)\,dL^{N}_k(x)\\
&-t\sum_k \int \Bigl[f_{k\tau^1,0}'  - f_{k,1}'\Bigr](x)\yy_{k,t}^1(x)\,dL^{N}_k(x) - t\sum_{k\ell}\int \Bigl[f_{k\tau^1,0}'  - f_{k,1}'\Bigr](x)\,\zzeta_{k\ell,t}(x,M^{N}_\ell)\,dL^{N}_k(x)\\
&-2tN\sum_{k\ell} \iint \partial_1f_{k\ell,0}(x,y)\yy_{k,t}^0(x)\,dM^{N}_\ell(y)\,dL^{N}_k(x)\\
&-2t\sum_{k\ell} \iint \partial_1f_{k\ell,0}(x,y)\yy_{k,t}^1(x)\,dM^{N}_\ell(y)\,dL^{N}_k(x)\\
& - 2t\sum_{k\ell  m}\iint \partial_1f_{k\ell,0}(x,y)\,\zzeta_{km,t}(x,M_m^N)\,dM^{N}_\ell(y)\,dL^{N}_k(x)\\
&-3t\sum_{k\ell m} \iiint \partial_1f_{k\ell m,0}(x,y,z)\,\yy_{k,t}^0(x)\,dM^{N}_\ell(y)\,dM_{m}^N(z)\,dL^{N}_k(x)\\
&-2t\sum_{k\ell} \iint \Bigl[\partial_1f_{k\ell \tau^1,0}+\partial_1f_{k\ell,1}\Bigr](x,y)\,\yy_{k,t}^0(x)\, dM^{N}_\ell(y)\,dL^{N}_k(x)\\
&-t \sum_k \int \Bigl[f_{k\tau^2,0}'+ f_{k\tau^1,1}'+f_{k,2}'\Bigr](x)\yy_{k,t}^0(x)\,dL^{N}_k(x)+ O\biggl(\frac{|\phi^M_\#M^N|^3}{N}\biggr).
\end{align*}
Recalling \eqref{eq:muk*} we observe that, for any function $f$,
\begin{multline}
\label{eq:identity}
N^2\int (W_{k,t}^{\rm eff})'f\,dL^{N}_k -\frac{\beta N^2}2 \iint \frac{f(x)-f(y)}{x-y}\,dL^{N}_k(x)\,dL^{N}_k(y)\\
=N\int \Xi_kf\,dM^{N}_k -\frac{\beta}2 \iint \frac{f(x)-f(y)}{x-y}\,dM^{N}_k(x)\,dM^{N}_k(y),
\end{multline}
where
\begin{equation}
\label{eq:Xi}
\Xi_kf(x):=-\beta \int \frac{f(x)-f(y)}{x-y}\,d\mu_{k,t}^*(y) + (W_{k,t}^{\rm eff})'(x)f(x).
\end{equation}
Also, observe that up to now the term $O\Big(\frac{|\phi^M_\#M^N|^3}{N}\Big)$ does not depend on the smoothness of the functions $\yy_{k,t}^0,\yy_{k,t}^1, \zz_{k\ell,t}$. However, in order to be able later to quantify the degree of smoothness required on the potentials $W_k$, we introduce a further notation:
we will denote by
$O\Big(\frac{|\phi^M_\#M^N|^3}{N};g_1,g_2,\ldots,g_p\Big)$ 
a quantity bounded by 
\begin{equation}
\label{eq:rest precise}
\sum_{m=1}^p R[g_m]+\frac{C}{N} \|\phi^M_\# M^N\|_{M\zeta}\,,
\end{equation}
where the functions $g_m$ map $\mathbb R^{\ell_m}$ into $\mathbb R$ for $\ell_m\in\{1,2\}$, and 
\begin{multline*}
R[g_m]:=\sum_{r_1,r_2,r_3=1}^d  \frac{1}{N}\, \int_0^1 d\alpha \biggl| \iiint g_m(\alpha z_1+(1-\alpha) z_2,z_3) \,dM_{r_1}^N(z_1) \, dM_{r_2}^N(z_2) \,dM_{r_{3}}^N(z_{3})\biggr|
\\
+\sum_{r_1,r_2=1}^d  \frac{1}{N}\,  \biggl| \iint g_m(z_1,z_2)
\,dM_{r_1}^N(z_1)\, dM_{r_2}^N(z_2)\biggr|
+\sum_{r_1=1}^d  \frac{1}{N}\,  \biggl| \int g_m(z_1,z_1)
\,dM_{r_1}^N(z_1)\biggr|
\end{multline*}
if $\ell_m=2$, while
\begin{multline*}
R[g_m]:=\sum_{r_1,r_2=1}^d  \frac{1}{N}\, \int_0^1 d\alpha \biggl| \iint g_m(\alpha z_1+(1-\alpha) z_2) \,dM_{r_1}^N(z_1) \, dM_{r_2}^N(z_2) \biggr|
\\
+\sum_{r_1=1}^d  \frac{1}{N}\,  \biggl| \int g_m(z_1)
\,dM_{r_1}^N(z_1)\biggr|
\end{multline*}
if $\ell_m=1$.
For instance, 
writing
$$
\frac{\zz_{k\ell,t}(x,z)-\zz_{k\ell,t}(y,z)}{x-y}=\int_0^1\partial_1\zz_{k\ell,t}(\alpha x+(1-\alpha)y,z)\,d\alpha
$$
and recalling the definition of $\zzeta_{k\ell,t}$,
we see that
$$
\frac{1}{N}\iint \frac{\zzeta_{k\ell,t}(x,M^{N}_\ell) -\zzeta_{k\ell,t}(y,M^{N}_\ell)}{x-y}\,dM^{N}_k(x)\,dM^{N}_k(y)
=O\biggl(\frac{|\phi^M_\#M^N|^3}{N};\partial_1\zz_{k\ell,t}\biggr).
$$
Thus, applying \eqref{eq:identity} to $f=\yy_{k,t}^0,\yy_{k,t}^1,\zzeta_{k\ell,t}(\cdot,M^{N}_\ell)$,
and using that
$L^{N}_k=\mu_{k,t}^*+\frac{M^{N}_k}{N}$ (recall that $\zz_{k\ell,t}=\zz_{\ell k,t}$ for all $k,\ell$), we get
\begin{align*}
\mathcal R_t^N(\YY^N)&=N\sum_k \int \biggl[\Xi_k\yy_{k,t}^0 
-2t\Bigl(\sum_\ell\int \yy_{\ell,t}^0 (y)\partial_1f_{k\ell,0}(y,\cdot)\,d\mu_{\ell,t}^*(y)\Bigr)
-f_{k,0}\biggr]\,dM^{N}_k \\
&+\sum_k\int \bigg(\Xi_k\yy_{k,t}^1 -2t \Bigl(\sum_{\ell} \int\yy_{\ell,t}^1(y) \partial_1f_{k\ell,0}(y,\cdot)\,d\mu_{\ell,t}^*(y)\Bigr)\\
&\qquad\qquad\qquad\qquad-f_{k,1}-t\Bigl[f_{k\tau^1,0}'  - f_{k,1}'\Bigr]\yy_{k,t}^0-\biggl(\frac{\beta}2-1\bigg)(\yy_{k,t}^0)'\\
&\qquad\qquad\qquad\qquad -\biggl(1-\frac\beta2\biggr)\sum_{\ell}  \int \partial_1\zz_{k\ell,t}(y,\cdot)\,d\mu_{\ell,t}^*(y)\\
&\qquad\qquad\qquad\qquad -
t\sum_\ell\int \bigl[f'_{\ell\tau^1,0}- f_{\ell,1}'\bigr](y) \,\zz_{k\ell,t}(y,\cdot)\,d\mu_{\ell,t}^*(y)\\
&\qquad\qquad\qquad\qquad -2t \sum_{\ell} \int\yy_{\ell,t}^0(y)\Bigl[\partial_1f_{k\ell \tau^1,0}+\partial_1f_{k\ell,1}\Bigr](y,\cdot)\,d\mu_{\ell,t}^*(y) \biggr)\,dM^{N}_k\\
&+\sum_{k\ell}\iint \bigg(\Xi_k[\zz_{k\ell,t}(\cdot,y)](x)
-2t\sum_{m}\int \zz_{km,t}(z,y)\,\partial_1f_{m\ell,0}(z,x)\,d\mu_{m,t}^*(z)\\
&\qquad\qquad\qquad\qquad- f_{k\ell,0}(x,y)-2t\,\partial_1f_{k\ell,0}(x,y) \yy_{k,t}^0(x)-\frac{\beta}{2}1_{k=\ell}  \frac{\yy_{k,t}^0(x)-\yy_{k,t}^0(y)}{x-y}\\
&\qquad\qquad\qquad\qquad -3 t \sum_m\int \yy_{m,t}^0(z)\,\partial_1f_{k\ell m,0}(x,y,z)\,d\mu_{m,t}^*(z)\bigg)\,dM^{N}_k(x)\,dM^{N}_\ell(y)\\
&+ C_t^N+O\biggl(\frac{|\phi^M_\#M^N|^3}{N};(\yy_{k,t}^1)',\partial_1\zz_{k\ell,t}, \partial_{2}\zz_{kk,t}\biggr)
\end{align*}
where $C^N_t$ is a constant. 
Let us consider the operator $\XXi_t$ defined on $d$-uple of functions by 
$$
\XXi_t(\Psi_1,\ldots,\Psi_d):=\Bigl(\XXi_t(\Psi_1,\ldots,\Psi_d)_1,\ldots,\XXi_t(\Psi_1,\ldots,\Psi_d)_d\Bigr),
$$
where
\begin{equation}
\label{eq:Xi t}
\XXi_t(\Psi_1,\ldots,\Psi_d)_k:=\Xi_k\Psi_k - 2t\sum_{\ell=1}^d\int \Psi_\ell(y)\,\partial_1f_{k\ell,0}(y,\cdot)\,d\mu_{\ell,t}^*(y)\qquad \forall\,k=1,\ldots,d.
\end{equation}
Then, 
for $\mathcal R_t^N(\YY^N)$ to be small we want to impose
\begin{equation}
\label{eq:main Xi}
\begin{array}{l}
\XXi_t\bigl(\yy_{1,t}^0,\ldots,\yy_{d,t}^0\bigr)_k=\bigl(g_1^0,\ldots,g_d^0\bigr),
\\
\XXi_t\bigl(\zz_{1\ell,t}(\cdot,y),\ldots,\zz_{d\ell,t}(\cdot,y)\bigr)_k=
\bigl(g_{1\ell}^2(\cdot,y),\ldots,g_{d\ell}^2(\cdot,y)\bigr)\qquad \forall\,\ell=1,\ldots,d,\,\forall\,y,
\\
\XXi_t\bigl(\yy_{1,t}^1,\ldots,\yy_{d,t}^1\bigr)_k=\bigl(g_1^1,\ldots,g_d^1\bigr),
\end{array}
\end{equation}
where 
$$
g_k^0(x):=f_{k,0}(x)+c_k,
$$
\begin{align*}
g_{k\ell}^2(x,y)&:= f_{k\ell,0}(x,y)+2t\,\partial_1f_{k\ell,0}(x,y) \yy_{k,t}^0(x)\\
&\quad+3 t \sum_m\int \yy_{m,t}^0(z)\,\partial_1f_{k\ell m,0}(x,y,z)\,d\mu_{m,t}^*(z)+c_{k\ell}(y)\qquad \text{if $k \neq \ell$},\\
g_{kk}^2(x,y)&:= f_{kk,0}(x,y)+2t\,\partial_1f_{kk,0}(x,y) \yy_{k,t}^0(x) -\frac{\beta}2  \frac{\yy_{k,t}^0(x)-\yy_{k,t}^0(y)}{x-y} \\
&\quad+3 t \sum_m\int \yy_{m,t}^0(z)\,\partial_1f_{kk m,0}(x,y,z)\,d\mu_{m,t}^*(z)+c_{kk}(y),\\
g_{k}^1(x)&:=f_{k,1}(x)+t\Bigl[f_{k\tau^1,0}'(x)  - f_{k,1}'(x)\Bigr]\yy_{k,t}^0(x)+\biggl(\frac{\beta}2-1\bigg)(\yy_{k,t}^0)'(x)\\
&\quad+\biggl(1-\frac\beta2\biggr)\sum_{\ell}  \int \partial_1\zz_{k\ell,t}(y,x)\,d\mu_{\ell,t}^*(y)
+
\sum_\ell\int f_{\ell,1}'(y) \,\zz_{k\ell,t}(y,x)\,d\mu_{\ell,t}^*(y)\\
&\quad+2t \sum_{\ell} \int\yy_{\ell,t}^0(y)
\Bigl[\partial_1f_{k\ell \tau^1,0}+\partial_1f_{k\ell,1}\Bigr](y,\cdot)\,d\mu_{\ell,t}^*(y)+c_{k}' ,
\end{align*}
where $c_k,c_k'$ are constants to be fixed later, and $c_{k\ell}(y)$ is a family of functions depending only on $y$ also to be fixed.

Indeed, noticing that $\int dM_k^N=0$ for all $k$, we see that all constants integrate to zero against $M^N$,
and we conclude that the following holds:
\begin{lem}
\label{lem:small Rt}
Let $\Xi_t$ be defined as in \eqref{eq:Xi t}, with $\{\Xi_k\}_{k=1}^d$ as in \eqref{eq:Xi}.
Also, recall the notation \eqref{eq:rest precise}.
Assume that we can find functions $\yy_{k,t}^0,\yy_{k,t}^1,\zz_{k\ell,t}$ 
solving \eqref{eq:main Xi}. Then 
$$
\mathcal R_t^N(\YY^N)=C_t^N+O\biggl(\frac{|\phi^M_\#M^N|^3}{N};(\yy_{k,t}^1)',\partial_1\zz_{k\ell,t}, \partial_{2}\zz_{kk,t}\biggr),
$$
where $C^N_t$ is a constant.
\end{lem}

\subsection{Invertibility properties of $\XXi_t$}
Lemma \ref{lem:small Rt} suggests that, to construct an approximate map, we need to solve an equation of the form
$$
\XXi_t (\Psi_1,\ldots,\Psi_d)=(g_1,\ldots,g_d).
$$
We remind that, in our setting, the functions $\partial_1f_{k\ell,0}(\cdot,y)$ 
are smooth and their $C^s$ norm is of size $O(|a|)$ for any $s>0$, where $a$ is a small number.
Also, note that the operators $\Xi_k$ defined in \eqref{eq:Xi} are continuous with respect to the $C^1$ topology.
This will allow us to show invertibility of $\XXi_t$ using Lemma \ref{lem:Xi} below and a fixed point argument

Before stating that result in our setting we recall that, given a function $f:\R\to \R$, the norm $C^s$ is defined as
$$
\|f\|_{C^{s}(\R)}:=\sum_{j=0}^s\|f^{(j)}\|_{L^\infty(\R)},
$$
where $f^{(j)}$ denotes the $j$-th derivative of $f$.
The next result is contained in  \cite[Lemma 3.2]{BFG}.

\begin{lem} \label{lem:Xi}
Given $V:\R\to \R$ be a function of class $C^\sigma$ with $\sigma  \geq 4$, assume that $\mu_{V}$ has support given by $[a,b]$
and that
\begin{equation}
\label{eq:non deg V}
\frac{d\mu_V}{dx}(x)=S(x)\sqrt{(a-x)(x-b)}\quad \text{with $S(x)\geq \bar c>0$ a.e. on $[a,b]$. }
\end{equation}
Define the operator
$$
\Xi \Psi(x):=-\beta\int\frac{\Psi(x)-\Psi(y)}{x-y} d\mu_V(x) +V'(x)\Psi(x),
$$
and fix an integer $3 \leq s \leq \sigma-1$.
Then, for 
any function $g:\mathbb R\ra\mathbb R$ of class $C^s$,
there exists a unique constant $c_g$ such that
the equation
$$\Xi \Psi(x)=g(x)+c_g$$
has a unique solution $\Psi:\R\to \R$ of class $C^{s-2}$, also denoted by $\Xi^{-1}g$,
which satisfies the estimate
\begin{equation}
\label{eq:bound norms}
\|\Psi\|_{C^{s-2}(\R)}\le 
 \hat C_s\|g\|_{C^{s}(\R)}.
\end{equation}

Moreover $\Psi$ (and its derivatives) behaves like $(g(x)+c_g)/V'(x)$ (and its corresponding derivatives) when $|x|\to +\infty$.
\end{lem}

We now want to apply this lemma with $V=W_{k,t}^{\rm eff}$
and $\mu_V=\mu_{k,t}^*$ (so that $\Xi=\Xi_k$, see \eqref{eq:Xi}),
and prove the invertibility of $\XXi_t$ by a fixed point argument.
We notice that the constants appearing in the above result depend only on the smoothness 
of $V$ and on the assumption \eqref{eq:non deg V}, that is satisfied by 
 $\mu_{k,t}^*$ thanks to Lemma \ref{lem:support}.
In particular, when applied with $V=W_{k,t}^{\rm eff}$ and $\mu_V=\mu_{k,t}^*$ all the constants are uniform for $t \in [0,1]$.
Also, being $F_0^a$ of class $C^\infty$, the smoothness of $W_{k,t}^{\rm eff}$ is the same as the one of $W_k$
(see \eqref{eq:Veff2}).

\begin{prop}
\label{prop:Wk}
There exists $\alpha>0$ such that the following holds.
Assume that the functions $W_{1},\ldots,W_d:\R\to \R$ are  of class $C^\sigma$ for some $\sigma \geq 4$.
Suppose that $|a| \leq \alpha$, and let $t \in [0,1]$. 
Then,
for any family of functions $g_1,\ldots,g_d:\mathbb R\ra\mathbb R$ of class $C^{s}$ with $s \in [3,\sigma-1]$,
there exist a unique family of constants $(c_{g_1},\ldots,c_{g_d})$ such that
the equation
\begin{equation}
\label{eq:Xi t g}
\XXi_t (\Psi_1,\ldots,\Psi_d)=(g_1,\ldots,g_d)+(c_{g_1},\ldots,c_{g_d})
\end{equation}
has a solution $\Psi_1,\ldots,\Psi_d:\mathbb R\ra\mathbb R$ of class $C^{s-2}$.
In addition, there exists a finite constant $\bar C_0$ such that
\begin{equation}
\label{eq:bound norms2}
\max_{k=1,\ldots,d}\|\Psi_k\|_{C^{1}(\R)}\le 
 \bar C_0 \max_{k=1,\ldots,d} \|g_k\|_{C^{3}(\R)}.
\end{equation}
Furthermore,
there exists $\gamma_s>0$ such that $\Psi_k$ and its derivatives 
up to order $s-1$ decay like $O\Big(\frac{1}{[(W_{k,t}^{\rm eff})'(x)]^{\gamma_s}}\Big)$ as $|x|\to +\infty$.
\end{prop}

\begin{proof}
Define the operator 
$$
\Upsilon_k^{aV}(\Psi_1,\ldots,\Psi_d):=\sum_{\ell=1}^d\int \Psi_\ell(y)\,\partial_1f_{k\ell,0}(y,\cdot)\,d\mu_{\ell,t}^*(y),
$$
so that \eqref{eq:Xi t g} can be rewritten as
$$
\Xi_k \Psi_k -2t\,\Upsilon_k^{aV}(\Psi_1,\ldots,\Psi_d)=g_k+c_{g_k}\qquad \forall\,k=1,\ldots,d.
$$
Recalling that $\partial_1f_{k\ell,0}(\cdot,y)$ is a smooth function with all derivatives of size $O(|a|)$,
for any family of bounded functions $\Psi_k:\R\to \R$ it holds
\begin{equation}
\label{eq:norm Ak}
\bigl\|\Upsilon_k^{aV}(\Psi_1,\ldots,\Psi_d)\bigr\|_{C^3(\R)}
\leq \bar C |a| \,\max_{k=1,\ldots,d}\|\Psi_k\|_{C^1(\R)}
\end{equation}
for some universal constant $\bar C$.
To prove the result we simply apply a fixed point argument:
more precisely, we set $(\Psi_{1,(0)},\ldots,\Psi_{d,(0)})=(0,\ldots,0)$ and we recursively define, for $j \geq 1$,
$$
\Psi_{k,(j+1)}:=(\Xi_k)^{-1}\Bigl(2t\,\Upsilon_k^{aV}(\Psi_{1,(j)},\ldots,\Psi_{d,(j)}) + g_k \Bigr),\qquad k=1,\ldots,d.
$$
Applying Lemma \ref{lem:Xi} with $V=W_{k,t}^{\rm eff}$
and $\mu_V=\mu_{k,t}^*$ (so that $\Xi=\Xi_k$) we deduce that
$$
\Psi_{k,(j)} \in C^{1}(\R) \qquad \forall\,j \geq 1,\,\forall\, k=1,\ldots,d.
$$
Also, by the linearity of $\Xi_k$ and $\Upsilon_k^{aV}$ we have 
$$
\Psi_{k,(j+1)} -\Psi_{k,(j)} = (\Xi_k)^{-1}\Bigl(2t\,\Upsilon_k^{aV}(\Psi_{1,(j)}-\Psi_{1,(j-1)},\ldots,
\Psi_{d,(j)} - \Psi_{d,(j-1)}) \Bigr),
$$
so it follows from \eqref{eq:bound norms} and \eqref{eq:norm Ak} that 
$$
\max_{k=1,\ldots,d}\bigl\|\Psi_{k,(j+1)} -\Psi_{k,(j)}\bigr\|_{C^{1}(\R)} \leq 2t \hat C_3\bar C |a|\,
\,\max_{k=1,\ldots,d}\|\Psi_{k,(j)} - \Psi_{k,(j-1)}\|_{C^1(\R)}.
$$
Hence, if we choose $\alpha$ small enough so that $\hat C_3\bar C \alpha \leq 1/4$
we deduce that $\{\Psi_{k,(j)}\}_{j \geq 1}$ is a Cauchy sequence in $C^{1}$ for all $k=1,\ldots,d$.
Recalling that the operator $\Xi_k$ are continuous with respect to the $C^1$ topology,
we deduce that the sequence $(\Psi_{1,(j)},\ldots,\Psi_{d,(j)})$ converges a solution of our problem $(\Psi_1,\ldots,\Psi_d)$.

Applying \eqref{eq:bound norms} and \eqref{eq:norm Ak}  again, we deduce that
\begin{align*}
\max_{k=1,\ldots,d}\|\Psi_{k,(j+1)}\|_{C^{1}(\R)}&\le 
2t\, \hat C_3 \bar C|a|\,
\,\max_{k=1,\ldots,d}\|\Psi_{k,(j)}\|_{C^1(\R)}
+ \hat C_3 \max_{k=1,\ldots,d} \|g_k\|_{C^{3}(\R)} \\
&\leq \frac12
\,\max_{k=1,\ldots,d}\|\Psi_{k,(j)} \|_{C^{1}(\R)}
+ \hat C_3\max_{k=1,\ldots,d} \|g_k\|_{C^{3}(\R)},
\end{align*}
so \eqref{eq:bound norms2} follows by letting $j \to \infty$.
In addition, Lemma \ref{lem:Xi} implies that 
$\Psi_k$ decays like $O\Big(\frac{1}{(W_{k,t}^{\rm eff})'(x)}\Bigr)$ as $|x|\to +\infty$.
Furthermore, since $\Upsilon_k^{aV}(\Psi_1,\ldots,\Psi_d) \in C^\infty$, it follows by 
 \eqref{eq:bound norms} that
$$
\max_{k=1,\ldots,d}\|\Psi_{k}\|_{C^{s}(\R)}\leq \bar C_s,
$$
showing that $\Psi_k \in C^s$.

To prove the final statement we note that, since $\|\Psi_k\|_{C^s(\R)}\leq \bar C_s$ and $\Psi_k$
decays like $O\Big(\frac{1}{(W_{k,t}^{\rm eff})'}\Big)$, by interpolation inequalities
the derivatives of $\Psi_k$ up to order $s-1$ decay as an inverse power of $(W_{k,t}^{\rm eff})'$.
\end{proof}

We can now apply the above proposition to invert the first equation in \eqref{eq:main Xi}
and find a solution $\yy_{k,t}^0$ of class $C^{\sigma-3}$.
Then (since now $\yy_{k,t}^0$ is given) we solve the second equation in \eqref{eq:main Xi}
using again the proposition above, and finally we invert the third equation.
In this way, in analogy with \cite[Lemma 3.3]{BFG}
 we obtain the following result (we recall that a function of two variables belongs to $C^{\tau,\tau'}$
 if it is $\tau$ times continuously differentiable with respect to the first variable and $\tau'$ times with respect to the second):

\begin{cor}\label{cor:regular} 
Let $\alpha$ be as in Proposition \ref{prop:Wk}.
Assume that $W_{k}:\R\to \R$ are  of class $C^\sigma$ for all $k=1,\ldots,d$ for some $\sigma \geq 10$,
and that $|a| \leq \alpha$. Then there exist functions $\yy_{k,t}^0,\yy_{k,t}^1,\zz_{k\ell,t}$
solving \eqref{eq:main Xi}, and a finite universal constant $C_\sigma$, such that
$$
\|\yy_{k,t}^0\|_{C^{\sigma-3}(\R)}+\|\yy_{k,t}^1\|_{C^{\sigma-9}(\R)}+\sum_{\tau+\tau' \leq \sigma-6}\|\zz_{k\ell,t}\|_{C^{\tau,\tau'}(\R\times \R)}\leq C_\sigma\qquad \forall\,k,\ell=1,\ldots,d.
$$
Moreover these functions and their derivatives (except the last ones) decay as an inverse power of $(W_{k,t}^{\rm eff})'(x)$ as $|x|\to +\infty$.\end{cor}

Recalling \eqref{eq:RNt},
it follows by Lemma \ref{lem:small Rt} and Corollary \ref{cor:regular} that
 $${\mathcal R}^N_t(\YY^N)=
 C^N_t+O\biggl(\frac{|\phi^M_\#M^N|^3}{N};(\yy_{k,t}^1)',\partial_1\zz_{k\ell,t},\partial_{2}\zz_{kk,t}\biggr).
$$
But in fact, since ${\mathcal R}^N_t(\YY^N)$ is centered (compare with \cite[Section 3.5]{BFG}),
we deduce that
$$
{\mathcal R}^N_t(\YY^N)=
O\biggl(\frac{|\phi^M_\#M^N|^3}{N};(\yy_{k,t}^1)',\partial_1\zz_{k\ell,t}, \partial_{2}\zz_{kk,t}\biggr).
$$
The goal of the next section is to control the right hand side.

\subsection{Getting rid of the rest}\label{sect:rest}
 We start by using concentration inequalities to control $M_k^N-\mathbb E[M_k^N]$.
\begin{lem}\label{ub1} Let Hypothesis \ref{hypo} hold, and let $a_0$ be as in Section \ref{eqmeas}.
 For $a\in [-a_0,a_0]$ there exists $c'>0$ such that,
for any Lipschitz function $f:\R\to\R$, for all $\delta>0$, all $t\in [0,1]$ and $k\in\{1,\ldots,d\}$,
$$Q^{N,aV}_{t}\biggl(\biggl|\sum_{i=1}^N f(\lambda_i^k)-\mathbb E\Bigl[\sum_{i=1}^N f(\lambda_i^k)\Bigr]\biggr|\ge \|f\|_L\delta\biggr)\le 2e^{-c'\delta^2},$$
where $\|f\|_L$ denotes the Lipschitz constant of $f$.
\end{lem}
\begin{proof}  $Q^{N,aV}_{t}$  being a probability measure with  uniformly log-concave density (see Section \ref{eqmeas}), Bakry-Emery and Herbst argument applies
(see e.g. \cite[Section 4.4]{AGZ}).
\end{proof}

We now need  to control the difference between $\mathbb E[L^{N}_k]$ and its limit $\mu^*_{k,t}$. We shall do this in two steps:
we first derive a rough estimate which only provides a bound of order $N^{-{1}/{2}}$ following  ideas initiated in  \cite{MMS}, and in a second step
we use loop equations to get a bound of order $\log N/N$, see e.g. \cite{Sh}. This two steps approach was already developed in \cite{BG1,BG2,BGK}.
To get the rough estimate, we shall use the distance $d(\mu,\mu')=d(\mu-\mu')$ on the space of probability measures on $\mathbb R$ defined on centered measures $\nu$ by
$$
d(\nu):=\left(2\iint \log |x-y|^{-1} d\nu(x)\,d\nu(y)\right)^{1/2}=\sqrt{\int_\R \frac{1}{|\tau|}\,|\hat \nu(\tau)|^2\, d\tau},
$$
where $\hat\nu$ denotes the Fourier transform of the measure $\nu$. Because this distance blows up on measures
with atoms, we shall consider the following regularization of the empirical measure: For a given vector $\lambda:=(\lambda_1<\lambda_2<\cdots<\lambda_N)$, we denote by $\tilde \lambda:=(\tilde\lambda_1<\cdots<\tilde\lambda_N)$ its transformation given by
$$\tilde\lambda_1:=\lambda_1,\quad\tilde\lambda_{i+1}:=\tilde\lambda_i+\max(\lambda_{i+1}-\lambda_i, N^{-3})\,.$$
We denote by $\tilde L^{N}_k$ the empirical measure of the $\tilde\lambda_i^k$, and by $\bar L^{N}_k$ its convolution with the uniform measure on $[0, N^{-4}]$. We then claim that:
\begin{lem}
\label{lem:Lip 12} Let Hypothesis \ref{hypo} hold.
Then there is $a_0>0$ so that, for $a\in [-a_0,a_0]$, there exist $c,C$ positive constants such that, for
all $\delta>0$ and $t\in [0,1]$: 
\begin{itemize}
\item $$Q^{N,aV}_{t}\left(\max_{1\le k\le d} d(\bar L^{N}_k,\mu^*_{k,t})\ge \delta\right)\le e^{CN\log N-\frac\beta6 \delta^2 
 N^2} +Ce^{-cN^2}\,.$$
\item If $f:\R\to\R$ is Lipschitz and belongs to $L^2(\R)$, then
$$Q^{N,aV}_{t}\left(\left|\int f(x)\,d (L^{N}_k-\mu^*_{k,t})(x)\right|\ge \delta \|f\|_{\frac{1}{2}}+N^{-4}\|f\|_L\right)
\le e^{CN\log N-\frac\beta8 \delta^2 N^2} +Ce^{-cN^2},$$
where $\|f\|_{\frac 1 2}:=(\int_\R |\tau|\,|\hat f(\tau)|^2 d\tau)^{1/2}$.
\end{itemize}
\end{lem}
\begin{rem}\label{remnorm}{\rm
Note for later use that if $f$ is supported in $[-M,M]$, then there exists a constant $C(M)$ finite such that
$$\|f\|_{\frac 1 2}\le C(M)\|f'\|_\infty\,.$$
Indeed,
$$\|f\|_{\frac 1 2}^2=\int |s||\hat f(s)|^2 ds=\int\frac{1}{|s|}|\widehat{f'}(s)|^2 ds=-2\iint\log|x-y|\,f'(x)\,f'(y)\,dx\,dy\le C(M)\,\|f'\|_\infty^2\,.$$}
\end{rem}
\begin{proof}[Proof of Lemma \ref{lem:Lip 12}]
We just recall the main point of the proof, which is almost identical to that of \cite[Corollary 3.5]{BGK}. In the
latter article, the potential is only depending polynomially on the measures rather than being an infinite series.
 It turns out that  the main point is to show that
$$S(\nu):=\frac{\beta}{2} \sum_k d(\nu_k)^2  -\sum_{k,\ell} D^2_{k\ell} F^a_0(\mu^*_{1,t},\ldots,\mu^*_{d,t},\tau_B^N) [\nu_k,\nu_\ell]$$
is uniformly convex on the  set $P([-M,M])^d$ of probability measures on $[-M,M]$, so that its square root defines a Lipschitz distance. Here, we more simply notice that for $a$ small enough
\begin{equation}\label{contd}S(\nu)\ge \frac{\beta}{4} \sum_{k=1}^d d(\nu_k)^2\,.\end{equation}
Indeed, the latter  amounts to bound from above the second term in the definition of $S$. But since $D^2_{k\ell} F_0^a(\mu_{1,t}^*,\ldots,\mu_{d,t}^*)[\delta_x,\delta_y]$ is smooth and compactly supported, so we can always write 
$$D^2_{k\ell} F_0^a(\mu_{1,t}^*,\ldots,\mu_{d,t}^*,\tau_B^N)[\delta_x,\delta_y]=\int d\xi\,\int d\zeta\, e^{i\xi x+i\zeta y} \widehat{D^2_{k\ell} F_0^a}(\mu_{1,t}^*,\ldots,\mu_{d,t}^*,\tau_B^N)(\xi,\zeta)$$
and for any centered measures $\nu_k,\nu_\ell$ we get, by Cauchy-Schwartz inequality,
$$ |D^2_{k\ell} F^a_0(\mu^*_{1,t},\ldots,\mu^*_{d,t},\tau_B^N) [\nu_k,\nu_\ell]|\le d(\nu_k)\,d(\nu_\ell)
\left(\int d\xi\int d\zeta\, | \widehat{D^2_{k\ell} F_0^a}(\mu_{1,t}^*,\ldots,\mu_{d,t}^*,\tau_B^N)(\xi,\zeta)|^2|\xi||\zeta|\right)^{\frac 1 2}.$$
Hence we can always choose $a$ small enough so that the last term is as small as wished, proving \eqref{contd}.

Let us sketch the rest of the proof. By localizing the eigenvalues in a very tiny neighborhood around the quantiles of $\mu^*_{k,t}$ it is possible to show (see e.g. \cite[Lemma 3.11]{BGK}) that there exists a finite constant $C$ such that
$$Z^{N,aV}_{t}\ge e^{-N^2 J_t^a(\mu_{1,t}^*,\ldots,\mu_{d,t}^*)-CN\log N}$$
where $Z^{N,aV}_t$ is as in \eqref{eq:ZN} and
$$J^a_t(\mu_1,\ldots,\mu_k):=
\frac{1}{2}\sum_{k=1}^d  \biggl(\iint \bigl[ W_k(x)+W_k(y) -\beta\log|x-y|\bigr]\,d\mu_k(x)\,d\mu_k(y)\biggr)
-tF_0^a(\mu_1,\ldots,\mu_k,\tau_B^N) \,.$$
Then, writing $L_N:=(L^N_1,\ldots,L_d^N)$, $\bar L_N:=(\bar L^N_1,\ldots,\bar L_d^N)$,
and $\mu^*:=(\mu_{1,t}^*,\ldots,\mu_{d,t}^*)$, one has
\begin{align*}
&\frac{\beta}{2}\int_{x\neq y} \log |x-y|\,dL_N(x)\, dL_N(y)-tF_0^a(L_N,\tau_B^N)-\sum_k\int W_k \,dL_k^N+J^a_t(\mu_{1,t}^*,\ldots,\mu_{d,t}^*)
\\
&=\frac{\beta}{2}
\int_{x\neq y}\log |x-y|\,d[L_N-\mu^*](x)\,d[L_N-\mu^*](y)+R(L_N-\mu^*)\\
&= \frac{\beta}{2}\int\log|x-y|\,d[\bar L_N-\mu^*](x)\,d[\bar L_N-\mu^*](y)
+R(L_N-\mu^*)+O(\log N/N),
\end{align*}
where we used the regularization $\bar L_N$ of $L_N$ to add the diagonal term $x=y$ in the logarithmic term up
to an error of order $N\log N$, we bounded uniformly $F_1^a$ and $F_2^a$ up to an error of order $N$, and we set
$$R(\nu):=\sum_k\int f_k(x)d\nu_k(x)-D^3 F^a_0(\mu^*+\theta \nu,\tau_B^N)[\nu^{\otimes 3}]$$
for some $\theta\in [0,1]$ and some functions $f_k$ vanishing on the support of the equilibrium measure $\mu^*_{k,t}$, positive outside, and going to infinity like $W_k$
(see \cite[Lemma 3.11]{BGK} for more details).
In this way one deduces that
$$Q^{N,aV}_{t}\left(\max_{1\le k\le d}d(\bar L^{N}_k,\mu^*_{k,t})\ge \delta\right)
\le e^{CN\log N} \int_{\max_{1\le k\le d}d(\bar L^{N}_k,\mu^*_{k,t})\ge \delta} e^{-N^2 d(\bar L_N,\mu^*)^2-N^2R(\bar L_N-\mu^*)} 
\prod d\lambda_i^k.$$
By the large deviation principle in Theorem \ref{ldplem}, we see that the cubic term in $R$  is negligible compared to the quadratic term on a set with probability greater than $1-e^{-c N^2}$. Thus, setting $\bar M^N_k:=N(\bar L^{N}_k-\mu^{*}_{k,t})$, we get
\begin{align*}
&Q^{N,aV}_{t}\left(\max_{1\le k\le d}d(\bar M^N_k)\ge  N\delta\right)\\
&\le  e^{CN\log N}\biggl[ \int_{\max_{1\le k\le d}d(\bar M^N_k)\ge N\delta} e^{-\frac{\beta}{5}
\sum_{k=1}^d d(\bar M^N_k)^2-N^2\sum_k\int f_k(x)dL^{N}_k(x)}
\prod d\lambda_i^k +e^{-cN^2}\biggr]\\
&\le e^{CN\log N}(e^{-\frac{\beta}{6} N^2 \delta^2}+e^{-cN^2})\end{align*} 
This gives the first bound of the lemma, from which the second is easily deduced since
$$
\biggl|\int f(x)\,d\nu(x) \biggr|=\biggl|\int \hat f(\tau) \hat \nu(\tau) \,d\tau\biggr|\le \|f\|_{\frac 1 2}d(\nu)
$$
and
$$
\biggl|\int f(x)\,d(L^{N}_k-\bar L^{N}_k)(x)\biggr|\le \frac{\|f\|_L}{N^{4}}.
$$
\end{proof}
We finally improve the previous bounds to get an error of order $\log N/N$ instead of $\log N/\sqrt{N}$. 
\begin{lem} \label{ub2}  Let Hypothesis \ref{hypo} hold, and given a function $f:\R\to \R$ define the norm given by
\begin{equation}
\label{eq:norm}
|||f|||:=\int (1+|\tau|^{7})|\hat f(\tau)| \,d\tau\,.
\end{equation}
There exists $a_0>0$ so that,
for all $a\in [-a_0,a_0]$ and all functions $f:\R\to \R$ with $|||f|||<\infty$,
$$
\biggl|\int \Bigl[\int f(x) \,d(L^{N}_k-\mu^{*}_{k,t})(x)\Bigr]\,dQ^{N,aV}_{t}\biggr|\le C ||| f|||\,\frac{\log N}{N}
$$
for some constant $C$ independent of $a$ and $f$.
\end{lem}
\begin{proof} 
Before starting the proof, we recall the notation $L_N:=(L^N_1,\ldots,L_d^N)$ and $\mu^*:=(\mu_{1,t}^*,\ldots,\mu_{d,t}^*)$.

To improve the bound we just obtained, we use the loop equation. Such an equation is simply obtained by integration by parts and, for any smooth test function,
reads as follows:
\begin{align*}
-\frac{1}{N}\iint f'(x)\,dL^{N}_k(x) \,dQ^{N,aV}_{t}&=
\frac{1}{N^2}\sum_{i=1}^N \int f(\lambda_i^k) \partial_{\lambda_i^k} \left(\frac{dQ^{N,aV}_{t}}{\prod d\lambda_j^\ell}\right)\prod d\lambda_j^\ell\\
&= \int\bigg(\int f(x)\left(t[\partial_x D_k F^a](L_N,\tau_B^N)[\delta_x]-W_k'(x)\right) dL^{N}_k(x)\\
&\qquad+\frac{\beta}{2}\iint \frac{f(x)-f(y)}{x-y} dL^{N}_k(x)dL^{N}_k(y)-\frac{\beta}{2N}\int f'(x) dL^{N}_k(x)\bigg)\,dQ^{N,aV}_{t},\end{align*}
where $F^a:=\sum_{l=0}^2 F^a_l N^{-l}$. Recalling that $M^N_k=N(L^{N}_k-\mu^{*}_{k,t})$ and \eqref{eq:Xi},
we rewrite the above equation as 
\begin{equation}\label{vc}
\int \biggl[\int \Xi_k f \,dM^{N}_k-t\sum_{\ell\neq k}  \iint \partial_ x D_{k\ell} F^a_0 (\mu^*,\tau_B^N)[\delta_y,\delta_x]f(x)\,d\mu^{*}_{k,t}(x) \,dM^{N}_\ell(y)\biggr]\,dQ^{N,aV}_{t}
=\sum_{\gamma=1}^4 R^N_\gamma (f)\end{equation}
where
\begin{align*}
R^N_1(f)&:=\biggl(1-\frac{\beta}{2}\biggr)\iint f'(x) \,dL^{N}_k(x) \,dQ^{N,aV}_{t},\\
R^N_2(f)&:= \frac{\beta }{2 N}\iint \frac{f(x)-f(y)}{x-y} \,dM^N_k(x)\,dM^N_k(y) \,dQ^{N,aV}_{t},\\
R^N_3(f)&:= Nt \int\biggl[\int f(x)\, \partial_x D_k (F^a-F^a_0)(L_N,\tau_B^N)[\delta_x] \,dL^{N}_k(x)\biggr] dQ^{N,aV}_{t},\\
R^N_4(f)&:=\frac{t}{N}\sum_{\ell }\int\left( \int f(x)\, \partial_x D_{k\ell} F^a_0 (\mu^*+\theta( L_{N}-\mu^*),\tau_B^N)[M^N_k,M^{N}_\ell]\right)\,dQ^{N,aV}_{t},\\
\end{align*}
and the last term was computed using a Taylor expansion.
Writing $f(x)=\int e^{ixs}\hat f(s) \,ds$ and
 noticing that $\|e^{i\lambda.}\|_{\frac 1 2}+ \|e^{i\lambda.}\|_L\le 2(1+|\lambda|)$ so that Lemma \ref{lem:Lip 12} entails
 $$\int \Bigl|\widehat{(M^N_k)}({\lambda})\Bigr|^2 dQ^{N,aV}_{t}\le C N \log N (1+|\lambda|)^2,$$
 we get
\begin{align*}
|R^N_1(f)|&\le\|f\|_L,\\
|R^N_2(f)|&\le N^{-1}\int d\tau  |\hat f(\tau)|\int_0^1 d\alpha \,|\tau|\,  \int \Bigl|\widehat{(M^N_k)}(\alpha\tau)\Bigr|\, \Bigl|\widehat{(M^N_k)}\bigl((1-\alpha)\tau\bigr)\Bigr| \,dQ^{N,aV}_{t}\\
&\le N^{-1}\int d\tau |\tau| |\hat f(\tau)|\int_0^1d\alpha\, \int \Bigl|\widehat{(M^N_k)}(\alpha\tau)\Bigr| ^2 dQ^{N,aV}_{t}\le \log N \int (1+|\tau|^3)|\hat f(\tau)|\,d\tau,\\
|R^N_3(f)|&\le C\|f\|_\infty,\\
|R^N_4(f)|&\le C\log N\|f\|_\infty,
\end{align*}
where we used Lemma \ref{lem:Lip 12} for the second and fourth terms, and to bound the last term we noticed that, since $F^a_0$ is smooth 
and it is of size $O(|a|)$ together with its derivatives, we have 
\begin{equation}\label{poi} \max_{k\ell} \left|\bigl[\widehat{ \partial_x D_{k\ell}F^a_0}\bigr](\lambda,\zeta)\right|\le\frac{
\hat C\,|a|}{(1+\lambda^2 )(1+|\zeta|^{10})}.\end{equation}
Hence, since
\begin{multline*}
\biggl| \int\biggl[ \int \partial_x D_{k\ell} F^a_0 (\mu^*,\tau_B^N)[\delta_y,\delta_x]\,f(x)\,d\mu^{*}_{k,t} (x) \,d M^{N}_\ell(y)\biggr]\,dQ^{N,aV}_{t}\biggr|\\
\le \iint \left|\widehat{f\cdot d\mu^{*}_{k,t}}(\zeta)  \right| |\widehat {\partial_x D_{k\ell }F_0^a}(\xi,\zeta)|\left|\int
\widehat{(M^{N}_\ell)}({\xi})\,dQ^{N,aV}_{t}\right|\, d\zeta\, d\xi \,,
 \end{multline*}
 we deduce from \eqref{vc}
that
\begin{multline*}
\biggl|\iint f(x) \,dM^N_k(x)\, dQ^{N,aV}_{t}\biggr|
\le \|\Xi_k^{-1}f\|_\infty \sum_{\ell\neq k }\int |\widehat {\partial_x D_{k\ell}F^a_0}(\xi,\zeta)|\left|\int\widehat{(M^{N}_\ell)}(\zeta)\,dQ^{N,aV}_{t}\right| \,d\zeta \,d\xi \\
\quad +C\|\Xi_k^{-1}f \|_{C^1(\mathbb R)}+\log N \int (1+|\tau|^3)|\widehat {\Xi_k^{-1} f}(\tau)|\,d\tau. \end{multline*}
Applying the above bound with $f(x)=e^{i\lambda x}$ and using  \eqref{eq:bound norms} with $\Xi=\Xi_k$, we get
\begin{equation}
\label{eq:deltaN}
\delta_N(\lambda):= \max_{1\le k\le d} \biggl|\int \widehat{M^N_k} (\lambda) \,dQ^{N,aV}_{t}\biggr|
\le\lambda^2 \int \max_{k,\ell} \left|\widehat{ \partial_x D_{k\ell}F^a_0}(\lambda,\zeta)\right| \delta_N(\zeta) \,d\zeta + C(1+|\lambda|^7)\log N.
\end{equation}
By \eqref{poi}, we deduce from the above equation that 
\begin{align*}
\int\frac{1}{1+|\lambda|^{10}}\,\delta_N(\lambda)\,d\lambda&\le \hat C\,|a|\,\biggl(\int \frac{1}{1+|\lambda|^{10}}\,d\lambda\biggr) \int\frac{1}{1+|\zeta|^{10}}\,\delta_N(\zeta)\,d\zeta
+C\biggl(\int \frac{1+|\lambda|^7}{1+|\lambda|^{10}}\,d\lambda \biggr)\log N\\
&\leq C\,\hat C\,|a|\,\int\frac{1}{1+|\zeta|^{10}}\,\delta_N(\zeta)\,d\zeta+C\,\log N.
\end{align*}
In particular, if $a$ is sufficiently small so that $C\,\hat C\,|a| \leq 1/2$,
we can reabsorb
the first term in the right hand side and obtain 
$$
\int\frac{1}{1+|\lambda|^{10}}\,\delta_N(\lambda)\,d\lambda \leq 2C\,\log N.
$$
Plugging back this control in \eqref{eq:deltaN} and using again  \eqref{poi},  we finally  get the bound
$$\delta_N(\lambda)\le C(1+|\lambda|^7 )\log N\,.$$
Therefore, using the identity
$f(x)=\int \hat f(\tau) e^{i\tau x} d\tau$ we conclude
$$\max_{1\le k\le d} \biggl|\int\Bigr[  \int f (x)\,dM^N_k(x)\Bigr]\,dQ^{N,aV}_{t}\biggr|\le
\int |\hat f(\tau)| \delta_N(\tau) \,d\tau\le C\log N \int (1+|\tau|^{7}) |\hat f(\tau)|\, d\tau\,,$$
as desired.
\end{proof}
A straightforward corollary of Lemmas \ref{ub1} and \ref{ub2} is the following:
\begin{cor}\label{ub3}
There exists $a_0>0$ so that,
for all $a\in [-a_0,a_0]$, there are finite positive constants $C,c'$ 
such that, for all $f:\R\to \R$ with $|||f|||<\infty$ and
all $\delta\ge 0$, we have
\begin{equation}
Q^{N,aV}_{t}\left( \biggl|\int f(x) \,dM^{N}_k(x)\biggr|\ge \delta \|f\|_L +
 C ||| f|||\log N\right)\le 2e^{-c'\delta^2}.
\end{equation}
In particular, for all $p\geq 1$ there exists a finite constant $C_p$ such that
$$\bigl\| M^{N}_k[f]\bigr\|_{L^p(Q^{N,aV}_{t})}=\biggl\|\int f(x) \,dM^{N}_k(x)\biggr\|_{L^p(Q^{N,aV}_{t})}\le C_p\bigl(\|f\|_L+|||f|||\log N\bigr).$$
\end{cor}
Thanks to this corollary we get:
\begin{cor}\label{ub3b} Assume $\phi^M\in C^9(\mathbb R)$, vanishes outside  $[-M,M]$ and   is bounded by $M$.
There exists $a_0>0$ so that,
for all $a\in [-a_0,a_0]$ and for all $\zeta>M$, there  are  finite constants $c_\zeta, C_\zeta,c>0$ 
so that, for all $\delta\ge 0$, we have
\begin{equation}
Q^{N,aV}_{t}\Big( \| \phi^M_\# M^{N}_k\|_\zeta \ge \delta c_\zeta +C_\zeta \log N\Big)\le 2e^{-c\delta^2}\,.
\end{equation}
\end{cor}

\begin{proof}
Using  Corollary \ref{ub3} with $f(x)=(\phi^M(x))^p$, together with Remark \ref{remnorm},  we deduce that there exist constants $c_0,C_0>0$, only depending on $\phi^M$, such that 
$$Q^{N,aV}_t\Big( \bigl|M^N_k\bigl((\phi^M)^p\bigr)\bigr|\ge c_0 pM^{p-1} \delta + C_0M^p p^7\log N\Big)\le 2e^{-c'\delta^2}\,.$$
Therefore, for $\zeta>M$ we find $c_1,C_1>0$ such that 
$$Q^{N,aV}_t\Big(\bigl|M^N_k\bigl((\phi^M)^p\bigr)\bigr| \ge c_1 \zeta^{p} \delta + C_1\zeta^p \log N\Big)\le 2e^{-c'\delta^2\left(\frac{\zeta^{2p}}{M^{2p}p^2}\right)}\,.$$
Applying this bound for $p\in [1,  e^{cN^2/2}]$, by a union bound we deduce that there exists $c''>0$ such that 
$$Q^{N,aV}_t\biggl( \max_{1\le p\le e^{cN^2/2}}\zeta^{-p}
\bigl|M^N_k\bigl((\phi^M)^p\bigr)\bigr|\ge c_1  \delta + C_1 \log N\biggr)\le 2e^{-c''\delta^2}\,.$$
On the other hand, for $p\ge e^{cN^2/2}$ the bound is trivial as 
$$\zeta^{-e^{cN^2/2}}
\bigl|M^N_k\bigl((\phi^M)^{e^{cN^2/2}}\bigr)\bigr| \leq N \biggl(\frac{M}{\zeta}\biggr)^{e^{cN^2/2}}\le c_1\delta+C_1\log N$$
as soon as $N$ is large enough. 
This concludes the proof.
\end{proof}
\smallskip

Thanks to this corollary, we can finally estimate the rest
$$
{\mathcal R}^N_t(\YY^N)=O\biggl(\frac{|\phi^M_\#M^N|^3}{N};(\yy_{k,t}^1)',\partial_1\zz_{k\ell,t}, \partial_{2}\zz_{kk,t}\biggr)
$$
with $C(\log N)^3/N$.
Indeed, recalling \eqref{eq:rest precise}, using Fourier transform we have
$$
\iiint \psi(x,y,z) dM^{N}_k(x)dM^{N}_\ell(y)dM_m^N(z)
=\iiint \hat \psi(\xi,\zeta,\theta) \,M^{N}_k[e^{i\xi \cdot}]\,M^{N}_\ell[e^{i\zeta \cdot}]\,M_m^N[e^{i\theta \cdot}]\,d\xi\,d\zeta\,d\theta,
$$
so applying Corollaries \ref{ub3} and \ref{ub3b}, and recalling \eqref{eq:norm}, we can bound our rest by
$$
C\frac{(\log N)^3}N +C\frac{(\log N)^3}{N} \iiint |\hat \psi(\xi,\zeta,\theta)|\,(1+|\xi|)^7\,(1+|\zeta|)^7\,(1+|\theta|)^7\,d\xi\,d\zeta\,d\theta
$$
with probability greater than $1-N^{-cN}$.
Since all the functions involved decay at infinity, for the above integral to converge it is enough to assume that $\psi \in C^{26}$,
as this ensures that
$$
|\hat \psi(\xi,\zeta,\theta)|\,(1+|\xi|)^7\,(1+|\zeta|)^7\,(1+|\theta|)^7\leq \frac{C}{1+|\xi|^5+|\zeta|^5+|\theta|^5} \in L^1(\R^3).
$$
Recalling that by assumption $\psi$ is as smooth as
$(\yy_{k,t}^1)',\partial_1\zz_{k\ell,t},$ or $ \partial_{2}\zz_{kk,t}$,
by Corollary \ref{cor:regular} the assumption is satisfied provided $W_k \in C^\sigma$ with $\sigma \geq 36$.
Thanks to our Hypothesis \ref{hypo}, this concludes the proof of \eqref{eq:toprove1}.
As explained at the end of Section \ref{sect:simplify} this implies \eqref{eq:toprove}, which combined with 
\eqref{eq:close1},  \eqref{eq:close2}, and \eqref{aptr2} proves \eqref{aptr}.\\

Before concluding this section, we prove an additional estimate on the size of the integral of smooth functions against the measure $M_N$.
Corollary \ref{ub3} provides a very strong bound on the probability that $\int f\,dM_N$ is large when $f$
is a fixed function.
We now show how to obtain an estimate that holds true when we replace $\int f\,dM_N$ by its supremum over smooth functions.
\begin{lem}\label{lem:ub4}
There exists $a_0>0$ so that,
for all $a\in [-a_0,a_0]$,
the following hold:
for any $\ell \geq 0$ there are finite positive constants $C_\ell,c_\ell$ such that  
\begin{equation}
Q^{N,aV}_{t}\biggl(\sup_{\|f\|_{C^{\ell+9}(\R)}\leq 1} \biggl|\int f(x) \,dM^{N}_k(x)\biggr|\ge  \log N\,N^{1/(\ell+1)} \biggr)\le C_\ell e^{-c_\ell(\log N)^{2+2/\ell}}.
\end{equation}
\end{lem}
\begin{proof}
Since the measure $Q^{N,aV}_{t}$ is supported inside the cube $[-M,M]^N$ (see Section \ref{sect:simplify}), we can assume that 
all functions $f$ are supported on $[-2M,2M]$.
Fix $L \in \mathbb N$ and define the points
$$
x_{m,L}:=-2M +m\frac{4M}L,\qquad m=0,\ldots, L.
$$
Given $f \in C^{\ell+9}_0([-2M,2M])$ with $\|f\|_{C^{\ell+9}}\leq 1$, we set
$g:=f^{(9)} \in C^\ell_0([-2M,2M])$ and define the function
$$
g_L(x):=\sum_{j=0}^{\ell-1}\frac{g^{(j)}(x_{m,L})}{j!}(x-x_{m,L})^j\qquad \forall\,x \in [x_{m,L},x_{m+1,L}].
$$
Note that, since $\|g\|_{C^{\ell}}\leq 1$,
$$
|g(x)-g_L(x)| \leq \|g^{(\ell)}\|_\infty(x-x_{m,L})^{\ell}\leq \biggl(\frac{4M}{L}\biggr)^{\ell} \qquad \forall\,x \in [x_{m,L},x_{m+1,L}],\,\forall\,m=0,\ldots, L-1,
$$
so, by the arbitrariness of $x$,
$$
\|g-g_L\|_{L^\infty([-2M,2M])} \leq (4M)^\ell L^{-\ell}.
$$
Hence, if we set
$$
f_L(x):=\int_{-2M}^x\frac{(x-y)^8}{8!}g_L(y)\,dy,
$$
since $f_L^{(9)}=g_L$ and $f^{(j)}(-2M)=0$ for all $j=0,\ldots,8$, we get
$$
\|f-f_L\|_{L^\infty([-2M,2M])} \leq C_{M,\ell} L^{-\ell}.
$$
Recalling that $M_N$ has mass bounded by $2N$, this implies that 
\begin{equation}
\label{eq:f fL}
\Bigl|\int f\,dM_N-\int f_L\,dM_N\Bigr| \leq 2\,C_{M,\ell}\, N\,L^{-\ell}.
\end{equation}
Fix now a smooth cut-off function $\psi_M:\R\to [0,1]$ satisfying $\psi_M=1$ inside $[-M,M]$
and $\psi_M=0$ outside $[-2M,2M]$, and define
$$
f_{L,M}(x)=\sum_{m=0}^{L-1}\sum_{j=0}^{\ell-1} g^{(j)}(x_{m,L})\hat f_{m,j}(x),
$$
where
$$
\hat f_{m,j}(x):=\psi_M(x)\,\int_{-2M}^x\frac{(x-y)^8}{8!} (y-x_{m,L})^j\chi_{[x_{m,L},x_{m+1,L}]}(y)\,dy
$$
It is immediate to check that $\hat f_{m,j} \in C^{8,1}_0([-2M,2M])$
(that is, $\hat f_{m,j}$ has $8$ derivatives, and its 8-th derivative is Lipschitz), and that $f_{L,M}=f_L$ on $[-M,M]$.
Also, since $\|f\|_{C^{\ell+9}}\leq 1$ we see that $| g^{(j)}(x_{m,L})|\leq 1$ for all $m,j$.
Hence, recalling \eqref{eq:f fL} and the fact that $M_N$ is supported on $[-M,M]$, this proves that for any function $f \in C^{\ell+9}_0([-2M,2M])$ with $\|f\|_{C^{\ell+9}}\leq 1$ there exist some coefficients $\a_{m,j} \in [-1,1]$ such that
$$
\Bigl|\int f\,dM_N-\sum_{m,j}\a_{m,j}\int \hat f_{m,j}\,dM_N\Bigr| \leq 2\,C_{M,\ell}\, N\,L^{-\ell}.
$$
Since $\#\{\hat f_{m,j}\}=\ell L$, this implies that 
\begin{multline}
\label{eq:prob sum}
Q_t^{N,aV}\biggl(\sup_{\|f\|_{C^{\ell+9}}\leq 1}\Bigl|\int f\,dM_N\Bigr|>\log N\,N^{1/(\ell+1)} \biggr) \\
\leq \sum_{m,j}Q_t^{N,aV}\biggl(\Bigl|\int \hat f_{m,j}\,dM_N\Bigr|>\frac{\log N\,N^{1/(\ell+1)} -2\,C_{M,\ell} N\,L^{-\ell}}{\ell L}\biggr).
\end{multline}
We now observe that $\|\hat f_{m,j}\|_{C^{8,1}} \leq A_{M,\ell}$, where $A_{M,\ell}$
is a constant depending only on $M$ and $\ell$.
Thus, recalling that the functions $\hat f_{m,j}$ are supported on $[-2M,2M],$
this yields 
$$
|||\hat f_{m,j}||| \leq A_{M,\ell}',
$$
where the norm $|||\cdot|||$ is defined in \eqref{eq:norm}.
Hence, choosing 
\begin{equation}
\label{eq:def L}
L:=\left \lfloor\hat C_{M,\ell}N^{1/(\ell+1)}(\log N)^{-1/\ell}\right\rfloor
\end{equation}
with $\hat C_{M,\ell}$ large enough so that
$$
\log N\,N^{1/(\ell+1)} -2\,C_{M,\ell} N\,L^{-\ell} \geq \frac{1}{2}\log N\,N^{1/(\ell+1)} ,
$$
we can apply Corollary \ref{ub3} to the functions $\hat f_{m,j}$,
and it follows from \eqref{eq:prob sum} and \eqref{eq:def L} that
\begin{multline*}
Q_t^{N,aV}\biggl(\sup_{\|f\|_{C^{\ell+9}}\leq 1}\Bigl|\int f\,dM_N\Bigr|>\log N\,N^{1/(\ell+1)} \biggr)\\
\leq 
C_{M,\ell}'  L \,e^{-c_{M,\ell}'\left(\frac{\log N\,N^{1/(\ell+1)} }{L}\right)^{2}} 
\leq C_{M,\ell}'' e^{-c_{M,\ell}''(\log N)^{2+2/\ell}}.
\end{multline*}
\end{proof}

\subsection{Reconstructing the transport map via the flow}
\label{sect:flow}
In this section we study the properties of the flow $X_t^N:\mathbb R^{dN}\to \mathbb R^{dN}$ generated by a vector field $\YY_t^N$
 as in \eqref{eq:psi},  i.e., $X_t^N$
solves the ODE
$$
 \dot X_t^N=\YY_t^N(X_t^N),\qquad X_0^N=\operatorname{Id},
$$
and we prove that
$T^N:=X_1^N$ satisfies all the properties stated in Theorem \ref{thm:transport}.

Recalling the form of $\YY_t^N$ (see \eqref{eq:psi}), it is natural to expect that for all $t \in [0,1]$ we can give an expansion for $X_t^N$ as
$$
X_t^N=X_{0,t}+\frac{1}{N}X_{1,t}+\frac{1}{N^2}X_{2,t},
$$
where each component $(X_{0,t})_i^k$ of $X_{0,t}$ should flow accordingly to 
$\yy_{k,t}^0$: more precisely, we define $(X_{0,t}{)_i^k}:=X_{0,t}^k(\lambda_i^k)$ with $X_{0,t}^k:\R\to \mathbb R$  the solution of
\begin{equation}
\label{eq:X0}
\dot X_{0,t}^k=\yy_{k,t}^0(X_{0,t}^k),\qquad X_{0,t}^k(\lambda)=\lambda.
\end{equation}
Recalling the notation
$\hat \lambda=(\lambda^1,\ldots,\lambda^d)$ where
$\lambda^k:=(\lambda_1^k,\ldots,\lambda_{N}^k)$,
we define $$X_{1,t}=\Bigl((X_{1,t})_1^1,\ldots,(X_{1,t})_N^1,\ldots,
(X_{1,t})_1^d,\ldots,(X_{1,t})_N^d\Bigr):\mathbb R^{dN}\to \mathbb R^{dN}$$ to be the solution of the linear ODE 
\begin{equation}
\label{eq:X1}
\begin{split}
(\dot X_{1,t})_i^k(\hat \lambda)&=(\yy_{k,t}^0)'\Bigl(X_{0,t}^k(\lambda_i^k)\Bigr)\cdot (X_{1,t})_i^k(\hat\lambda)
+\yy_{k,t}^1\Bigl(X_{0,t}^k(\lambda_i^k)\Bigr)\\
&+\sum_{\ell=1}^d\int \zz_{k\ell,t}\Bigl(X_{0,t}^k(\lambda_i^k), y\Bigr)\,dM^N_{X_{0,t}^\ell}(y)\\
&+\frac1N\sum_{\ell=1}^d\sum_{j=1}^N \partial_2\zz_{k\ell,t}\Bigl(X_{0,t}^k(\lambda_i^k),X_{0,t}^\ell(\lambda_j^\ell)\Bigr)\cdot (X_{1,t})_j^\ell(\hat \lambda)
\end{split}
\end{equation}
with the initial condition $(X_{1,0})_i^k=0$, where $M^N_{X_{0,t}^\ell}$ is defined as 
$$
\int f(y) \,dM^N_{X_{0,t}^\ell}(y)=\sum_{i=1}^N \biggl[f\bigl(X_{0,t}^\ell(\lambda_i^\ell)\bigr)-\int f\, d\mu_{\ell,t}^*\biggr]\qquad \forall\,f \in C_c(\mathbb R).
$$

\begin{prop} \label{theflow}
Let $\alpha$ be as in Proposition \ref{prop:Wk}.
Assume that $W_{k}:\R\to \R$ are  of class $C^\sigma$ for all $k=1,\ldots,d$, for some $\sigma \geq 16$,
and that $|a| \leq \alpha$.
Then the flow $$X_t^N=\left( (X_t^{N})_1^1,\ldots,(X_t^{N})_N^1,\ldots,(X_t^{N})_1^d,\ldots, (X_t^{N})_N^d\right):\R^{dN}\to\R^{dN}$$ is of class $C^{\sigma-9}$ and the following properties hold:
Let $(X_{0,t})_i^k$ and $(X_{1,t})_i^k$ be as in \eqref{eq:X0} and \eqref{eq:X1} above,
and define $X_{2,t}:\R^{dN}\to\R^{dN}$ via the identity 
$$X_t^N=X_{0,t}+ \frac1N X_{1,t}+\frac{1}{N^2} X_{2,t}\,.$$
Then, for any $t \in [0,1]$,
\begin{equation}\label{boi1}
\max_{k,i}\|(X_{1,t})_i^k\|_{L^4(Q^{N,aV}_0)} \leq C\log N,\qquad \max_{k,i} \|(X_{2,t})_i^k\|_{L^2(Q^{N,aV}_0)} \leq C(\log N)^2.
\end{equation}
In addition, there exist constants $C,c>0$ such that, with probability greater than $1-e^{-c(\log N)^2}$,
\begin{equation}\label{eq:bound infty X12}
\sup_{t\in[0,1]} \max_{i,k}\bigl|(X_{1,t})_i^k\bigr|\leq C\,\log N \,N^{1/(\sigma-14)},
\qquad 
\sup_{t\in[0,1]} \max_{i,k}\bigl|(X_{2,t})_i^k\bigr|\leq C\,(\log N)^2 \,N^{2/(\sigma-15)},
\end{equation}
\begin{equation}\label{boi2}
\sup_{t\in[0,1]} \max_{i,i'}\bigl|(X_{1,t})_i^k(\hat\lambda)-(X_{1,t})_{i'}^k(\hat \lambda)\bigr|\le C\,\log N\,N^{1/(\sigma-15)}|\lambda_i^k-\lambda_{i'}^{k}|\qquad \forall\,k=1,\dots,d,\end{equation}
\begin{equation}\label{eq:X2lip}
\sup_{t\in[0,1]} \max_{i,i'}\bigl|(X_{2,t})_i^k(\hat\lambda)-(X_{2,t})_{i'}^k(\hat \lambda)\bigr|\le C\,(\log N)^2\,N^{2/(\sigma-17)}|\lambda_i^k-\lambda_{i'}^{k}|\qquad \forall\,k=1,\dots,d,\end{equation}
\begin{equation}\label{eq:DX1}
\sup_{t\in[0,1]}\max_{i,j}
\left|\partial_{\lambda_j^\ell}(X_{1,t})_i^k\right|(\hat\lambda)\le C\,\log N\,N^{1/(\sigma-15)}\qquad \forall\,k,\ell=1,\dots,d.
\end{equation}
\end{prop}

\begin{proof}
Since $\YY_t^N \in C^{\sigma-9}$ (see Corollary \ref{cor:regular})
it follows by Cauchy-Lipschitz theory that $X_t^N$ is of class $C^{\sigma-9}$.
Define $$
(X_t^{N,\tau})_i^k(\hat\lambda):=
X_{0,t}^k(\lambda_i^k)+
\tau\frac{(X_{1,t})_i^k}{N}(\hat\lambda)+\tau\frac{(X_{2,t})_i^{k}}{N^2}(\hat\lambda)
=(1-\tau)X_{0,t}^k(\lambda_i^k)+\tau (X_t^{N})_i^k(\hat\lambda).
$$
Also, we define the measure $M^N_{(X_t^{N,\tau})^k}$ as 
\begin{equation}
\label{eq:MN1}
\int f(y)\, dM^N_{(X_t^{N,\tau})^k}(y)=\sum_{i=1}^N \biggl[f\bigl((1-\tau)X_{0,t}^k(\lambda_i^k)+\tau(X_t^{N})_i^k(\hat\lambda)\bigr)-\int f \,d\mu_{k,t}^*\biggr]\qquad \forall\,f \in C_c(\mathbb R).
\end{equation}
In order to get an ODE for $X_{2,t}$, the strategy is to use the Taylor formula with integral rest to expand the ODE $\dot X_t^N=\YY_t^N(X_t^N)$, and then use \eqref{eq:X0} and \eqref{eq:X1} to simplify the terms involving
$\dot X_{0,t}$ and $\dot X_{1,t}$.
In this way we get  
\begin{equation}
\label{eq:ODE X2 main}
\begin{split}
(\dot X_{2,t})_i^{k}(\hat \lambda)&= 
\int_0^1 (\yy_{k,t}^0)'\Bigl((X_{t}^{N,\tau})_i^k(\hat\lambda) \Bigr)\,d\tau\cdot
(X_{2,t})_i^k(\hat\lambda)\\
&+N\int_0^1\Bigl[(\yy_{k,t}^0)'\Bigl((X_{t}^{N,\tau})_i^k(\hat\lambda)\Bigr)
-(\yy_{k,t}^0)'\Bigl(X_{0,t}^k(\lambda_i^k) \Bigr)\Bigr]\,d\tau
\cdot (X_{1,t})_i^k(\hat\lambda)\\
&+ \int_0^1 (\yy_{k,t}^1)'\Bigl((X_{t}^{N,\tau})_i^k(\hat\lambda) \Bigr)
\,d\tau\cdot \Bigl((X_{1,t})_i^k(\hat\lambda)+\frac{(X_{2,t})_i^k(\hat\lambda)}{N} \Bigr)\\
&+\int_0^1\sum_\ell \bigg[
\int \partial_1\zz_{k\ell,t}\Bigl((X_{t}^{N,\tau})_i^k(\hat\lambda),y \Bigr)\,dM^N_{(X_t^{N,\tau})^\ell}(y)\\
&\qquad    -
\int \partial_1\zz_{k\ell,t}\Bigl(X_{0,t}^k(\lambda_i^k),y \Bigr)\,dM^N_{X_{0,t}^\ell}(y)\bigg]\,d\tau\cdot \Bigl((X_{1,t})_i^k(\hat\lambda)+\frac{(X_{2,t})_i^k(\hat\lambda)}{N} \Bigr)\\
&+ \sum_\ell\int \partial_1\zz_{k\ell,t}\Bigl(X_{0,t}^k(\lambda_i^k),y \Bigr)\,dM^N_{X_{0,t}^\ell}(y)\cdot \Bigl((X_{1,t})_i^k(\hat\lambda)+\frac{(X_{2,t})_i^k(\hat\lambda)}{N} \Bigr)\\
&+\sum_\ell\sum_{j=1}^N\int_0^1 \bigg[\partial_2\zz_{k\ell,t}\Bigl((X_{t}^{N,\tau})_i^k(\hat\lambda),(X_t^{N,\tau})_j^\ell(\hat\lambda)\Bigr)\\
&\qquad\qquad \qquad\qquad-\partial_2\zz_{k\ell,t}\Bigl(X_{0,t}^k(\lambda_i^k),X_{0,t}^\ell(\lambda_j^\ell)\Bigr)\bigg]\,d\tau\cdot (X_{1,t})_j^\ell(\hat \lambda)\\
&+\sum_\ell\sum_{j=1}^N\int_0^1 \biggl[\partial_2\zz_{k\ell,t}\Bigl((X_{t}^{N,\tau})_i^k(\hat\lambda),(X_t^{N,\tau})_j^\ell(\hat\lambda)\Bigr)\biggr]\,d\tau\cdot \frac{(X_{2,t})_j^\ell(\hat\lambda)}{N},
\end{split}
\end{equation}
with the initial condition $(X_{2,t})_i^{k}=0$.
Using that 
$$
\|\yy_{k,t}^0\|_{C^{\sigma-3}(\R)}\leq C
$$
(see Corollary \ref{cor:regular})
we obtain \begin{equation}
\label{eq:bound X0}
\|X_{0,t}\|_{C^{\sigma-4}(\R)}\leq C.
\end{equation}
We now start to control $(X_{1,t})_i^k$.
First, simply by using that $M_N$ has mass bounded by $2N$ we obtain the rough bound $|(X_{1,t})_i^k|\leq C\,N$.
Inserting this bound into \eqref{eq:ODE X2 main} one easily obtains
$|(X_{2,t})_i^k|\leq C\,N^2$.

We now prove finer estimates.
First, by Lemma \ref{lem:ub4} together with the fact that $(X_{0,t})_i^k$ and $y\mapsto \zz_{k\ell,t}(x,y)$ are of class $C^{\sigma-6}$ uniformly in $x$ and $t$ (see Corollary \ref{cor:regular}),
it follows that there exists a finite constant $C$ such that, with  probability greater than $1- e^{-c(\log N)^2}$,
$$
\sup_{x \in \R,\,t\in [0,1]}\biggl| \int \zz_{k\ell,t}(x, \lambda)\,dM^N_{X_{0,t}^\ell}(\lambda)\biggr|\leq C\,
\log N \,N^{1/(\sigma-14)}.
$$
Hence, using  \eqref{eq:X1} we easily deduce the first bound in \eqref{eq:bound infty X12}.

In order to control $X_{2,t}$ we first estimate $(X_{1,t})_i^k$ in $L^4(Q^{N,aV}_0)$:
using \eqref{eq:X1} again, we get
\begin{multline}
\label{eq:ODE X1}
\frac{d}{dt}\max_{i,k}\|(X_{1,t})_i^k\|_{L^4(Q^{N,aV}_0)} \leq C\bigg(
\max_{i,k}\|(X_{1,t})_i^k\|_{L^4(Q^{N,aV}_0)}+1\\
+  \max_{i,k,\ell}\biggl\|\int\zz_{k\ell,t}\Bigl(X_{0,t}^k(\lambda_i^k), y\Bigr)\,dM^N_{X_{0,t}^\ell}(y)\biggr\|_{L^4(Q^{N,aV}_0)}\bigg).
\end{multline}
To bound $(X_{1,t})_i^k$ in $L^4(Q^{N,aV}_0)$, and then to be able to estimate $X_{2,t}$, we will use the following:
\begin{lem} Assume that $s\geq 15$.
Then, for any $i=1,\ldots,N$ and $k,\ell=1,\ldots,d$,
\begin{equation}
\label{eq:bound psi2 p1}
\biggl\|\int\zz_{k\ell,t}\Bigl(X_{0,t}^k(\lambda_i^k), y\Bigr)\,dM^N_{X_{0,t}^\ell}(y)\biggr\|_{L^4(Q^{N,aV}_0)}
 \leq C \log N,
\end{equation}
\begin{equation}
\label{eq:bound psi2}
\biggl\|\int \partial_1\zz_{k\ell,t}\Bigl(X_{0,t}^k(\lambda_i^k),y \Bigr)\,dM^N_{X_{0,t}^\ell}(y)\biggr\|_{L^4(Q^{N,aV}_0)} \leq C\log N.
\end{equation}
\end{lem}\begin{proof}
Fix indices $i,k,\ell$ and
write the Fourier decomposition of 
$$
\eta_{2,t}(x,y):=\zz_{k\ell,t}\Bigl(X_{0,t}^k(x), X_{0,t}^\ell(y)\Bigr)
$$
to get 
$$\int  \eta_{2,t}(x,y)\, d M^{N}_\ell(y)=\int \hat  \eta_{2,t}(x,\xi)  \int e^{i\xi y} \,dM^{N}_\ell(y) \,d\xi\,.$$
Since $\zz_{k\ell,t} \in C^{u,v}$ for $u,v\le \sigma-6$
 and $X_{0,t}^k\in C^{\sigma-4}$ (see \eqref{eq:bound X0}) with derivatives decaying fast at infinity, we deduce that 
$$ |  \hat\eta_{2,t}(x,\xi) |\le \frac{C}{1+ |\xi|^{\sigma-6}},$$
so, using Corollary \ref{ub3}, we get 
\begin{align*} 
  \biggl\|\sup_x\biggl|\int \eta_{2,t}(x,y) \,dM^{N}_k(y)\biggr|\biggr\|_{L^4(Q^{N,aV}_0)}
&\le \int  \Bigl\| \hat\eta_{2,t}(\cdot,\xi)\Bigr\|_{\infty}  \biggl\| \int e^{i\xi y} dM^{N}_k(y)\biggr\|_{L^4(Q^{N,aV}_0)}\,d\xi \\
&\le C\log N \int \Bigl\|\hat\eta_{2,t}(\cdot,\xi)\Bigr\|_{\infty} \left(1+|\xi|^7\right)\,d\xi\\
& \le C\log N
\end{align*}
provided $\sigma>13$.
 The same argument works
for  $\partial_1\zz_{k\ell,t}$ provided $\sigma>15,$ which  concludes the proof.
\end{proof}
 Inserting \eqref{eq:bound psi2 p1} into \eqref{eq:ODE X1}, we obtain the validity of
the first bound in \eqref{boi1}.

We now bound the time derivative of $X_{2,t}$:
using that $M_N$ has mass bounded by $2N$,
in \eqref{eq:ODE X2 main} we can easily estimate
\begin{multline*}
\biggl| N\int_0^1\Bigl[(\yy_{k,t}^0)'\Bigl((X_{t}^{N,\tau})_i^k(\hat\lambda)\Bigr)
-(\yy_{k,t}^0)'\Bigl(X_{0,t}^k(\lambda_i^k) \Bigr)\Bigr]\,d\tau
\cdot (X_{1,t})_i^k(\hat\lambda)\biggr|\\
\leq C|(X_{1,t})_i^k|^2 + \frac{C}N|(X_{1,t})_i^k|\,|(X_{2,t})_i^k|,
\end{multline*}
\begin{multline*}
\int_0^1 \bigg|
\int \partial_1\zz_{k\ell,t}\Bigl((X_{t}^{N,\tau})_i^k(\hat\lambda),y \Bigr)\,dM^N_{(X_t^{N,\tau})^\ell}(y)-
\int \partial_1\zz_{k\ell,t}\Bigl(X_{0,t}^k(\lambda_i^k),y \Bigr)\,dM^N_{X_{0,t}^\ell}(y)\bigg|\,d\tau\\
\leq C|(X_{1,t})_i^k| + \frac{C}N|(X_{2,t})_i^k|+\frac{C}N \sum_j \biggl(|(X_{1,t})_j^\ell| + \frac{1}N|(X_{2,t})_j^\ell|\biggr),
\end{multline*}
and
\begin{multline*}
\sum_{j=1}^N\int_0^1 \bigg|\partial_2\zz_{k\ell,t}\Bigl((X_{t}^{N,\tau})_i^k(\hat\lambda),(X_t^{N,\tau})_j^\ell(\hat\lambda)\Bigr)
-\partial_2\zz_{k\ell,t}\Bigl(X_{0,t}^k(\lambda_i^k),X_{0,t}^\ell(\lambda_j^\ell)\Bigr)\bigg|\,d\tau\,|(X_{1,t})_j^\ell|\\
\leq\frac{C}N\biggl(|(X_{1,t})_i^k| + \frac{1}N|(X_{2,t})_i^k|\biggr) \sum_j |(X_{1,t})_j^\ell|+ \frac{C}N \sum_j \biggl(|(X_{1,t})_j^\ell|^2 + \frac{1}N|(X_{2,t})_j^\ell|\,|(X_{1,t})_j^\ell|\biggr),
\end{multline*}
hence, noticing that $\frac{d}{dt}|(X_{2,t})_i^k|\leq |(\dot X_{2,t})_i^k|$, we get
\begin{align*}
\frac{d}{dt}|(X_{2,t})_i^k|
& \leq C\, |(X_{2,t})_i^k| + 
C \,|(X_{1,t})_i^k|^2+\frac{C}{N} \,|(X_{1,t})_i^k| |(X_{2,t})_i^k|
+C \,|(X_{1,t})_i^k| +\frac{C}{N^2}\,|(X_{2,t})_i^k|^2\\
& + \frac{C}N\,\sum_{\ell,j} |(X_{1,t})_j^\ell|\,|(X_{1,t})_i^k|
+\frac{C}{N^3}\,\sum_{\ell,j}|(X_{1,t})_i^k||\,|(X_{2,t})_j^\ell|
+\frac{C}{N^3}\,\sum_{\ell,j}|(X_{2,t})_i^k|\,|(X_{2,t})_j^\ell|\\
&+
 \biggl|\int \partial_1\zz_{k\ell,t}\Bigl(X_{0,t}^k(\lambda_i^k),y \Bigr)\,dM^N_{X_{0,t}^\ell}(y)\biggr|\,|(X_{1,t})_i^k|
 + \frac{C}N\,\sum_{\ell,j} |(X_{1,t})_j^\ell|^2\\
 &
 +\frac{C}{N^2}\,\sum_{\ell,j}|(X_{2,t})_j^\ell|\,|(X_{1,t})_j^\ell|
+ \frac{C}{N^2}\,\sum_{\ell,j} |(X_{1,t})_j^\ell|\,|(X_{2,t})_i^k|+\frac{C}N \,\sum_{\ell,j}|(X_{2,t})_j^\ell|.
\end{align*}
Using the trivial bounds 
$|(X_{1,t})_i^k|\leq C\,N$ and
$|(X_{2,t})_i^k|\leq C\,N^2$,
and the elementary inequality $ab\leq a^2+b^2$,
we obtain
\begin{multline}
 \label{eq:ODEX2 1}
\frac{d}{dt}|(X_{2,t})_i^k|
 \leq C\,\bigg( |(X_{2,t})_i^k|+|(X_{1,t})_i^k|^2
 + \frac1N \sum_{\ell,j}|(X_{1,t})_j^\ell|^2 +\frac1{N} \sum_{\ell,j}|(X_{2,t})_j^\ell|
\\+ \biggl|\int \partial_1\zz_{k\ell,t}\Bigl(X_{0,t}^k(\lambda_i^k),y \Bigr)\,dM^N_{X_{0,t}^\ell}(y)\biggr|^2\bigg).
\end{multline}
In particular, if we set $A_{1,t}:=\max_{i,k}|(X_{1,t})_i^k|$
and $A_{2,t}:=\max_{i,k}|(X_{2,t})_i^k|$ we obtain
\begin{equation}
 \label{eq:ODEX2 2}
\frac{d}{dt} A_{2,t} \leq C\biggl(A_{2,t}+(A_{1,t})^2+\max_{i,k} \biggl|\int \partial_1\zz_{k\ell,t}\Bigl(X_{0,t}^k(\lambda_i^k),y \Bigr)\,dM^N_{X_{0,t}^\ell}(y)\biggr|^2 \biggr).
\end{equation}
Hence, noticing that
\begin{equation}
\label{eq:bounds X1 2}
\sup_{x \in \R,\,t\in [0,1]}\biggl| \int \partial_1\zz_{k\ell,t}(x, \lambda)\,dM^N_{X_{0,t}^\ell}(\lambda)\biggr|\leq C\,
\log N\, N^{1/(\sigma-15)}
\end{equation}
with probability greater than $1- e^{-c(\log N)^2}$ (see Corollary \ref{cor:regular} and Lemma \ref{lem:ub4})
and recalling the first bound in \eqref{eq:bound infty X12}, using \eqref{eq:ODEX2 2}
and a Gronwall argument we deduce the validity also of the second bound in \eqref{eq:bound infty X12}.

Going back to \eqref{eq:ODEX2 1} and again the inequality $ab \leq a^2+b^2$, we also see that 
\begin{multline*}
\frac{d}{dt}\|(X_{2,t})_i^k\|_{L^2(Q^{N,aV}_0)}^2 \leq 
C\bigg( \|(X_{2,t})_i^k\|_{L^2(Q^{N,aV}_0)}^2 
 + \|(X_{1,t})_i^k\|_{L^4(Q^{N,aV}_0)}^4  +\frac1N\sum_{\ell,j} \|(X_{1,t})_j^\ell\|_{L^4(Q^{N,aV}_0)}^4 \\
 +\frac1{N} \sum_{\ell,j} \|(X_{2,t})_j^\ell\|_{L^2(Q^{N,aV}_0)}^2
 +\biggl\|\int \partial_1\zz_{k\ell,t}\Bigl(X_{0,t}^k(\lambda_i^k),y \Bigr)\,dM^N_{X_{0,t}^\ell}(y)\biggr\|_{L^4(Q^{N,aV}_0)}^4\biggr).
\end{multline*}
Hence, recalling the first bound in \eqref{boi1} and \eqref{eq:bound psi2}, we get
$$
\frac{d}{dt}\|(X_{2,t})_i^k\|_{L^2(Q^{N,aV}_0)}^2 \leq C\biggl(\|(X_{2,t})_i^k\|_{L^2(Q^{N,aV}_0)}^2 +(\log N)^4 \biggr),
$$
so a Gronwall argument concludes the proof of \eqref{boi1}.

We now prove \eqref{boi2}: recalling \eqref{eq:X1} we have 
\begin{align*}
&| (\dot X_{1,t})_i^k(\hat \lambda)-(\dot X_{1,t})_{i'}^{k}(\hat \lambda)|\\
&\leq
\bigl|(\yy_{k,t}^0)'(X_{0,t}^k(\lambda_i^k))-(\yy_{k,t}^0)'(X_{0,t}^{k}(\lambda_{i'}^{k})) \bigr|\,|(X_{1,t})_i^k(\hat\lambda)|\\
&+\bigl|(\yy_{k,t}^0)'(X_{0,t}^{k}(\lambda_{i'}^{k})) \bigr|\,|(X_{1,t})_i^k(\hat\lambda)-X_{1,t}^{N,k'}(\lambda_{i'}^{k})|
+ \bigl|\yy_{k,t}^1(X_{0,t}^k(\lambda_i^k))-\yy_{k,t}^1(X_{0,t}^{k}(\lambda_{i'}^{k}))\bigr|\\
&+\sum_\ell\biggl|\int \Bigl(\zz_{k\ell,t}\Bigl(X_{0,t}^k(\lambda_i^k), y\Bigr)-
\zz_{k\ell,t}(X_{0,t}^{k}(\lambda_{i'}^{k}), y)\Bigr)\,dM^N_{X_{0,t}^\ell}(y)\biggr|\\
&+\frac1N\sum_{\ell,j} \Bigl|\partial_2\zz_{k\ell,t}\Bigl(X_{0,t}^k(\lambda_i^k),X_{0,t}^\ell(\lambda_j^\ell)\Bigr) -\partial_2\zz_{k\ell,t}\Bigl(X_{0,t}^{k}(\lambda_{i'}^{k}),X_{0,t}^\ell(\lambda_j^\ell)\Bigr) \Bigr|\,|(X_{1,t})_j^\ell(\lambda_j^{\ell})|.
\end{align*}
Hence, using that $|X_{0,t}^k(\lambda_i^k)-X_{0,t}^{k}(\lambda_{i'}^{k})| \leq C|\lambda_i^k - \lambda_{i'}^{k}|$,
the bounds \eqref{eq:bound infty X12}
 and \eqref{eq:bounds X1 2},
and the Lipschitz regularity of $(\yy_{k,t}^0)'$, $\yy_{k,t}^1$, $\zz_{k\ell,t}$,
and $\partial_2\zz_{k\ell,t}$,
we get
$$
|(\dot X_{1,t})_i^k(\hat \lambda)-(\dot X_{1,t})_{i'}^{k}(\hat \lambda)|
\leq C|(X_{1,t})_i^k(\hat\lambda)- (X_{1,t})_{i'}^{k}(\hat \lambda)|+C\,\log N\, N^{1/(\sigma-15)}|\lambda_i^k - \lambda_{i'}^{k}|
$$
outside a set of probability less than $e^{-c(\log N)^2}$,
so \eqref{boi2} follows from Gronwall's inequality.

By a completely analogous argument, 
it follows from \eqref{eq:ODE X2 main}, \eqref{boi2},  \eqref{eq:bound infty X12}, and estimates analogue to
\eqref{eq:bounds X1 2} for the higher derivatives of $\zz_{k\ell,t}$, that
$$
|(\dot X_{2,t})_i^k(\hat \lambda)-(\dot X_{2,t})_{i'}^{k}(\hat \lambda)|
\leq C|(X_{2,t})_i^k(\hat\lambda)- (X_{2,t})_{i'}^{k}(\hat \lambda)|+C\,(\log N)^2\, N^{2/(\sigma-17)}|\lambda_i^k - \lambda_{i'}^{k}|
$$
holds outside a set of probability less than $e^{-c(\log N)^2}$.
Thus \eqref{eq:X2lip} follows.

Finally, denoting by $\delta_{j}^\ell$ the vector with zero entries except at position $j,\ell$ where there is a one
(so that $\hat \lambda+\epsilon \delta_{j}^\ell=(\lambda_1^1,\ldots, \lambda_j^\ell +\epsilon,\ldots \lambda_N^d)$),
one can differentiate in time $| (X_{1,t})_i^k(\hat \lambda+\epsilon \delta_{j}^\ell)-(X_{1,t})_{i'}^{k}(\hat \lambda)|$
and argue as above to deduce that
$$
|(X_{1,t})_i^k(\hat \lambda+\epsilon \delta_{j}^\ell)-(X_{1,t})_{i'}^{k}(\hat \lambda)| \leq C\,\log N\,N^{1/(\sigma-15)}\,\epsilon
$$
outside a set of probability less than $e^{-c(\log N)^2}$.
Dividing by $\epsilon$ and letting $\epsilon \to 0$, this proves \eqref{eq:DX1}.
\end{proof}

\section{Universality results}
\label{sect:extensions}
In this section we explain how Corollaries  \ref{thm:univ}, \ref{thm:univ2} and  \ref{univmax} 
follow from our Theorem~\ref{thm:transport}.

\begin{proof}[Proof of Corollary \ref{thm:univ}]
Given $\vartheta>0$, we define the set
\begin{equation}\label{eq:Gtheta}
G_{\vartheta}:=\Bigl\{ \hat\lambda \in \R^{dN} \, :\, 
|\lambda_{i}^\ell-\gamma_{i/N}^\ell|\le N^{\vartheta-2/3}\min\{ i,N+1-i\}^{1/3}\quad\forall\,i,\ell\Bigr\}.
\end{equation}
As proved in \cite{EYY} in the special case of the Gaussian ensembles
and then generalized in \cite[Theorem 2.4]{BEY1} to  potentials $W_k$ satisfying much weaker conditions
than the ones assumed here,
the following rigidity estimate holds:
for all $\vartheta>0$ there exist $\bar c>0$ and $\bar C<\infty$  such that for all $N\ge 0$
\begin{equation}\label{rigidity}
\tilde P^{N,0}_{\beta}\bigl( \R^N\setminus G_\vartheta \bigr)\le \bar C e^{-
N^{\bar c}}\,.
\end{equation}
Also, thanks to the fact that $\mu_k^0$ has a density which is strictly positive inside its support $[a_k^0,b_k^0]$
except at the two boundary points where it goes to zero as a square root (see Lemma \ref{lem:support}), we deduce that
$$
\frac{m}{N} \geq \frac{1}{C} \int_{\gamma_{i/N}^k}^{\gamma_{(i+m)/N}^k}\min\Bigl\{\sqrt{s-a_k^0},\sqrt{b_k^0-s}\Bigr\}\,ds,
$$
from which it follows easily that
\begin{equation}\label{rigidity gamma}
\bigl|\gamma_{(i+m)/N}^k- \gamma_{i/N}^k\bigr| \leq \frac{C}{N^{2/3}}\min\biggl\{ m^{2/3},\frac{m}{\min\{ i,N+1-i\}^{1/3}}\biggr\}\,.
\end{equation}
Since
\begin{equation}
\label{eq:spacing}
|\lambda_{i+m}^k-\lambda_i^k| \leq |\lambda_i^k -\gamma_{i/N}^k|+|\lambda_{i+m}^k -\gamma_{(i+m)/N}^k|+\bigl|\gamma_{(i+m)/N}^k- \gamma_{i/N}^k\bigr|,
\end{equation}
using \eqref{rigidity} and \eqref{rigidity gamma} and recalling that by assumption $m \ll N$,
we deduce that
\begin{equation}
\label{eq:Gtheta bulk}
|N(\lambda_{i_k+j}^k-\lambda_{{i_k}}^k)|\leq 
C\left(N^\vartheta + m\right)\qquad \forall\,\hat\lambda \in G_\vartheta,\quad {i_k}\in [N\epsilon, N(1-\epsilon)],\,j=1,\ldots,m,
\end{equation}
and
\begin{equation}
\label{eq:Gtheta edge}
|N^{2/3}(\lambda_{j}^k-a_k^0)|\leq 
C\left(N^\vartheta + m^{2/3}\right)\qquad \forall\,\hat\lambda \in G_\vartheta,\quad j=1,\ldots,m.
\end{equation}

Now, given a bounded function $\chi:\R^{dN}\to \R$,
applying \eqref{aptr} to $\frac{1}{2}(1+\frac{\chi}{\|\chi\|_\infty})$ with $k=0$ and $\eta=\vartheta$,
we deduce that
\begin{equation}
\label{eq:aptr2}
\left|\int \chi\circ T^N\, d P^{N,0}_\beta-\int \chi\, d P^{N,aV}_\beta\right|\le C\,N^{\vartheta-1} \|\chi\|_\infty.
\end{equation}
Recall that the map $T^N$ is given by $X^N_1,$ where $X^N_t$ is the flow of the vector-field $\YY_t^N$ that has the very special
form \eqref{eq:psi} (see Proposition \ref{theflow}).
In particular, since the functions $\yy_{k,t}^0$, $\yy_{k,t}^1$, $\zeta_{k\ell,t}(\cdot,y)$ 
are uniformly Lipschitz, we see that
$$
|(\dot X^N_t)_i^k -(\dot X^N_t)_j^k| \leq L\,|(X^N_t)_i^k -(X^N_t)_j^k|\qquad \forall\,i,j=1,\ldots,N,\,k=1,\ldots,d.
$$
Hence, since $X_1^N=T^N$ and $X_0^N=\operatorname{Id}$, Gronwall's inequality yields
\begin{equation}
\label{eq:TN}
e^{-L}(\lambda_i^k-\lambda_j^k) \leq (T^N)_i^k(\hat \lambda) -(T^N)_j^k(\hat \lambda) \leq e^L(\lambda_i^k-\lambda_j^k)
\qquad \forall\,\lambda_i^k \geq \lambda_j^k.
\end{equation}
We now remark that the law $\tilde P^{N,aV}_{\beta}$ is obtained as the image of the law of $\lambda^k
=(\lambda^k_1,\ldots,\lambda^k_N), 1\le k\le d$ under $P^{N,aV}_\beta$ under the map
\begin{equation}
\label{eq:Rk}
\hat\RR:\R^{dN}\to \R^{dN}, \qquad \hat\RR(\lambda^1,\ldots,\lambda^k,\ldots,\lambda^d):=
\bigl(\RR(\lambda^1),\ldots,\RR(\lambda^k),\ldots,\RR(\lambda^d)\bigr),
\end{equation}
where $\RR:\R^N\to \R^N$ is defined as
\begin{equation}
\label{eq:R}
[\RR(x_1,\ldots,x_N)]_i:=\min_{\#J=i}\max_{j \in J}x_j\qquad \forall\,i=1,\ldots,N.
\end{equation}
Hence, thanks to \eqref{eq:TN},
it follows that $T^N$ and $\hat \RR$ commute, namely
\begin{equation}
\label{eq:TNR}
\hat \RR\circ T^N=T^N\circ \hat \RR.
\end{equation}
We now consider a test function $\chi$ of the form
\begin{equation}
\label{eq:chi 1}
\chi(\hat\lambda)=f\Bigl(\bigl(N(\lambda_{i_{k}+1}^k-\lambda_{i_{k}}^k),\ldots,N(\lambda_{i_k+m}^k-\lambda_{i_k}^k)\bigr)_{1\le k\le d}\Bigr).
\end{equation}
Then
$$
\int f\Bigl(\bigl(N(\lambda_{i_{k}+1}^k-\lambda_{i_{k}}^k),\ldots,N(\lambda_{i_k+m}^k-\lambda_{i_k}^k)\bigr)_{1\le k\le d}\Bigr)\,
d \tilde P^{N,aV}_{\beta}
=\int \chi\circ \hat \RR\,
d P^{N,aV}_\beta,
$$
and it follows by \eqref{eq:aptr2} and \eqref{eq:TNR}
that
$$
\biggl|\int \chi\,d \tilde P^{N,aV}_{\beta} - \int \chi\circ T^N\circ \RR\, d P^{N,0}_\beta
\biggl| \le C\,N^{\vartheta-1} \|f\|_\infty.
$$
Let $X_{0,t}$, $X_{1,t},$ and $X_{2,t}$ be as in Proposition \ref{theflow}, and note the following fact:
whenever $\hat\lambda \in G_{\vartheta}$ we know that,
for any $\ell=1,\ldots,d$, the numbers $\{\lambda_i^\ell\}_{1\leq i \leq N}$
are close, up to an error $N^\vartheta$, to the quantiles of the stationary measure $\mu_\ell^0=\mu_{\ell,0}^*$.
Hence, given any $1$-Lipschitz function $\psi$,
$$
\biggl|\int \psi \,dM^N_\ell\biggr| \leq C\,N^\vartheta\qquad \forall\,\ell=1,\ldots,d.
$$
Since $X_{0,t}^\ell$ is a smooth diffeomorphism which sends the quantiles of $\mu_{\ell,0}^*$
onto the quantiles of $\mu_{\ell,t}^*$, we deduce that 
$$
\biggl|\int \psi \,dM^N_{X_{0,t}^\ell}\biggr| \leq C\,N^\vartheta\qquad \forall\,\ell=1,\ldots,d,\,\forall\,t \in [0,1].
$$
This implies that
$$
\sup_{x,t}\int \zz_{k\ell,t}(x, y)\,dM^N_{X_{0,t}^\ell}(y)=O\bigl(N^{\vartheta}\bigr),
\qquad
\sup_{x,t} \int \partial_1\zz_{k\ell,t}(x, \lambda)\,dM^N_{X_{0,t}^\ell}(\lambda)=O\bigl(N^{\vartheta}\bigr),
$$
and by the same argument as the one used in the proof of Proposition \ref{theflow} to show \eqref{eq:bound infty X12} and \eqref{boi2}
we get
\begin{equation}
\label{eq:X1theta}
\max_{i,k}\bigl|(X_{1,1})_{i}^k(\hat \lambda)\bigr|\leq C\,N^\vartheta,\qquad \bigl|(X_{1,1})_{i}^k(\hat \lambda)-(X_{1,1})_{i'}^k(\hat \lambda)\bigr|\leq C\,N^\vartheta\,|\lambda_i^k-\lambda_{i'}^k|
\qquad \forall\,\hat\lambda \in G_{\vartheta}.
\end{equation}
Then, noticing that  $\|\nabla \chi \|_\infty \leq N\, \|\nabla f \|_\infty $,
thanks to \eqref{eq:X1theta}, \eqref{eq:X2lip}, and \eqref{eq:Gtheta bulk},
we get 
\begin{align*}
&\biggl|\int_{G_\vartheta} \chi \circ T^N\circ \hat \RR\, d P^{N,0}_\beta - \int_{G_\vartheta} \chi \circ X_{0,1}\circ \hat \RR\, d P^{N,0}_\beta\biggr|\\
&\leq \,\|\nabla \chi \|_\infty 
\int_{G_\vartheta}\biggl[\sum_{k=1}^d\sum_{j=1}^m \biggl(\frac{|(X_{1,1})_{i_k+j}^k - (X_{1,1})_{i_k}^k|}{N}+\frac{|(X_{2,1})_{i_k+j}^k - (X_{2,1})_{i_k}^k|}{N^2}\biggr)^2\biggr]^{1/2}\, d \tilde P^{N,0}_\beta\\
&\leq C\,\|\nabla f\|_\infty \,N^\vartheta\int_{G_\vartheta} \biggl(\sum_{k=1}^d\sum_{j=1}^m|\lambda_{i_k+j}^k - \lambda_{i_k}^k|^2\biggr)^{1/2}\, d \tilde P^{N,0}_\beta\leq  C\,\|\nabla f\|_\infty\,\frac{m^{1/2} \,N^{\vartheta}\,(N^\vartheta+m)}{N}.
\end{align*}
Note now that $(X_{0,1})_i^k=T_0^k$ for all $i=1,\ldots,N$,
and that
\begin{equation}
\label{eq:T0k}
e^{-L}\leq (T_0^k)' \leq e^L
\end{equation}
(this follows by the same proof as the one of \eqref{eq:TN}, compare also with \cite[Equation (5.2)]{BFG}).
In addition
$$
(T_{0,1})_{i_k+j}^k(\hat \lambda) - (T_{0,1})_{i_k}^k(\hat\lambda)=(T_0^k)'(\lambda_{i_k}^k)\,[\lambda_{{i_k}+j}^k-\lambda_{i_k}^k]
+O\bigl(|\lambda_{{i_k}+j}^k-\lambda_{i_k}^k|^2\bigr),
$$
hence, by the definition of $G_\vartheta$,
\begin{multline*}
\int_{G_\vartheta} \chi \circ X_{0,1}\circ \hat \RR\, d P^{N,0}_\beta\\
=\int_{G_\vartheta} f\Bigl(\bigl ( (T_{0}^k)'(\lambda_{i_k}^k)\,N(\lambda_{i_k+1}^k-\lambda_{i_k}^k),\ldots,(T_{0}^k)'(\lambda_{i_k}^k)\,N(\lambda_{i_k+m}^k-\lambda_{i_k}^k)\bigr)_{1\le k\le d} \Bigr)\, d\tilde P^{N,0}_{\beta}\\
+O\Bigl(\|\nabla f\|_\infty\,m^{1/2}\,(N^{\vartheta}+m)^2\,N^{-1}\Bigr).
\end{multline*}
Also, in the integral above we can replace 
$(T_{0}^k)'(\lambda_{i_k}^k)$ with
$(T_{0}^k)'(\gamma_{i_k/N}^k)$,
up to an error bounded by
$$
C\,\|\nabla f\|_\infty  \int_{G_\vartheta}
\biggl(\sum_{k=1}^d\sum_{j=1}^m|\lambda_{i_k}^k -\gamma_{i/N}^k|^2 \,(N|\lambda_{i_k+j}^k-\lambda_{i_k}^k|)^2\biggr)^{1/2}\,d\tilde P^{N,0}_{\beta}
=O\Bigl(\|\nabla f\|_\infty\,m^{1/2}\,(N^{\vartheta}+m)\,N^{\vartheta-1}\Bigr).
$$
Finally, 
it follows by \eqref{rigidity} that
all integrals on ${\R^N\setminus G_\vartheta}$ are bounded by $C\,\|f\|_\infty \,e^{-N^{\bar c}}$.
Hence, we proved that\footnote{This estimate, as well as the one at the edge that we shall prove below, should be compared with the one obtained in
 \cite[Theorem 1.5]{BFG}. While the estimates here are considerably stronger that the ones in \cite[Theorem 1.5]{BFG} (this follows from the fact that we have better bounds on our
 approximate transport maps), as a small ``loss'' we now have $N^{\vartheta-1}$
 instead of a term $(\log N)^3/N$. The reason for this small difference comes from the fact that we decided to apply \eqref{aptr} to deduce \eqref{eq:aptr2}. It is worth noticing that the argument in Section \ref{transport}
 combined with \cite[Lemma 2.2]{BFG} proves that also the stronger bound
$$
\left|\int \chi\circ T^N\, d P^{N,0}_\beta-\int \chi\, d P^{N,aV}_\beta\right|\le C\,\frac{(\log N)^3}{N} \|\chi\|_\infty
$$
holds. However, since in general \eqref{aptr} is much more powerful than the estimate above (as it allows to deal with 
functions that grow polynomially with respect to the dimension) and the improvement between $(\log N)^3/N$ and $N^{\vartheta-1}$ is minimal, we have decided not to state also this second estimate.
}

\begin{multline*}
\bigg|\int f\Bigl(\bigl(N(\lambda_{i_{k}+1}^k-\lambda_{i_{k}}^k),\ldots,N(\lambda_{i_k+m}^k-\lambda_{i_k}^k)\bigr)_{1\le k\le d}\Bigr)\, d\tilde P^{N,aV}_{\beta}\\
\qquad -\int f\Bigl(\bigl((T_{0}^k)'(\gamma_{i_k/N}^k)\,N(\lambda_{i_k+1}^k-\lambda_{i_k}^k),\ldots,(T_{0}^k)'(\gamma_{i_k/N}^k)\,N(\lambda_{i_k+m}^k-\lambda_{i_k}^k)\bigr)_{1\le k\le d})\Bigr) \,d\tilde P^{N,0}_{\beta}\bigg|\\
\leq \hat C \,\Big(N^{\vartheta-1} +e^{-N^{\bar c}}\Big)
\|f\|_\infty + \hat C\,\frac{m^{1/2} \,N^{2\vartheta}+m^{3/2}\,N^\vartheta}{N}\,
\|\nabla f\|_{\infty}
\end{multline*}
Since
$e^{-N^{\bar c}} \leq C\,N^{\theta-1}$, choosing $\vartheta \leq \theta/2$
we conclude the validity of the first statement.

For the second statement we choose 
$\chi(\hat\lambda)=f\Bigl(\bigl(N^{2/3}(\lambda_{1}^k-a_k^{aV}),\ldots,N(\lambda_{m}^k-a_k^{aV})\bigr)_{1\le k\le d}\Bigr)$ 
and we note that $T_{0}^k(a_k^0)=a_k^{aV}$.
Then, thanks to \eqref{eq:bound infty X12} and  \eqref{eq:X1theta},
we get
\begin{align*}
&\biggl|\int_{G_\vartheta} \chi \circ T^N\circ \hat \RR\, d P^{N,0}_\beta - \int_{G_\vartheta} \chi \circ X_{0,1}\circ \hat \RR\, d P^{N,0}_\beta\biggr|\\
&\leq \frac{\|\nabla f \|_\infty }{N^{1/3}}
\int_{G_\vartheta} \biggl[\sum_{k=1}^d\sum_{j=1}^m\biggl(|(X_{1,1})_{j}^k|+\frac{|(X_{2,1})_{j}^k|}{N}\biggr)^2\circ \hat \RR\biggr]^{1/2}
\, d P^{N,0}_\beta\\
&\leq \frac{\|\nabla f \|_\infty }{N^{1/3}}\,(d\,m)^{1/2} 
\int_{G_\vartheta} \biggl(\max_{i,k}|(X_{1,1})_{i}^k|+\frac{\max_{i,k}|(X_{2,1})_{i}^k|}{N}\biggr)\, d P^{N,0}_\beta\leq  C\,\|\nabla f\|_\infty\,\frac{m^{1/2} \,N^{\vartheta}}{N^{1/3}}.
\end{align*}
Also,
since
$$
T_0^k(\lambda_1^k)-T_0^k(a_{k}^{0})=
(T_{0}^k)'(a_{k}^{0})\,[\lambda_1^k-a_{k}^{0}]+O\bigl(|(\lambda_1^k-a_{k}^{0}|^2\bigr),
$$
using the rigidity estimate \eqref{eq:Gtheta edge}, we can replace 
$N^{2/3}\bigl(T_0^k(\lambda_1^k)-T_0^k(a_{k}^{0})\bigr)$
with $(T_{0}^k)'(a_{k}^{0})\,N^{2/3}(\lambda_1^k-a_{k}^{0})$ up to an error 
of size $m^{1/2}\,(N^\vartheta+m^{2/3})\,N^{-2/3}$.
Hence, arguing as above we conclude that 
\begin{multline*}
\bigg|
\int f\Bigl(\bigl(N^{2/3}(\lambda_1^k-a_{k}^{aV}), \ldots,N^{2/3}(\lambda_m^k-a_{k}^{aV})\bigr)_{1\le k\le d}\Bigr) \,d\tilde P^{N,aV}_{\beta}\\
-\int f\Bigl(\bigl((T_{0}^k)'(a_{k}^{0})\,N^{2/3}(\lambda_1^k-a_{k}^{0}), \ldots,(T_{0}^k)'(a_{k}^{0})\,N^{2/3}(\lambda_m^k-a_{k}^{0})\bigr)_{1\le k\le d}\Bigr)\, d\tilde P^{N,0}_{\beta}\bigg|\\
\le \hat C \,N^{\vartheta-1} 
\|f\|_\infty+\hat C\,\biggl(\frac{m^{1/2} \,N^{\vartheta}}{N^{1/3}}+\frac{m^{1/2}\,(N^\vartheta+m^{2/3})}{N^{2/3}}\biggr)\,
\|\nabla f\|_{\infty}.
\end{multline*}
which proves the second statement choosing $\vartheta \leq \theta$.
\end{proof}

\begin{proof}[Proof of Corollary \ref{thm:univ2}]
We first note that the proof of Corollary \ref{thm:univ} could be repeated verbatim in the context of \cite{BFG}
to show that \cite[Theorem 1.5]{BFG} holds with the same estimates as we obtained here.
Hence, by combining this result with Corollary \ref{thm:univ}  
we have
\begin{multline*}
\bigg|\int f\Bigl(\bigl(N(\lambda_{i_{k}+1}^k-\lambda_{i_{k}}^k),\ldots,N(\lambda_{i_k+m}^k-\lambda_{i_k}^k)\bigr)_{1\le k\le d}\Bigr)\, d\tilde P^{N,aV}_{\beta}\\
\qquad -\int f
\Bigl(\bigl((T_0^k \circ S_0^k)'(\gamma_{i_k/N})\,N(\lambda_{i_{k}+1}^k-\lambda_{i_{k}}^k),\ldots,(T_0^k \circ S_0^k)'(\gamma_{i_k/N})\,N(\lambda_{i_k+m}^k-\lambda_{i_k}^k)\bigr)_{1\le k\le d}\Bigr) \,d (\tilde P_{{\rm GVE},\beta}^N)^{\otimes d}\bigg|\\
 \le 
  \hat C
  \,N^{\theta-1} \,
\|f\|_\infty + \hat C\,
m^{3/2}\,N^{\theta-1}\,
\|\nabla f\|_{\infty},
\end{multline*}
where $\gamma_{i_k/N}$ satisfies 
 $\mu_{\rm sc}((-\infty,\gamma_{i_k/N}))=i_k/N$.
 Then we notice that the transport relations \eqref{eq:Sk}
 and \eqref{eq:Tk0} 
 imply that $T_0^k \circ S_0^k(\gamma_{i_k/N})=\gamma_{i_k/N,a}^k$ where $\gamma_{i_k/N,a}^k$ satisfies 
 $\mu_{k}^{aV}((-\infty,\gamma_{i_k/N,a}^k))=i_k/N$, hence (again by \eqref{eq:Sk}
 and \eqref{eq:Tk0})
 $$
 (T_0^k \circ S_0^k)'(\gamma_{i_k/N})=\frac{\rho_{sc}(\gamma_{i_k/N})}{\rho_{k}^{aV}(\gamma_{i_k/N,a}^k)}.
 $$
 Finally, since $|\sigma_k-i_k/N| \leq C/N$ and $\sigma_k \in (0,1)$,
 arguing as we did for proving \eqref{rigidity gamma},
 we deduce that $|\gamma_{i_k/N}-\gamma_{\sigma_k}| \leq \tilde C/N$,
 so up to another small error we can replace $\frac{\rho_{sc}(\gamma_{i_k/N})}{\rho_{k}^{aV}(\gamma_{i_k/N,a}^k)}$ with $\frac{\rho_{sc}(\gamma_{\sigma_k})}{\rho_{k}^{aV}(\gamma_{\sigma_k,k})}$.
 This concludes the proof of of the first statement, while the second one is just 
a consequence of Corollary \ref{thm:univ}(2) and \cite[Theorem 1.5(2)]{BFG}.
\end{proof}

\begin{proof}[Proof of Corollary \ref{thm:energy}]
As it is clear by looking at the proof of 
Corollaries \ref{thm:univ} and \ref{thm:univ2}, 
the fact of dealing at the same time with the eigenvalues of different matrices does not complicate the proof.
For this reason, since the proof of Corollary \ref{thm:energy} is already very involved,
to make the argument more transparent we shall prove the result when the test function is of the form
$$
 \negint_{R_k(E)-N^{-\zeta}\,R_k'(E)}^{R_k(E)+N^{-\zeta}\,R_k'(E)}\sum_{i_1\neq \ldots\neq i_m} f\bigl(N(\lambda_{i_1}^k-\tilde E),\ldots,N(\lambda_{i_m}^k-\tilde E)\bigr)\,d\tilde E
$$
for some $E \in (-2,2),$ the proof in the general case being completely analogous and just notationally heavier.

To simplify the notation, we set 
$$
g_{\tilde E}(\hat \lambda):=\sum_{i_1\neq \ldots\neq i_m} f\bigl(N(\lambda_{i_1}^k-\tilde E),\ldots,N(\lambda_{i_m}^k-\tilde E)\bigr),
\qquad
A_{k}:=\int\biggl[ \negint_{R_k(E)-N^{-\zeta}\,R_k'(E)}^{R_k(E)+N^{-\zeta}\,R_k'(E)}g_{\tilde E}\,d\tilde E\, \biggr]\,dP^{N,aV}_{\beta}.
$$
It follows by \eqref{aptr} with $\eta=\theta$ that
\begin{equation}
\label{eq:log 1}
|\log(1+A_{k}) - \log(1+{A}_{1,k})| \leq C\,N^{\theta-1},
\end{equation}
where
\begin{align*}
{A}_{1,k}&:=\int\biggl[ \negint_{R_k(E)-N^{-\zeta}\,R_k'(E)}^{R_k(E)+N^{-\zeta}\,R_k'(E)}
g_{\tilde E}\circ (T^N)^k\,d\tilde E\, \biggr]\,d P^{N,0}_{\beta}\\
&=\int\biggl[ \negint_{R_k(E)-N^{-\zeta}\,R_k'(E)}^{R_k(E)+N^{-\zeta}\,R_k'(E)}\sum_{i_1\neq \ldots\neq i_m} f\Bigl(N\bigl((T^N)_{i_1}^k(\hat\lambda)-\tilde E\bigr),\ldots,N\bigl((T^N)_{i_m}^k(\hat\lambda)-\tilde E\bigr)\Bigr)\,d\tilde E\, \biggr]\,dP^{N,0}_{\beta}.
\end{align*}
Define the quantiles $\gamma_{i/N}^k \in \bigl(S_k^0(-2),S_k^0(2)\bigr)$ 
as in Corollary \ref{thm:univ}, and given $\vartheta>0$ small (to be fixed later) we consider the set $G_\vartheta$ defined in \eqref{eq:Gtheta}.

Since the integrand $g_{\tilde E}\circ (T^N)^k$ is pointwise bounded by $\|f\|_\infty N^m$, it follows by \eqref{rigidity}
that
\begin{equation}
\label{eq:log 2}
\begin{split}
{A}_{1,k}&=A_{2,k}+O(e^{-N^c})\\
&:=\int_{G_{\vartheta}}\biggl[ \negint_{R_k(E)-N^{-\zeta}\,R_k'(E)}^{R_k(E)+N^{-\zeta}\,R_k'(E)}
g_{\tilde E}\circ (T^N)^k\,d\tilde E\, \biggr]\,dP^{N,0}_{\beta}+O(e^{-N^{\bar c}}).
\end{split}
\end{equation}
Observe that if $\hat\lambda \in G_{\vartheta}$ then, by definition,
$$
|\lambda_i^k-\lambda_j^k| \geq |\gamma_{i/N}^k-\gamma_{j/N}^k|-N^{-2/3+\vartheta}\min\{ i,N+1-i\}^{-1/3}-N^{-2/3+\vartheta}\min\{ j,N+1-j\}^{-1/3}.
$$
Hence, since $\gamma_{(i+1)/N}^k-\gamma_{i/N}^k \geq c_0 N^{-2/3}\min\{ i,N+1-i\}^{-1/3}$ for all $i$, we deduce that
$$
|\lambda_i^k-\lambda_j^k| \geq N^{\vartheta-1}\qquad \text{provided $|i-j|\geq C_0N^\vartheta$,}
$$
that combined with \eqref{eq:TN} yields, for
$\hat\lambda \in G_{\vartheta}$,
\begin{equation}
\label{eq:lipTN}
|(T^N)_i^k(\hat \lambda) -(T^N)_j^k(\hat \lambda)| \geq e^{-L}N^{\vartheta-1}
\qquad \text{provided $|i-j|\geq C_0N^\vartheta$}.
\end{equation}
We now notice that, since $f$ is compactly supported, the
quantity 
$$
f\Bigl(N\bigl((T^N)_{i_1}^k(\hat \lambda)-\tilde E\bigr),\ldots,N\bigl((T^N)_{i_m}^k(\hat \lambda) -\tilde E\bigr)\Bigr)
$$
can be nonzero only if 
$$
|(T^N)_{i_j}^k(\hat \lambda)-\tilde E|\leq \frac{C_1}{N}\qquad \forall\,j=1,\ldots,m.
$$
Therefore, if $\bar i \in \{1,\ldots,N\}$ is an index (depending on $\hat \lambda$ and $\tilde E$)
such that
$$
|(T^N)_{\bar i}^k(\hat \lambda)-\tilde E|\leq \frac{C_1}{N},
$$
then \eqref{eq:lipTN} yields
$$
|(T^N)_{i}^k(\hat \lambda)-\tilde E|\leq \frac{C_1}{N} \quad\Rightarrow\quad |i-\bar i|\leq C_0N^\vartheta.
$$
This proves that, for any $\hat\lambda \in G_\vartheta$, there exists a set
of indices
$$
J_{\hat\lambda,\tilde E}\subset \{(i_1,\ldots,i_m) \in \{1,\ldots,N\}^m\,:\,i_1\neq \ldots\neq i_m\}$$ such that $\# J_{\hat\lambda,\tilde E} \leq CN^{m\vartheta}$
and
$$
A_{2,k}=\int_{G_\vartheta}\biggl[ \negint_{R_k(E)-N^{-\zeta}\,R_k'(E)}^{R_k(E)+N^{-\zeta}\,R_k'(E)}
\hat g_{\tilde E}\circ (T^N)^k\,d\tilde E\, \biggr]\,dP^{N,0}_{\beta},
$$
where
$$
\hat g_{\tilde E}(\hat\lambda):=\sum_{(i_1, \ldots, i_m) \in J_{\hat\lambda,\tilde E}} f\bigl(N(\lambda_{i_1}^k-\tilde E),\ldots,N(\lambda_{i_m}^k-\tilde E)\bigr)
$$
satisfies $|\hat g_{T_0^k(\tilde E)}| \leq C \|f\|_\infty N^{m\vartheta}$.

We now perform the change of variable $\tilde E\mapsto T_0^k(\tilde E)$, which gives
$$
\int_{R_k(E)-N^{-\zeta}\,R_k'(E)}^{R_k(E)+N^{-\zeta}\,R_k'(E)}
\hat g_{\tilde E}\circ (T^N)^k\,d\tilde E
=\int_{(T_0^k)^{-1}[R_k(E)-N^{-\zeta}\,R_k'(E)]}^{(T_0^k)^{-1}[R_k(E)+N^{-\zeta}\,R_k'(E)]}
\hat g_{T_0^k(\tilde E)}\circ (T^N)^k\,(T_0^k)'(\tilde E)\,d\tilde E.
$$
Recalling that $R_k=T_0^k\circ S_0^k$ and that these maps are all smooth diffeomorphisms of $\R$, we see that
for $\tilde E \in \bigl[(T_0^k)^{-1}[R_k(E)-N^{-\zeta}\,R_k'(E)],(T_0^k)^{-1}[R_k(E)+N^{-\zeta}\,R_k'(E)]\bigr]$
it holds
$$
|(T_0^k)'(\tilde E) - (T_0^k)'\circ S_0^k(E)| \leq CN^{-\zeta},
\qquad R_k'(E)=[(T_0^k)'\circ S_0^k(E)]\,(S_0^k)'(E),
$$
and
$$
(T_0^k)^{-1}[R_k(E)\pm N^{-\zeta}\,R_k'(E)]=S_0^k(E)\pm N^{-\zeta}(S_0^k)'(E)+O(N^{-2\zeta}).
$$
Hence, since $|\hat g_{T_0^k(\tilde E)}| \leq CN^{m\vartheta},$
\begin{multline*}
\negint_{(T_0^k)^{-1}[R_k(E)-N^{-\zeta}\,R_k'(E)]}^{(T_0^k)^{-1}[R_k(E)+N^{-\zeta}\,R_k'(E)]}
\hat g_{T_0^k(\tilde E)}\circ (T^N)^k\,(T_0^k)'(\tilde E)\,d\tilde E\\
=\negint_{S_0^k(E)-N^{-\zeta}(S_0^k)'(E)}^{S_0^k(E)-N^{-\zeta}(S_0^k)'(E)}
\hat g_{T_0^k(\tilde E)}\circ (T^N)^k\,d\tilde E+O(N^{m\vartheta-\zeta}),
\end{multline*}
which proves that 
\begin{equation}
\label{eq:log 3}
\begin{split}
A_{2,k}&=A_{3,k}+ O(N^{m\vartheta-\zeta})\\
&:=\int_{G_{\vartheta}}\biggl[ \negint_{S_0^k(E)-N^{-\zeta}(S_0^k)'(E)}^{S_0^k(E)-N^{-\zeta}(S_0^k)'(E)}
\hat g_{T_0^k(\tilde E)}\circ (T^N)^k\,d\tilde E\, \biggr]\,dP^{N,0}_{\beta}+ O(N^{m\vartheta-\zeta}).
\end{split}
\end{equation}
We now estimate $A_{3,k}$.

Thanks to Theorem \ref{thm:transport} we can write 
\begin{multline*}
\hat g_{T_0^k(\tilde E)}\circ (T^N)^k(\hat \lambda)=\sum_{(i_1, \ldots, i_m) \in J_{\hat\lambda,\tilde E}} f\Big(N\bigl(
T_0^k(\lambda_{i_1}^k)-T_0^k(\tilde E)\bigr) +(X_{1,1}^N)_{i_1}^k(\hat \lambda),\ldots\\
\ldots,N\bigl(T_0^k(\lambda_{i_m}^k)-T_0^k(\tilde E)\bigr)
+(X_{1,1}^N)_{i_m}^k(\hat \lambda)\Big)\\+ O\biggl(\frac{\|\nabla f\|_\infty}{N} \,\sum_{(i_1, \ldots, i_m) \in J_{\hat\lambda,\tilde E}}|(X_{2,1}^N)_{i_j}^k|\biggr).
\end{multline*}
thus
\begin{equation}
\label{eq:log 4}
\begin{split}
A_{3,k}&=A_{4,k}+O\biggl(\frac{1}{N} \,\int_{G_\vartheta} \sum_{(i_1, \ldots, i_m) \in J_{\hat\lambda,\tilde E}}|(X_{2,1}^N)_{i_j}^k|\,dP_{\beta}^{N,0}\biggr)\\
&:=\int_{G_{\vartheta}}\biggl[ \negint_{S_0^k(E)-N^{-\zeta}(S_0^k)'(E)}^{S_0^k(E)-N^{-\zeta}(S_0^k)'(E)}
h_{\tilde E}\,d\tilde E\, \biggr]\,d\tilde P^{N,0}_{\beta,k}+O\biggl(\frac{1}{N} \,\int_{G_\vartheta} \sum_{(i_1, \ldots, i_m) \in J_{\hat\lambda,\tilde E}}|(X_{2,1}^N)_{i_j}^k|\,dP_{\beta}^{N,0}\biggr),
\end{split}
\end{equation}
with
$$
h_{\tilde E}(\hat\lambda):=\sum_{(i_1, \ldots, i_m) \in J_{\hat\lambda,\tilde E}} f\Bigl(N\bigl(
T_0^k(\lambda_{i_1}^k)-T_0^k(\tilde E)\bigr) +(X_{1,1}^N)_{i_1}^k(\hat \lambda),\ldots,N\bigl(T_0^k(\lambda_{i_m}^k)-T_0^k(\tilde E)\bigr)
+(X_{1,1}^N)_{i_m}^k(\hat \lambda)\Bigr).
$$
We now want to get rid of the terms $(X_{1,1}^N)_{i_j}^k$ and $|(X_{2,1}^N)_{i_j}^k|$.

Motivated by \eqref{eq:X1},
for any $\tilde E \in \R$ we define
$X_{1,\hat\lambda}^k(\tilde E)$ as the solution of the ODE
\begin{multline*}
\dot X_{t,\hat\lambda}^k(\tilde E)=(\yy_{k,t}^0)'\Bigl(X_{0,t}^k(\tilde E)\Bigr)\cdot X_{t,\hat\lambda}^k(\tilde E)
+\yy_{k,t}^1\Bigl(X_{0,t}^k(\tilde E)\Bigr)\\
+\sum_{\ell=1}^d\int \zz_{k\ell,t}\Bigl(X_{0,t}^k(\tilde E), y\Bigr)\,dM^N_{X_{0,t}^\ell}(y)
+\frac1N\sum_{\ell=1}^d\sum_{j=1}^N \partial_2\zz_{k\ell,t}\Bigl(X_{0,t}^k(\tilde E),X_{0,t}^\ell(\lambda_j^\ell)\Bigr)\cdot (X_{1,t})_j^\ell(\hat\lambda),
\end{multline*}
with $X_{0,\hat\lambda}^k(\tilde E)=\tilde E$, and
we note the following fact: whenever $\hat\lambda \in G_{\vartheta}$ 
we know that $\{\lambda_i^\ell\}_{1\leq i \leq N}$
are close, up to an error $N^\vartheta$, to the quantiles of the stationary measure $\mu_\ell^0=\mu_{\ell,0}^*$. Hence, arguing as we did for \eqref{eq:X1theta}
we get
\begin{equation}
\label{eq:DXE}
\left|\partial_{\tilde E}X_{1,\hat\lambda}^k(\tilde E)\right|\leq C\,N^\vartheta,\qquad |(X_{1,1}^N)_{i}^k(\hat \lambda)-X_{1,\hat\lambda}^k(\tilde E)|\leq C\,N^\vartheta\,|\lambda_i^k-\tilde E|
\qquad \forall\,\hat\lambda \in G_{\vartheta}.
\end{equation}
In addition, by the same reasoning,
$$
\max_{i,k}\int \partial_1 \zz_{k\ell,t}\Bigl(X_{0,t}^k(\lambda_i^k), y\Bigr)\,dM^N_{X_{0,t}^\ell}(y)=O\bigl(N^{\vartheta}\bigr)
\qquad \forall\,\hat\lambda \in G_{\vartheta},
$$
and the argument used to prove 
\eqref{eq:bound infty X12} (see in particular \eqref{eq:ODEX2 2}) yields
$$
\max_{i,k} |(X_{2,1}^N)_{i}^k| \leq C\,N^{2\vartheta}\qquad \forall\,\hat\lambda \in G_{\vartheta}.
$$
Hence, since $\# J_{\hat\lambda,\tilde E} \leq CN^{m\vartheta}$
we immediately deduce that 
\begin{equation}
\label{eq:X2infty}
O\biggl(\frac{1}{N} \,\int_{G_\vartheta} \sum_{(i_1, \ldots, i_m) \in J_{\hat\lambda,\tilde E}}|(X_{2,1}^N)_{i_j}^k|\,dP_{\beta}^{N,0}\biggr)
=O\bigl(N^{(m+2)\vartheta-1} \bigr).
\end{equation}
Now, to get rid of the term $X_{1,\hat\lambda}^k(\tilde E)$ inside $h_{\tilde E}$ we take advantage of \eqref{eq:DXE} and the average with respect to $\tilde E$:
more precisely, we consider the change of variable
$$
\tilde E\mapsto \Phi_{\hat\lambda}(\tilde E):=(T_0^k)^{-1}\Bigl[T_0^k(\tilde E)+\frac{1}{N}X_{1,\hat\lambda}^k(\tilde E)\Bigr]
$$
so that
\begin{multline*}
\negint_{S_0^k(E)-N^{-\zeta}(S_0^k)'(E)}^{S_0^k(E)-N^{-\zeta}(S_0^k)'(E)}
h_{\tilde E}\,d\tilde E=\\
\negint_{S_0^k(E)-N^{-\zeta}(S_0^k)'(E)}^{S_0^k(E)-N^{-\zeta}(S_0^k)'(E)}
\sum_{(i_1, \ldots, i_m) \in J_{\hat\lambda,\tilde E}} f\Bigl(N\bigl(
T_0^k(\lambda_{i_1}^k)-T_0^k(\tilde E)\bigr) +\bigl[(X_{1,1}^N)_{i_1}^k(\hat \lambda)-X_{1,\hat\lambda}^k(\tilde E)\bigr],\\
\ldots,N\bigl(T_0^k(\lambda_{i_m}^k)-T_0^k(\tilde E)\bigr)
+\bigl[(X_{1,1}^N)_{i_m}^k(\hat \lambda)-X_{1,\hat\lambda}^k(\tilde E)\bigr]
\Bigr)\,\partial_{\tilde E}\Phi_{\hat\lambda}(\tilde E)\,d\tilde E.
\end{multline*}
Therefore, since $\partial_{\tilde E}\Phi_{\hat\lambda}(\tilde E)=1+O\bigl(N^{\vartheta-1}\bigr)$ (thanks to \eqref{eq:DXE}),
$|h_{\tilde E}| \leq CN^{m\vartheta}$, and the interval $[S_0^k(E)-N^{-\zeta}(S_0^k)'(E),S_0^k(E)-N^{-\zeta}(S_0^k)'(E)]$
has length of order $N^{-\zeta}$,
we deduce that 
\begin{multline}
\label{eq:h tilde}
\negint_{S_0^k(E)-N^{-\zeta}(S_0^k)'(E)}^{S_0^k(E)-N^{-\zeta}(S_0^k)'(E)}
h_{\tilde E}\,d\tilde E=\\
\negint_{S_0^k(E)-N^{-\zeta}(S_0^k)'(E)}^{S_0^k(E)-N^{-\zeta}(S_0^k)'(E)}
\sum_{(i_1, \ldots, i_m) \in J_{\hat\lambda,\tilde E}} f\Bigl(N\bigl(
T_0^k(\lambda_{i_1}^k)-T_0^k(\tilde E)\bigr) +\bigl[(X_{1,1}^N)_{i_1}^k(\hat \lambda)-X_{1,\hat\lambda}^k(\tilde E)\bigr],\\
\ldots,N\bigl(T_0^k(\lambda_{i_m}^k)-T_0^k(\tilde E)\bigr)
+\bigl[(X_{1,1}^N)_{i_m}^k(\hat \lambda)-X_{1,\hat\lambda}^k(\tilde E)\bigr]
\Bigr)\,d\tilde E+O\bigl(N^{\zeta}N^{m\vartheta}N^{\vartheta-1}\bigr).
\end{multline}
We now observe that,
since $T_0^k:\R\to \R$ is a diffeomorphism with  $(T_0^k)' \geq e^{-L}>0$ (see \eqref{eq:T0k}), it follows by \eqref{eq:DXE} that 
$$
|(X_{1,1}^N)_{i_1}^k(\hat \lambda)-X_{1,\hat\lambda}^k(\tilde E)|\leq C\,N^\vartheta\,|T_0^k(\lambda_{i}^k)-T_0^k(\tilde E)|.
$$
Therefore, since $f$ is compactly supported, we see that the expression
\begin{multline*}
f\Bigl(N\bigl(
T_0^k(\lambda_{i_1}^k)-T_0^k(\tilde E)\bigr) +\bigl[(X_{1,1}^N)_{i_1}^k(\hat \lambda)-X_{1,\hat\lambda}^k(\tilde E)\bigr],\\
\ldots,N\bigl(T_0^k(\lambda_{i_m}^k)-T_0^k(\tilde E)\bigr)
+\bigl[(X_{1,1}^N)_{i_m}^k(\hat \lambda)-X_{1,\hat\lambda}^k(\tilde E)\bigr]
\Bigr)
\end{multline*}
is nonzero only if 
$$
|T_0^k(\lambda_{i_j}^k)-T_0^k(\tilde E)| \leq \frac{C_1}{N}
\qquad \forall\,j=1,\ldots,m.$$
In particular, using again that $(T_0^k)' \geq e^{-L}>0$, this implies that $|\lambda_{i_j}^k-\tilde E|\leq C/N$.
Thus
$$
\left|T_0^k(\lambda_{i_j}^k)-T_0^k(\tilde E) - (T_0^k)'(E)\,[\lambda_{i_j}^k-\tilde E]\right| =O\biggl(\frac{1}{N^2}\biggr)
$$
and
$$
N^{\vartheta}|T_0^k(\lambda_{i_j}^k)-T_0^k(\tilde E)|=O\bigl(N^{\vartheta-1}\bigr),$$
and we get
\begin{multline*}
f\Bigl(N\bigl(
T_0^k(\lambda_{i_1}^k)-T_0^k(\tilde E)\bigr) +\bigl[(X_{1,1}^N)_{i_1}^k(\hat \lambda)-X_{1,\hat\lambda}^k(\tilde E)\bigr],\\
\ldots,N\bigl(T_0^k(\lambda_{i_m}^k)-T_0^k(\tilde E)\bigr)
+\bigl[(X_{1,1}^N)_{i_m}^k(\hat \lambda)-X_{1,\hat\lambda}^k(\tilde E)\bigr]
\Bigr)\\
= f\Bigl((T_0^k)'(E)\,N\bigl(\lambda_{i_j}^k-\tilde E\bigr) ,
\ldots,(T_0^k)'(E)\,N\bigl(\lambda_{i_j}^k-\tilde E\bigr)\Bigr)+O\left(\|\nabla f\|_\infty N^{\vartheta-1}\right).
\end{multline*}
Combining this estimate with \eqref{eq:h tilde} and the fact that $\# J_{\hat\lambda,\tilde E} \leq CN^{m\vartheta}$
we conclude that
$$
\negint_{S_0^k(E)-N^{-\zeta}(S_0^k)'(E)}^{S_0^k(E)-N^{-\zeta}(S_0^k)'(E)}
h_{\tilde E}\,d\tilde E=\bar g_{E}+O\bigl(N^{(m+1)\vartheta+\zeta-1}\bigr),
$$
where
\begin{multline*}
\bar g_{E}(\hat\lambda):=\negint_{S_0^k(E)-N^{-\zeta}(S_0^k)'(E)}^{S_0^k(E)-N^{-\zeta}(S_0^k)'(E)}
\sum_{(i_1, \ldots, i_m) \in J_{\hat\lambda,\tilde E}} f\Bigl((T_0^k)'(E)\,N\bigl(\lambda_{i_j}^k-\tilde E\bigr) ,
\ldots,(T_0^k)'(E)\,N\bigl(\lambda_{i_j}^k-\tilde E\bigr)
\Bigr)\,d\tilde E.
\end{multline*}
Also, by the argument above it follows that we can add back into the sum all the indices
outside $J_{\hat\lambda,\tilde E}$ (since, up to infinitesimal errors, the function above vanishes on such indices),
therefore
$$
\negint_{S_0^k(E)-N^{-\zeta}(S_0^k)'(E)}^{S_0^k(E)-N^{-\zeta}(S_0^k)'(E)}
h_{\tilde E}\,d\tilde E=\bar{\bar g}_{E}+O\bigl(N^{(m+1)\vartheta+\zeta-1}\bigr),
$$
with
\begin{multline*}
\bar{\bar g}_{E}(\hat\lambda):=\negint_{S_0^k(E)-N^{-\zeta}(S_0^k)'(E)}^{S_0^k(E)-N^{-\zeta}(S_0^k)'(E)}
\sum_{i_1\neq \ldots\neq i_m} f\Bigl((T_0^k)'(E)\,N\bigl(\lambda_{i_j}^k-\tilde E\bigr) ,
\ldots,(T_0^k)'(E)\,N\bigl(\lambda_{i_j}^k-\tilde E\bigr)
\Bigr)\,d\tilde E.
\end{multline*}
Combining this bound with \eqref{eq:log 1}, \eqref{eq:log 2}, \eqref{eq:log 3}, \eqref{eq:log 4},
and \eqref{eq:X2infty},
we conclude that 
\begin{equation}
\label{eq:bound T1}
|\log(1+A_k) - \log(1+\bar{\bar A}_k)| \leq C\,\Bigl(
N^{m\vartheta-\zeta} +N^{(m+2)\vartheta-1}+N^{(m+1)\vartheta+\zeta-1}\Bigr),
\end{equation}
where $\bar{\bar A}_{k}:=\int\bar{\bar g}_E \,d P_\beta^{N,0}$.

We now repeat this very same argument replacing $P_\beta^{N,aV}$, $P_\beta^{N,0}$, and $T^N$,
with $P_\beta^{N,0}$, $(P_{{\rm GVE},\beta}^N)^{\otimes d}$, and $S^N=(S_1^N,\ldots,S_d^N)$, respectively
(see the discussion before Corollary \ref{thm:univ2}), and
we deduce that
$$
|\log(1+\bar{\bar A}_k) - \log(1+\hat A_k)| \leq C\,\Bigl(N^{m\vartheta-\zeta}+N^{(m+2)\vartheta-1}+
N^{(m+1)\vartheta+\zeta-1} \Bigr),
$$
where
$$
\hat A_{k}:=\int\biggl[ \negint_{E-N^{-\zeta}}^{E+N^{-\zeta}} \sum_{i_1\neq \ldots\neq i_m} f\bigl (R_k'(E)\,N(\lambda_{i_1}-\tilde E),\ldots,R_k'(E) \,N(\lambda_{i_m}-\tilde E)\bigr)\,d\tilde E\, \biggr] \,d P_{{\rm GVE},\beta}^N.
$$
Combining this estimate with \eqref{eq:bound T1} we get
 $$
 |\log(1+A_k) - \log(1+\hat A_k)| \leq C\,\Bigl(
N^{m\vartheta-\zeta} +N^{(m+2)\vartheta-1} +N^{(m+1)\vartheta+\zeta-1}\Bigr).
 $$
 Choosing $\vartheta$ small enough so that $(m+2)\vartheta<\theta$,
this gives
$$
|\log(1+A_k) - \log(1+\hat A_k)| \leq C\,\Bigl(N^{\theta+\zeta-1}+N^{\theta-1/2}+
N^{\theta-\zeta} \Bigr) \leq C\,\Bigl(N^{\theta+\zeta-1}+
N^{\theta-\zeta} \Bigr),
$$
and since  $\hat A_k$ is uniformly bounded in $N$ (see for instance \cite{VV})
and the right hand side is infinitesimal (recall that $\theta <\min\{\zeta,1-\zeta\}$),
we conclude that
$$
|A_k-\hat A_k|\leq C\,\Bigl(N^{\theta+\zeta-1}+
N^{\theta-\zeta} \Bigr).
$$
Recalling the definition of $A_k$ and $\hat A_k$, this proves that 
\begin{multline*}
\bigg|\int\biggl[ \negint_{R_k(E)-N^{-\zeta}\,R_k'(E)}^{R_k(E)+N^{-\zeta}\,R_k'(E)}\sum_{i_1\neq \ldots\neq i_m} f\bigl(N(\lambda_{i_1}^k-\tilde E),\ldots,N(\lambda_{i_m}^k-\tilde E)\bigr)\,d\tilde E\, \biggr]\,dP^{N,aV}_{\beta}\\
\qquad -\int\biggl[ \negint_{E-N^{-\zeta}}^{E+N^{-\zeta}} \sum_{i_1\neq \ldots\neq i_m} f\bigl (R_k'(E)\,N(\lambda_{i_1}-\tilde E),\ldots,R_k'(E) \,N(\lambda_{i_m}-\tilde E)\bigr)\,d\tilde E\, \biggr] \,d P^N_{{\rm GVE}}\bigg|\\
 \le  \hat C\,\Bigl(N^{\theta+\zeta-1}+
N^{\theta-\zeta} \Bigr),
 \end{multline*}
 which corresponds to our statement when $f$ depends only on the eigenvalues of one matrix.
 As explained at the beginning of the proof, the very same argument presented above extends also to the general case.
\end{proof}

\begin{proof}[Proof of Corollary  \ref{univmax}]
We begin by noticing that the proof of Theorem \ref{thm:transport} could be repeated verbatim in the context of \cite{BFG}
to show that \cite[Theorem 1.4]{BFG} holds with the same estimates as we obtained here.

To prove the  gaps estimates, it is enough to show that the approximate transport maps do not change gaps
in the bulk uniformly (away from the edges). 
Thanks to Theorem \ref{thm:transport} and \cite[Theorem 1.4]{BFG},
we have the expansions
$$
(T^N)^k_{i}(\hat\lambda)=T_0^k(\lambda_i^k)+\frac{1}{N} (X^N_{1,1})^k_{i}(\hat \lambda)+\frac{1}{N^2} (X^N_{2,1})^k_{i}(\hat \lambda),
$$
$$
(S^N_k)_i(\lambda^k)=S_0^k(\lambda_i^k)+\frac{1}{N} (S_{k,1})_{i}(\lambda^k)+\frac{1}{N^2} (S_{k,2})_{i}(\lambda^k),
$$
where $(S_{k,1})_{i}$ and $(S_{k,2})_{i}$ satisfy the same estimates as $(X^N_1)^k_{i}$ and $(X^N_2)^k_{i}$.
Hence, by the formulas above we deduce that
\begin{multline}
\label{eq:expTS}
(T^N)^k_{i}\bigl(S^N_1(\lambda^1),\ldots,S^N_d(\lambda^d)\bigr)
=T_0^k \circ S_0^k(\lambda_i^k)
+\frac{1}{N} \bigl[(T_0^k)' \circ S_0^k(\lambda_i^k) \bigr](S_{k,1})_{i}(\lambda^k) \\+ \frac{1}{N} (X^N_{1,1})^k_{i}\biggl(S_0^1(\lambda_1^1)+\frac{1}{N} (S_{1,1})_{1}(\lambda^1),\ldots,S_0^d(\lambda_N^d)
+\frac{1}{N} (S_{N,d})_{N}(\lambda^d)\biggr)+\mathcal E_i
\end{multline}
where the error $\mathcal E_i$ satisfies (thanks to the bounds from Theorem \ref{thm:transport} and \cite[Theorem 1.4]{BFG})
\begin{equation}
\label{eq:error i}
\sqrt{\sum_{i}\|\mathcal E_i\|^2_{L^2(P^N_{\rm GVE,\beta})} }=O\biggl(\frac{(\log N)^2}{N^{3/2}}\biggr)
\end{equation}
Also, using again Theorem \ref{thm:transport} and \cite[Theorem 1.4]{BFG},
with probability greater than $1-e^{-c(\log N)^2}$ and uniformly with respect to $i\in\{1,\ldots,N\}$, it holds
$$
\bigl|\bigl[(T_0^k)' \circ S_0^k(\lambda_{i+1}^k) \bigr](S_{k,1})_{i+1}(\lambda^k) 
-\bigl[(T_0^k)' \circ S_0^k(\lambda_i^k) \bigr](S_{k,1})_{i}(\lambda^k)  \bigr|
\leq C\,\log N \,N^{1/(\sigma-15)}|\lambda^k_{i+1}-\lambda^k_i|,
$$
\begin{multline*}
\bigl|(X^N_{1,1})^k_{i+1}-(X^N_{1,1})^k_{i}\bigr|\circ \biggl((S_0^1)^{\otimes N}+\frac{1}{N}S_{1,1},\ldots,
(S_0^d)^{\otimes N}+\frac{1}{N}S_{d,1}\biggr)(\hat\lambda)\\
\leq C\,\log N \,N^{1/(\sigma-15)} \biggl(|S_0^k(\lambda^k_{i+1})-S_0^k(\lambda_i^k)|
+\frac1N |(S_{k,1})_{i+1}(\lambda^k)-(S_{k,1})_{i}(\lambda^k)| \biggr)
\\
\leq C\,\log N \,N^{1/(\sigma-15)} |\lambda^k_{i+1}-\lambda^k_i|
\end{multline*}
while
$$
T_0^k \circ S_0^k(\lambda_{i+1}^k)
-T_0^k \circ S_0^k(\lambda_i^k)=(T_0^k \circ S_0^k)'(\lambda_i^k)[\lambda_{i+1}^k-\lambda_i^k]
+O(|\lambda_{i+1}^k-\lambda_i^k|^2).
$$
Recalling that, with probability greater than $1-e^{-N^{\bar c}}$,
$|\lambda_{i+1}^k-\lambda_i^k|\leq CN^{\theta-1}$  when the $\{\lambda_i^k\}_{1\leq i \leq N}$ are ordered and $i\in [\epsilon N, (1-\epsilon)N]$
(see \eqref{rigidity} and \eqref{eq:spacing}), 
we conclude that, with probability greater than $1-e^{-c(\log N)^2}$, uniformly with respect to $i\in [\epsilon N, (1-\epsilon)N]$, we have 
$$
\bigl[(T^N)^k_{i+1}-(T^N)^k_{i}\bigr]\bigl(S^N_1(\lambda^1),\ldots,S^N_d(\lambda^d)\bigr)
=(T_0^k \circ S_0^k)'(\lambda_i^k)[\lambda_{i+1}^k-\lambda_i^k]
+O\biggl(\frac{\log N\,N^{1/(\sigma-15)}}{N^{2-\theta}} \biggr).
$$
Combining this estimate with \eqref{eq:error i}
and noticing that
$$
N^{4/3}\biggl(\frac{\log N\,N^{2/(\sigma-15)}}{N^{2-\theta}} +\frac{(\log N)^2}{N^{3/2}}\biggr)\to 0\qquad \text{as }N\to \infty
$$
provided $\theta<1/6$
(recall that by assumption $\sigma \geq 36$, see Hypothesis \ref{hypo}),
the two statements follow from the fact that $T^N \circ( S^N_1,\ldots S^N_d):\R^{dN}\to \R^{dN}$ is an approximate transport map from 
$(P^N_{{\rm GVE},\beta})^{\ot d}$ to $P^{N,aV}_\beta$
and that the results are true under $P^N_{{\rm GVE},\beta}$ thanks to  \cite[Theorem 1.3 and Corollary 1.5]{BAB}. 
\end{proof}

\section{Matrix integrals}\label{matrixintegrals}
In this section, we consider the integral
$$I^{N,V}_\beta(A_1,\ldots,A_d, B_1,\ldots, B_m):=\int e^{  N^{2-r} \Tr ^{\otimes r} V(U_1A_1U_1^*,\ldots,U_dA_dU_d^*, B_{1},\ldots, B_m)} \,dU_1\ldots dU_d$$
where $\beta=2$ (resp. $\beta=1$) corresponds to integration over the unitary (resp. the orthogonal) group $U(N)$ (resp. $O(N)$).  Here $A_1,\ldots,A_d, B_1,\ldots, B_m$ are $m+d$ Hermitian (resp. symmetric) matrices 
such that
\begin{equation}\label{born}
\max_{1\le i\le d}\|A_i\|_\infty\le 1,\quad \max_{1\le i\le m}\|B_i\|_\infty\le 1,
\end{equation}
and
$V$ belongs to the tensor product $\mathbb C\langle x_1,\ldots,x_d;b_1,\ldots,b_m\rangle^{\otimes r}$ 
(or more generally to its closure for the norm defined below),
where 
$\mathbb C\langle x_1,\ldots,x_d;b_1,\ldots,b_m\rangle$ denotes the set of polynomial in $d+m$ self-adjoint variables.

We shall see $V$ as a Laurent polynomial in $\{u_i,u_i^*,a_i\}_{1\le i\le d}$ and
$\{b_i\}_{1\leq i \leq m}$,
where $x_i=u_ia_iu_i^{-1}$. The set $\LL$ of Laurent polynomials is equipped with 
 the involution $*$ given by  $u_i^*=u_i^{-1}$, $a_i^*=a_i$, $b_i^*=b_i$, and for any Laurent polynomials $p$
 and $q$ one has $(zpq)^*=\bar z q^* p^*$.
We
denote by $p=\sum\langle p,q_1\otimes \cdots\otimes q_r\rangle q_1\otimes \cdots\otimes q_r$ the decomposition of a polynomial 
$p$ in the set $${\LL}^{\ot r}:=\mathbb C\langle u_1,u_1^*,\ldots,u_d,u_d^*;a_1,\ldots,a_d;b_1,\ldots,b_m\rangle^{\otimes r}$$ in the basis of tensor of monomials, and for $\xi,\zeta\ge 1$ we set
$$\|p\|_{\xi,\zeta}:=\sum|\langle p,q_1\otimes\cdots\otimes q_r \rangle| \xi^{\sum_{i=1}^r {\rm deg}_U(q_i)}\zeta^{\sum_{i=1}^r \deg_{A,B}(q_i)}\,,$$
where $\deg_U(q)$ (resp. $\deg_{A,B}(q)$) is the number of letters in $\{u_i,u_i^*\}_{1\le i\le d}$  (resp. $\{a_i\}_{1\le i\le d}$
and $\{b_i\}_{1\le i\le m}$) in the word $q$. 
We let ${\LL}_{\xi,\zeta}^r:=\overline{{\LL}^{\otimes r}}^{\|\cdot\|_{\xi,\zeta}}$ be the closure of ${\LL}^{\otimes r}$
for the norm $\|\cdot\|_{\xi,\zeta}$. We endow
the space of linear forms  $\La_{\xi,\zeta}^r$ 
on ${\LL}_{\xi,\zeta}^r$ with the weak topology, that can be recast in terms of the norm
$$\|\tau\|_{\xi,\zeta}:=\sup_{\|p\|_{\xi,\zeta}\le 1}|\tau(p)|.$$
Notice that, by abuse of notation, we use $\|\cdot\|_{\xi,\zeta}$ to denote both the norm and the dual norm. It will always be clear from the context which one we are referring to. For later purpose, observe that $\xi,\zeta\mapsto\|p\|_{\xi,\zeta}$
is increasing for any $p\in\LL_{\xi,\zeta}^r$, whereas $\xi,\zeta\mapsto\|\tau\|_{\xi,\zeta}$ is decreasing for any $\tau\in\La_{\xi,\zeta}^r$. 
In the case where $r=1$, we denote in short ${\LL}_{\xi,\zeta},\La_{\xi,\zeta},\ldots$

We denote by $\La(\mathscr S)$  the set of linear forms  on a vector subspace $\mathscr S$ of ${\LL}$, and endow it with the weak norm $\|\cdot\|_{\xi,\zeta}$.
In particular if $\AA\BB$ is the algebra generated by $\{a_1,\ldots,a_d,b_1,\ldots,b_m\}$, the parameter $\xi$ does not appear
and we write in short $\|\cdot\|_\zeta$.
In case of a linear form on the algebra generated by a single self-adjoint variable, that corresponds simply to measure on the real line, this is
$$\|\nu\|_{\zeta}:=\sup_k\zeta^{-k}|\nu(x^k)|.$$
We denote by $\Ma(K)$ (resp. $\Pa(K)$) the set of Borel measures (resp. probability measure) on the set $K\subset\mathbb R$
and by $\BB$ the algebra generated by $\{b_1,\ldots,b_m\}$, and we
 write  $\|\nu\|_\zeta:=\sum_{i=1}^d\|\nu_i\|_{\zeta}+\|\tau\|_\zeta$
for $d+1$ tuples consisting of $d$ probability measures on $[-1,1]$ and
one linear form in $\La(\BB)$. Notice that, for $\tau\in \La(\BB)$, $$\|\tau\|_\zeta:=\sup_{k,i_j\in\{1,\ldots,m\}}\zeta^{-k} |\tau(b_{i_1}\cdots b_{i_k})|$$
as in this case the degree $\deg_{A,B}$ is simply the degree in $\{b_i\}_{1\le i\le m}$.
 We  assume without loss of generality that $V$ is symmetric, in the sense that for any permutation $\sigma$ on $\{1,\ldots,r\}$

$$\sum\langle V,q_1\otimes\cdots\otimes q_r\rangle q_1\otimes\cdots\otimes q_r =\sum\langle V,q_{\sigma(1)}\otimes\cdots\otimes q_{\sigma(r)}\rangle q_{\sigma(1)}\otimes\cdots\otimes q_{\sigma(r)}\,.$$
Compared to the notation used in \eqref{tyu}, we have rescaled $V$ so that the $A_i$ are bounded by $1$ instead of $M$, but otherwise
we can compare the norms as the diverse degrees are related by $\deg_U(q)\le \frac{1}{2}  \deg_X(q)$ and $\deg_{A,B}(q)= \deg_X(q)+\deg_B(q)$. In particular, the norm
$\|V\|_{\xi,\zeta}$ used in this section can be compared to the norm $\|V\|_{M\xi^{1/2}\zeta,\zeta}$ used in \eqref{tyu}. Once this is said, the
two notions are sufficiently close that we keep the same notation.

The main result of this section is the following.
\begin{thm}\label{topexpthm} Let $\beta=2$ (resp. $\beta=1$). Let $\{\a_j^i\}_{1\le i\le d, 1\le j\le N}\subset [-1,+1]^{dN}$ and set  $L^{N}_i:=\frac{1}{N}\sum_{j=1}^N\delta_{\a^i_j}$. Let $A_1,\ldots,A_d$ be Hermitian (resp. symmetric) matrices  with eigenvalues
$(\a_1^i,\ldots,\a^i_N)$, let $B_1,\ldots, B_m$ be Hermitian (resp. symmetric)  matrices, and
let $$p\mapsto \tau_B^N(p):=\frac{1}{N}\Tr\bigl(p(B_1\ldots,B_k)\bigr)$$ be the non-commutative distribution of $B_1,\ldots, B_m$.
 
 \noindent
Let $V\in \LL^r_{\|\cdot\|_{\xi,\zeta}}$ be self-adjoint. Then, if $\|V\|_{\xi,\zeta}$ is finite for some $\xi$ large enough and $\zeta\ge 1$,
there exists $a_0>0$ such that, for all $a\in [-a_0,a_0]$,
$$I^{N,aV}_\beta(A_1,\ldots,A_d, B_1,\ldots, B_m)=e^{\sum_{l=0}^2 N^{2-l} F^{aV}_{l,\beta}(L^{N}_1,\ldots,L^{N}_d,\tau_B^N)}\biggl(1+O\biggl(\frac{1}{N}\biggr)\biggr),$$
where the error is uniform on the set of matrices satisfying \eqref{born} and $F^{aV}_l$ are smooth functions on  $\mathcal P([-1,1])^d\times \La(\BB)$: more precisely, for any $\ell \geq 0$, the $\ell$-{th} derivative of $F^{aV}_{l,\beta}$ 
at $\mu\in \Pa([-1,1])^d\times \La(\BB)$ in the direction $\nu$ is such that
$$\bigl|D^\ell F^{aV}_{l,\beta}[\mu](\nu)^{\ot \ell}\bigr|\le C_\ell\,|a|\, \|\nu\|_\zeta^\ell\,, $$ 
where $C_\ell$ is a finite constant, uniform with respect to $\mu$.
\end{thm}

The proof of this theorem is split over the next sections.
For notational convenience, instead of adding a small parameter $a$ in front of $V$ we rather write down our hypotheses in terms of the smallness of the  norms of $V$.

\subsection{Integrals over the unitary or orthogonal group}
The goal of this section is to prove Theorem \ref{topexpthm}. 
Recall that
${\LL}_{\xi,\zeta},{\LL}_{\xi,\zeta}^{r}$ denote the completion of ${\LL},{\LL}^{\ot r}$ for the norm
$\|\cdot\|_{\xi,\zeta}$. 

We shall prove Theorem \ref{topexpthm} in two steps. First we extend the results of \cite{GN} to the case $\beta=1$ and $r\ge 1$:
\begin{prop}\label{topexpprop} Let $\beta \in \{1,2\}$.
Let $\tau^N_{AB}$ 
be the non-commutative distribution of $(A_1,\ldots,A_d, B_1,\ldots,B_m)$, that is, the linear form on $\AA\BB$ given by
$$\tau^N_{AB}(p):=\frac{1}{N}\Tr\left(p(A_1,\ldots,A_d,B_1,\ldots,B_m)\right)\qquad \forall\,p \in \LL.$$
There exist $\xi_0>1,\zeta\ge 1$, and $\epsilon_0>0$ such that if $\|V\|_{\xi_0,\zeta}\le\epsilon_0$ then,  uniformly on the set of matrices $ A_1,\ldots,A_d,B_1,\ldots,B_m$ satisfying \eqref{born} and with respect to the dimension $N$,
we have
$$I^{N,V}_\beta(A_1,\ldots,A_d, B_1,\ldots, B_m)=e^{N^2G_{0,\beta}^V (\tau^N_{AB})+NG_{1,\beta}^V (\tau^N_{AB})+G_{2,\beta}^V(\tau^N_{AB})}\biggl(1+O\biggl(\frac{1}{N}\biggr)\biggr),$$
where $G_{l,\beta}^V$ are real valued functions on $\LL(\AA\BB)$ and the error is uniform for the norm $\|\cdot\|_\zeta$.
\end{prop}
Next, we show   that the functions $\{G_{l,\beta}^V\}_{l=0,1, 2}$ depend only on the spectral measures of the matrices $A_i$ and on $\tau^N_B$. More precisely,
let $\Ta$ be the set of tracial states on ${\LL}$, that is, the set  of linear forms $\tau$ on ${\LL}$ satisfying
\begin{equation}\label{trace}\tau(pp^*)\ge 0, \quad \tau(pq)=\tau(qp),\quad \mbox{and}\quad\tau(1)=1\,.\end{equation}
Also, denote by $\Ta(\BB)\subset \LL(\BB)$ the set of tracial states on $\BB$.

Recall that, given $\nu=(\nu^1,\ldots,\nu^{d+1}) \in \mathcal M([-1,1])^d\times \La(\BB)$,
we have $\|\nu\|_\zeta =\sum_{i=1}^d \|\nu^i\|_\zeta+\|\nu^{d+1}\|_\zeta$ where
\begin{equation}
\label{eq:norm zeta}
\|\mu\|_\zeta=\max_{k\ge 1}\zeta^{-k}|\nu(x^k)|,\,\, \mu\in\mathcal P([-1,1]),\qquad \|\mu\|_\zeta=\max_{i_1,\cdots,i_k}\zeta^{-k}|\mu(B_{i_1}\cdots B_{i_k})|,\,\, \mu\in\Ta(\BB)\,.
\end{equation}

\begin{lem} \label{lem:FreeEnergy} The functions $\{G_l^V\}_{l=0,1, 2}$  are absolutely summable series whose coefficients depend only on $\tau_B^N$ and  the moments 

			\begin{equation*}
				L^{N}_i(x^k)=\frac1N \Tr\bigl[  (A_i)^k\bigr],  \qquad 1 \le i\le m,\ \ k\in\mathbb N.
			\end{equation*}
			In other words, there exists a function $F_{l,\beta}^V:\mathcal P([-1,1])^d\times \Ta(\BB)\to \R$ such that
			$$G_{l,\beta}^V(\tau^N_{AB})=F_{l,\beta}^V(L^N_1,\ldots,L^N_d,\tau_B^N)\,.$$
			Moreover, $F_{l,\beta}^V$ is Fr\'echet differentiable and its  derivatives are bounded by
			$$\left|D^\ell F_{l,\beta}^V[\mu](\nu_1,\ldots,\nu_\ell)\right|\le C_\ell\|\nu_1\|_\zeta \cdots\|\nu_\ell\|_\zeta.
			$$
	\end{lem}

As in \cite{GMS1,GMS2, CGM, BG1, GN}, the derivation of the expansion for large $N$ of the free energy 
$$F^{N,V}_\beta(A_1,\ldots,A_d,B_1,\ldots, B_m):=\frac{1}{N^2} \log I^{N,V}_\beta(A_1,\ldots,A_d,B_1,\ldots, B_m)$$
is based on the  expansion of the function given, for any polynomial $p\in \LL$,
 by 
\begin{equation}
\label{eq:W1N}
\W_{1N}^{V,\beta}(p):=\int \Tr\Bigl(p(U_1,\ldots,U_d,U_1^{*},\ldots,U_d^{*},
A_1,\ldots,A_d,B_1,\ldots,B_m)\Bigr)\, d\mathbb Q^{N,V}_\beta(U_1,\ldots,U_d),
\end{equation}
where $d\mathbb Q^{N,V}_\beta$ is the measure on $U(N)^d$ defined as
\begin{equation}
\label{eq:QNV}
d\mathbb Q^{N,V}_\beta(U_1,\ldots,U_d):=\frac{1}{I^{N,V}_\beta} e^{N^{2-r}\Tr^{\otimes r} V(U_1A_1U_1^*,\ldots,U_dA_dU_d^*,B_1,\ldots, B_m)} \,dU_1\ldots dU_d\,.
\end{equation}

The main step to prove Proposition \ref{topexpprop} is the following large dimension expansion:

\begin{prop}\label{corexpprop} Let $\beta= 1$ (resp. $\beta=2$). Let $A_1,\ldots, A_d$ be symmetric (resp. Hermitian) matrices with real eigenvalues $(\a_i^1,\ldots,\a^N_i)_{1\le i\le d}$ and satisfying \eqref{born}.
Let $V$ be a self-adjoint polynomial in $\LL_{\xi,\zeta}^r$
for some $\xi>1,\zeta\ge 1$. There exist  $\xi_0>1$, and $\epsilon_0>0$ so that if $\xi\ge \xi_0$ and  $\|V\|_{\xi,\zeta}\le\epsilon_0$ then$$\W_{1N}^{V,\beta}(p)=N\tau_{10}^\beta(p)+ \tau_{11}^\beta(p) +\frac{1}{N}\tau_{12}^\beta(p)+O\biggl(\frac{1}{N^2}\biggr)
\qquad\forall\,p \in \LL,$$
for some $\tau_{10}^\beta,\tau_{11}^\beta,\tau_{12}^\beta\in {\LL}_{\xi,\zeta}$.
Moreover, the error is uniform in $\|\cdot\|_{\xi,\zeta}$.
\end{prop}
Notice that this result
 implies Proposition \ref{topexpprop} provided we prove also  the convergence of the second correlator $\W_{2N}^{V,\beta}$, see \eqref{eq:WkN} and Section \ref{convfreeE}.

Hereafter we will drop the index $\beta$, but all our results will remain true both for $\beta=1$ and $\beta=2$.

The proof of Proposition \ref{corexpprop} is based on Schwinger-Dyson's equation and a priori concentration of measures' properties, which
depend
on differentials acting on the space ${\LL}$
of Laurent polynomial in letters $\{u_1,\ldots,u_d,u_1^{-1},\ldots,u_d^{-1}, a_1,\ldots,a_d,
b_1,\ldots,b_m\}$.  Recall that $\AA\BB$ denotes the
Laurent polynomial with degree zero, that is the linear span of words in $\{a_1,\ldots, a_d,
b_1\ldots,b_m\}$.
We now introduce some notation.

\begin{itemize}
\item The non-commutative derivative with respect to the $i$-th variable $u_i$ is defined by its action on monomials of ${\LL}$:

\begin{equation}\label{ncder}
\partial_i p:= \sum_{p=p_1u_ip_2} p_1u_i \otimes p_2
			- \sum_{p=p_1u_i^{-1}p_2} p_1 \otimes u_i^{-1}p_2.
			\end{equation}

\item The cyclic derivative with respect to $u_i$ is defined as  the endomorphism of
	${\LL}$ which acts on monomials according to 

		\begin{equation*}
			\D_i p := \sum_{p=p_1u_ip_2} p_2p_1u_i - \sum_{p=p_1u_i^{-1}p_2} u_i^{-1}p_2p_1.
		\end{equation*}
We can think about $\D_i$ as $\D_i=m\circ \partial_i$ with $m(p\otimes q):=qp$ for all $p,q\in {\LL}$.  We will denote $\tilde m(p\otimes q):= q^* p$. 

Note that $\D_i$ appears naturally when differentiating the trace of a polynomial. More precisely, if we let $u_j(t)=u_j$ for $j\neq i$ and $u_i(t) =u_i e^{tB}$ then, for any
Laurent polynomial $p$ and any tracial state $\tau$, we have 
$$\frac{d}{dt}|_{t=0} \tau\bigl( p(u(t))\bigr)=\tau\bigl(\D_i p(u(0) )B\bigr).$$
As we shall apply it to differentiate quantities of the form $\Tr^{\otimes r}V(U(t))$,  let us introduce the
following notation: for $p\in {\LL}^{\ot r}$ with $p=p_1\otimes p_2\otimes\cdots\otimes p_r$ and a tracial state $\tau$, we set
$$\D_{i,\tau} p:=\sum_{k=1}^r\biggl( \prod_{j=1}^{k-1} \tau(p_j) \biggr)\D_ip_k\biggl(\prod_{j=k+1}^r \tau(p_j)\biggr)\,.$$
Hence, if $B$ is a anti-symmetric matrix (that is $B=-B^*$) and  $U_j(t)=U_j e^{t\one_{j=i} B}$, 
$$\frac{d}{dt}|_{t=0} \frac{1}{N^r}\Tr^{\otimes r}V\bigl(U(t)\bigr)=\frac{1}{N}\Tr\bigl(B \D_{i,\frac{1}{N}\Tr } V\bigr)\,.$$

		\item 
		We will consider linear transformations

		\begin{equation*}
			\operator{T}:({\LL}^{\otimes k_1},\|\cdot\|_{\xi_1,\zeta}) \rightarrow ({\LL}^{\otimes k_2},\|\cdot\|_{\xi_2,\zeta})
		\end{equation*}

	\noindent
	mapping between the various tensor powers of ${\LL}$.
	A linear transformation $\operator{T} :{\LL}^{\otimes k_1}\ra {\LL}^{\otimes k_2}$
	is $(\xi_1,\xi_2;\zeta)$-continuous if and only if there exists a constant $C$ such that

		\begin{equation*}
			\|\operator T (p_1 \otimes \dots \otimes p_{k_1} )\|_{\xi_2,\zeta} \leq C \|p_1 \otimes \dots \otimes p_{k_1}\|_{\xi_1,\zeta}
		\end{equation*}

	\noindent
	for all monomials $p_1 \otimes \dots \otimes p_{k_1} \in {\LL}^{\otimes k_1}$.
	The operator norm of $\operator{T}$, denoted $\|\operator{T}\|_{\xi_1,\xi_2,\zeta}$,
	can be calculated by considering the smallest constant $C$ for which the above inequality holds.

	Allowing different instances of the $\xi$-norm on the source and target of our
	linear maps is useful for the following reason: certain linear transformations that
	we will need to deal with are not $(\xi,\xi;\zeta)$-continuous for any $\xi \geq 1$,
	but are $(\xi_1,\xi_2;\zeta)$-continuous, and even contractive, 
	if the ratio $\xi_1/\xi_2$ is large enough. When $\xi_1=\xi_2$ we simplify the notation by putting only one index $\xi$.

				\item
		Recall that for $\nu$ a multilinear form on ${\LL}^{\otimes k}$, we set
		$$\|\nu\|_{\xi,\zeta}=\max_{\|p\|_{\xi,\zeta} \le 1}|\nu(p)|,$$
		and denote by $\La_{\xi,\zeta}^{k,k'}$ the set of linear maps from $({\LL}^{\otimes k},\|\cdot\|_{\xi,\zeta})$ into $({\LL}^{\otimes k'},\|\cdot\|_{{\xi,\zeta}})$, 
		and $\La_{\xi,\zeta}^{k}$ denotes the set of linear maps from $({\LL}^{\otimes k},\|\cdot\|_{\xi,\zeta})$ into $\mathbb C$. Also, 
			if $\mathscr S$ is a vector subspace  of $({\LL}^{\otimes k},\|\cdot\|_{\xi,\zeta})$, then $\La(\mathscr S)$ is the set of linear forms on $\mathscr S$ (if $\mathscr S=\LL$, we simply denote it by $\La$). One can check that $\La_{\xi,\zeta}^{k,k'}$, $\La_{\xi,\zeta}^{k}$, and $\La(\mathscr S)$ are Banach spaces (see for instance \cite[Proposition 7]{GN} to see that $\|\cdot\|_{\xi,\zeta}$ is
			a vector space norm on ${\LL}^{\otimes k}$, and in fact an algebra norm). 
We denote by $\Ta_{\xi,\zeta}^k$ the subset of  tracial states on $({\LL}^{\otimes k},\|\cdot\|_{\xi,\zeta})$.  

\end{itemize}
The basis of the Schwinger-Dyson equation is the following equation: 
\begin{lem}\label{lem:SDeq} Let $V$ be a self-adjoint polynomial, $p\in{\LL}$, and $i\in\{1,\ldots, d\}$. Then
\begin{equation}\label{SD}
\mathbb E \biggl[\frac{1}{N}\Tr\otimes\frac{1}{N}\Tr(\partial_i p)+\frac{1+\one_{\beta=1}}{N} \Tr(\D_{i,\frac{1}{N}\Tr} V \,
p)\biggr]=\one_{\beta=1} \frac{1}{N} \mathbb E\biggl[ \frac{1}{N}\Tr\bigl(\tilde m\circ \partial_i p \bigr)\biggr],\end{equation}
where $\mathbb E$ denotes the expectation under $\mathbb Q_{\beta,N}^{V}$ (see \eqref{eq:QNV}).
\end{lem}
\begin{proof} We focus on the case $\beta=1$, the proof for $\beta=2$ is similar and detailed in \cite{GN} in the case $r=1$. 
This equation is derived by performing an infinitesimal change of variable 
 $U_i\mapsto U_i(t):=U_i e^{tD_i}$, where $D_i$ is a $N\times N$ matrix with real entries  such that $D_i^*=-D_i$, and writing
 that for any polynomial function $p\in {\LL}$ and any $k,\ell\in \{1,\ldots, N\}$
 $$\frac{d}{dt}|_{t=0}\int p\bigl(
 U_1(t),\ldots,U_d(t),U_1^{*}(t),\ldots,U_d^{*}(t),
A_1,\ldots,A_d,B_1,\ldots,B_m\bigr)_{k\ell} \,d\mathbb Q_{1,N}^{ V} (U_1(t),\ldots, U_d(t)) =0.$$
 Taking $D_j:=\one_{j=i}(\Delta(k,\ell)-\Delta(\ell,k))$, with $\Delta(k,\ell)$ the matrix with zero entries except at $(k,\ell)$ where the entry equals one,
 and summing over $k,\ell\in\{1,\ldots, N\}$, yields 
$$
\mathbb E \biggl[\frac{1}{N}\Tr\otimes\frac{1}{N}\Tr(\partial_i p)+\frac{1}{N} \Tr\Bigl(\bigl(\D_{i,\frac{1}{N}\Tr} V-(\D_{i,\frac{1}{N}\Tr} V)^*\bigr)  p\Bigr)\biggr]=\frac{1}{N} \mathbb E\biggl[ \frac{1}{N}\Tr\bigl( \tilde m\circ \partial_i p\bigr)\biggr].
$$
The last thing to check is that $(\D_{i,\frac{1}{N}\Tr}V)^*=-\D_{i,\frac{1}{N}\Tr} V$. Indeed,  it is enough to check it
for $r=1$. Then, for all $i$ and $p\in {\LL}$ we have
$$
{\mathcal D}_ip=\sum\langle p, q\rangle\, {\mathcal D}_i q
= \sum\langle p, q\rangle\,\biggl[\sum_{q=q_1u_iq_2} q_2q_1 u_i-\sum_{q=q_1 u_i^*q_2}u_i^*q_2 q_1\biggr],
$$
$$
{\mathcal D}_i(p^*)= \sum {\langle p, q\rangle}\,\biggl[-\sum_{q=q_1u_iq_2} u_i^* q_1^* q_2^* +\sum_{q=q_1 u_i^* q_2}q_1^* q_2^*u_i\biggr]=-({\mathcal D}_ip)^*.
$$
Since $V$ is self-adjoint, the proof is complete.
\end{proof}
Equation \eqref{SD} can be reinterpreted as a relation between the ``correlators'' $\W_{kN}^V$ defined as (see also \eqref{eq:W1N})
\begin{equation}
\label{eq:WkN}
\begin{split}
\W_{kN}^V(p_1,\ldots, p_k)&:=\frac{d}{dt_1}\cdots\frac{d}{dt_k} |_{t_1=0,\ldots,t_k=0} \log I_{\beta,N}^{ V+\frac{t_1 }{N} p_1+\ldots+\frac{t_k}{N} p_k}\\
&= \frac{d}{dt_2}\cdots\frac{d}{dt_k} |_{t_2=0,\ldots,t_k=0} \W_{1,N}^{ V+\frac{t_2}{N} p_2+\ldots+\frac{t_k}{N} p_k} (p_1).
\end{split}
\end{equation}
Notice that here the $p_i$'s belong to ${\LL}$,
but we can identify them with $p_i\otimes \one^{\ot r-1}\in 
{\LL}^{\ot r}$. 
Observe that we can always write the following expansion
$$\mathbb E\biggl[ \prod_{j=1}^r \Tr(q_j)\biggr]=   \prod_{j=1}^r \W_{1N}^V(q_j)+ \sum_{j\neq k}\W_{2N}^V(q_j,q_k)\prod_{\ell\neq j,k} \W_{1N}^V(q_\ell)+R_N(q_1,\ldots,q_r)$$
where $R_N(q_1,\ldots,q_r)$ is a sum of  product of correlators, each of which 
contains either a correlator of order at least $3$, or two correlators of order $2$.  We define
\begin{equation}\label{defS}\operator{S}^i_{V,\tau} p:=\sum_{j=1}^r \sum \langle V,q_1\otimes\cdots\otimes q_r\rangle \sum_{k\neq j}\biggl[\biggl(\prod_{\ell\neq k,j} \tau(q_\ell) \biggr)\, \D_i q_j \, p\otimes q_k 
+\sum_{m\neq j\neq k} \biggl(\prod_{\ell\neq k,j,m} \tau(q_\ell) \biggr)\,\tau( \D_i q_j \, p)\, q_m\otimes q_k\biggr]\,.\end{equation}
Using this expansion, we can rewrite \eqref{SD} as follows.
\begin{cor}\label{SD1} Let $V$ be a self-adjoint polynomial, $p\in{\LL}$, and $i\in\{1,\ldots, d\}$.
Then the first Schwinger-Dyson equation reads
\begin{multline*}
\frac{1}{N}\W_{1N}^V\otimes \frac{1}{N}\W_{1N}^V(\partial_i p) +\frac{1+\one_{\beta=1}}{N}\W_{1N}^V
(\D_{i,\frac{1}{N}\W_{1N}} V \, p)\\
= \frac{\one_{\beta=1} }{N^2} \W_{1N}^V(\tilde m\circ \partial_i p)-\frac{1}{N^2}  \W_{2N}^V(\partial_i p) -\frac{\one_{r\ge 2} }{N^2}\W_{2N}^V(\operator{S}^i_{V,\frac{1}{N}\W_{1N}^V} p
)+\frac{1}{N^r}R(\W_{1N}^V,\ldots, \W_{rN}^V: p),
\end{multline*}
where
$R$  is a sum (independent of $N$) of  product of correlators of polynomials extracted from $p$ and
$V$, each of which 
contains either a correlator of order at least $3$, or two  correlators of order $2$.

		\end{cor}

		To derive asymptotics from the Schwinger-Dyson equations we shall use a priori upper bounds on the correlators $\W_{kN}^V$.
		The next result (proved in  Appendix \ref{secconc})  is a direct consequence of concentration of measures and states as follows:
		\begin{lem}\label{conc} Let $p_1,\ldots,p_k$ be monomials in ${\LL}$. Then there exists a finite constant $C_k$, independent of  $N$ and the $p_i$'s, such that for $k\ge 2$
		$$|\W^V_{kN}(p_1,\ldots,p_k)|\le C_k \prod_{i=1}^k  \deg_U(p_i),\qquad | \W_{1N}^V(p)|\le N. $$
		In particular
		$\|\W_{kN}^V\|_{\xi,\zeta}\le C_k(\max_{\ell\ge 1}\xi^{-\ell}\ell)^k$ is finite for all $\xi>1,\zeta\ge 1$, and $k\ge 2$, whereas $ \|\W_{1N}^V(p)\|_{\xi,\zeta}\le N$  for any
		$\xi,\zeta\ge 1$.
		\end{lem}
				
		We now deduce the expansion of $\W^{ V}_{1N}$ up to order $O(N^{-2})$, and of $\W_{2N}^{V}$ up to $O(N^{-1})$. 
		
		As $N^{-1}\W_{1N}^V(p)$ is bounded by $1$ for all $p\in {\LL}$, we deduce that $N^{-1}\W_{1N}^V$ has limit 
		points. Let $\tau$ be such a limit point. As $N^{-1}\W_{2N}^V(\partial_i p)$ goes to zero for any polynomial $p\in {\LL}$  (see Lemma \ref{conc}), we deduce from
		the Schwinger-Dyson equation (see Corollary \ref{SD1}) that the limit point $\tau$ satisfies the limiting Schwinger-Dyson equation 
	
\begin{equation}\label{SDlim}		\tau\otimes \tau(\partial_i p) +(1+\one_{\beta=1})\tau(\D_{i,\tau} V\, p)= 0\qquad
		\forall\,p \in {\LL}.
		\end{equation}
		
		Hereafter we denote 
		$$V_\beta:=(1+\one_{\beta=1})V,$$ and we 
		show uniqueness of the solutions to such an equation whenever $\tau$ restricted to
		$\AA\BB$ is prescribed, $\|\tau\|_{1,1}\le 1$,  and $\|V\|_{\xi,\zeta}$ is small enough. 
		In our application $\tau_1:=\tau|_{\AA\BB}$ will simply be given by $\tau^N_{AB}$, the non-commutative distribution of $(A_1,\ldots,A_d,B_1,\ldots,B_m)$. It could also be given by
		its limit, if any, but we prefer to take it dependent on the dimension $N$.
		
		To show uniqueness, we apply the above equation to $p_i=\D_i q$ and sum over $i\in\{1,\ldots,d\}$. We will use that  (see \cite[Proposition 10]{GN})
			\begin{equation}\label{eq1}
\tau\otimes \tau\biggl(\sum_{i=1}^d \partial_i \mathcal D_i q\biggr)=\tau\bigl(\operator{D}q\bigr)+ \tau \otimes\tau\biggl(\sum_{i=1}^d \Delta_iq\biggr),
\end{equation}
		where:
		\begin{itemize}
		\item $\operator{D}$ is the degree operator: $\operator{D}p:={\rm deg}_U(p) \,p$.
		\item  $\Delta_i$ that  acts on monomials
		according to 
		$$\Delta_i p:= \partial_i \D_i p-\sum_{p=p_1u_ip_2} p_2p_1u_i\otimes \one-\sum_{p=p_1u_i^{-1} p_2} \one\otimes u_i^{-1} p_2p_1,$$
		that is,

			\begin{multline}
				\label{eqn:ReducedLaplacian}
				\Delta_ip = \sum_{p=p_1u_ip_2} \bigg{(} \sum_{p_2p_1u_i=q_1u_iq_2u_i} q_1\,u_i \otimes q_2\,u_i
				- \sum_{p_2p_1u_i=q_1u_i^{-1}q_2u_i} q_1 \otimes \,q_2\bigg{)} \\
				-\sum_{p=p_1u_i^{-1}p_2} \bigg{(} \sum_{u_i^{-1}p_2p_1=u_i^{-1}q_1u_iq_2} q_1\otimes q_2
				- \sum_{u_i^{-1}p_2p_1=u_i^{-1}q_1u_i^{-1}q_2} u_i^{-1}\,q_1 \otimes u_i^{-1}\,q_2 \bigg{)} ,
			\end{multline}
			where the sum is over all possible decompositions as specified. 			\end{itemize}
			
			We write in short
			$\Delta:=\sum_{i=1}^d \Delta_i\,,$
and we rewrite equation \eqref{SDlim}
		 as 
		\begin{equation}\label{eq1b}
		\tau\left( \Bigl(\operator{D}+\frac{1}{2}\operator{T}_\tau +\operator{P}^{V_\beta}_\tau
		\Bigr)q\right)=0\end{equation}
		where $\operator{T}_\tau$ and $\operator{P}^{V_\beta}_\tau$ are the following operators:
		\begin{itemize}
				\item $\operator{T}_\tau$ arises as the analogue of the Laplacian:
	\begin{equation*}
		\quad \operator{T}_\tau := (\operator{Id} \otimes \tau + \tau \otimes \operator{Id})\Delta.	\end{equation*}
			\item The operator $\operator{P}^{{V_\beta}}_\tau$ is the dot product  of the cyclic gradient of ${V_\beta}$ with 
	the cyclic gradient of $p$:
	 $$\operator{P}^{{V_\beta}}_\tau p:= \D_{\tau} {V_\beta} \cdot \D p=\sum_{i=1}^d \D_{i,\tau} {V_\beta}\cdot \D_ip.$$
\end{itemize}
	More generally, for linear forms $\tau_1,\ldots,\tau_{r-1}$ on ${\LL}$,  we define
	$$
	\operator{P}^{{V_\beta}}_{\tau_1,\ldots,\tau_{r-1}} p:=\sum_{i=1}^d \sum_{j=1}^ r\sum\langle {V_\beta}, q_1\otimes\cdots\otimes q_r \rangle
	\biggl(\prod_{k=1}^{j-1} \tau_k(q_k)\biggr)\,\D_i q_j\cdot\D_i p\,\biggl( \prod_{k=j+1}^r\tau_{k-1} (q_k)\biggr).
	$$
When $r\ge 2$, we also define a companion operator $\operator{Q}^{{V_\beta}}_{\tau_1,\ldots,\tau_{r-1}}$ to $\operator{P}^{{V_\beta}}_{\tau_1,\ldots,\tau_{r-1}}$:
	$$\operator{Q}^{{V_\beta}}_{\tau_1,\ldots,\tau_{r-1}}p:=\sum_{i=1}^d \sum_{1\le j< \ell\le r}\sum\langle {V_\beta}, q_1\otimes\cdots\otimes q_r \rangle\,
	\biggl(\prod_{k\in \{j,\ell\}^c } \tau_{k-\one_{k>\ell}}(q_k)\biggr)\,\tau_{j-\one_{j=r}}(\D_i q_j\cdot\D_i p) \,q_\ell\,.$$
	
We set $\Pi'$  (resp. $\Pi$) to be the orthogonal projection onto (resp. onto the complement of) the algebra $\AA\BB$ generated by $\{a_1,\ldots,a_d, b_1,\ldots,b_m\}$. 	For any linear transformation $\operator{T}$
			with domain ${\LL}$, we define its \emph{degree
			regularization} by $$\overline{\operator{T}}:=
			\operator{T}\operator{D}^{-1},$$
			where $\operator{D}$ is the degree operator defined above.
			It is understood that the domain of 
			the regularized operator $\overline{\operator{T}}$
			is restricted to $(\AA\BB)^\perp$. We recall that, for our applications, we assume that the restriction of $\tau$ to $\AA\BB$ is given and equal to $\tau_1$, therefore
						$$\tau=\tau\Pi+\tau_1\Pi'\,.$$
			Hence, we can see  \eqref{eq1b} as a fixed point equation for $\tau\in \La_{\xi,\zeta}$ given by
			\begin{equation}\label{deffix}F[\tau;\tau_1,V_\beta]=0,\qquad \tau|_{\AA\BB} =\tau_1,\end{equation}
			where $$F: \La_{\xi,\zeta}\times \left(  \Ta(\AA\BB),\|\cdot\|_\zeta) \times ({\LL}^{\ot r},\|\cdot\|_{\xi,\zeta}\right)\ra 
			\La_{\xi,\zeta} $$ is given by 
			$F[\tau; \tau_1,V_\beta]:=G[\tau\Pi+\tau_1\Pi';V_\beta]$ with 			\begin{equation}\label{defG}G[\tau;V_\beta](q):=
			\tau\left(\Bigl(\operator{Id}+ \frac{1}{2}  \overline{\operator{T}}_\tau +\overline{\operator{P}}^{V_\beta}_\tau\Bigr)\Pi q\right)\qquad
			\forall\, q\in \LL_{\xi,\zeta},\, \tau\in  \La_{\xi,\zeta}.
			\end{equation}
			When $V=0$ and $\tau_1\in\Ta(\AA\BB)$, the equation $F[\tau; \tau_1,0]=0$ has a unique solution  $\tau_{10}^{0,\tau_1}$  since the moments of $\tau$ are defined recursively from those of $\tau_1$. 
						In this case,  			$\tau$ is the non-commutative distribution of $(\{a_i,u_i,u_i^*\}_{1\le i\le d},\{b_j\}_{1\le j\le m})$  so that $(a_1,\ldots,a_d,b_1,\ldots,b_m)$
						has  law $\tau_1$, and is  free from the $d$ free unitary variables $(\{u_i,u_i^*\}_{1\le i\le d})$, 
			see \cite{Voi91} and \cite[Theorem 5.4.10]{AGZ}.

Observe that
			we know that solutions exist   in $\Ta(\AA\BB)$ as limit points of  $N^{-1}\W_{N1}^V$ (which is tight in any $\La_{\xi,\zeta}$ by Lemma 
			\ref{conc}); we shall  prove uniqueness of such solutions for $V$ small by applying ideas similar to those of the  implicit function theorem.
			
			 To state our result precisely, for $\xi>1$ and $\zeta\ge 1$ we define 			\begin{equation}\label{defdel}
			\delta_{\xi,\zeta}(V):= \frac{8}{(\xi-1)}+  \sum |\langle {V_\beta},q_1\otimes \cdots\otimes q_r\rangle| \Bigl(\sum_{j=1}^r {\rm deg}_U(q_j)\Bigr)\Bigl[\sum_{\ell=1 }^r 
\xi^{{\rm deg}_U(q_{\ell})}\zeta^{\deg_{A,B}(q_\ell)}\Bigr].\end{equation}
Observe that for $\xi\geq \xi_0$ with $\xi_0$ sufficiently large so that $\frac{8}{(\xi_0-1)}\leq\frac{1}{2(1+\max\{2,r\})}$, if $\|V\|_{\xi,\zeta}$ is finite one can choose $a_0$ small enough so that $\delta_{\xi,\zeta}(aV)<\frac{1}{1+\max\{2,r\}}$
for all $a\in[-a_0,a_0]$.
			\begin{lem}\label{conv}
			 Assume that there exist $\zeta\ge 1$ and $\xi> 1$ such that
\begin{equation}\label{thecond}
 \delta_{\xi,\zeta}(V)<\frac{1}{1+ \max\{2,r\} }\,.\end{equation}
Then, for any law $\tau_1\in \Ta(\AA\BB)$, there exists a unique solution $\tau_{10}^{V,\tau_1}\in \Ta \cap\La_{\xi,\zeta}$ to $$F[\,\cdot\, ; \tau_1,V_\beta]=0$$
such that $\tau|_{\AA\BB}=\tau_1$ and  $\|\tau\|_{1,1}\le 1$.   Moreover the map $\Ta(\AA\BB)\ni \tau_1\mapsto \tau_{10}^{V,\tau_1}\in\Ta_{\xi,\zeta}$ is Fr\'echet differentiable at all orders,
and its derivatives $D^\ell\tau_{10}^{V,\tau_1}$ satisfy, for any $\nu_1,\ldots,\nu_\ell\in \La_\zeta(\AA\BB)$,
$$\bigl\|D^\ell\tau_{10}^{V,\tau_1}[\nu_1,\ldots,\nu_\ell]\bigr\|_{\xi,\zeta}\le C_{\xi,\zeta,\ell}\|\nu_1\|_\zeta\cdots\|\nu_\ell\|_\zeta$$
for some finite constant $C_{\xi,\zeta,\ell}$.
 Finally, $$\lim_{N\ra\infty} \bigl\|N^{-1}\W_{1N}-\tau_{10}^{V,\tau^N_{AB}}\bigr\|_{\xi,\zeta}=0\,.$$
\end{lem}

Before proving Lemma \ref{conv},
we need the following technical result.

\begin{lem} Let $\xi>1,\tilde\xi\ge 1$ and $\zeta,\tilde \zeta\ge 1$. Then the following hold:
\begin{itemize}
\item Let $\operator{f}\in \La_{\tilde \xi,\tilde \zeta}$ and  $\xi> \tilde\xi$ and $\zeta\ge \tilde\zeta$. Then 

				\begin{equation}\label{boundT}
				 	\|\overline{\operator{T}}_{\operator{f}}\|_{\xi,\zeta} < 8\|\operator{f}\|_{\tilde \xi,\tilde\zeta}\frac{\tilde \xi}{(\xi-\tilde\xi)}.
				\end{equation}
				
		\item
		Let $\operator{f_1},\ldots,\operator {f}_{r-1} \in \La$. Then, for any $V \in \LL^r_{\xi,\zeta}$ self-adjoint and any $\tilde\xi,\tilde\zeta
		 \geq 1$, we have

			\begin{equation} \label{boundP}
				\left\|\overline{\operator{P}}^{{V_\beta}}_{\operator{f}_1,\ldots,\operator{f}_{r-1}}\right\|_{\xi,\zeta}
				 \leq \prod_{j=1}^{r-1} \|\operator{f}_j\|_{\tilde\xi,\tilde \zeta} \bigl\||\Pi {V_\beta}|\bigr\|_{\xi,\zeta,\tilde\xi,\tilde\zeta} \,,
			\end{equation}
			with
			\begin{multline*} \bigl\||\Pi {V_\beta}|\bigr\|_{\xi,\zeta,\tilde\xi,\tilde\zeta
			} := \sum |\langle {V_\beta},q_1\otimes \cdots\otimes q_r\rangle| \\
			\sum_{j=1}^r {\rm deg}_U(q_j)\xi^{{\rm deg}_U(q_j)}
			\zeta^{\deg_{A,B}(q_j)} \tilde\xi^{
			\sum_{i\neq j}{\rm deg}_U(q_i)}\tilde\zeta^{ \sum_{i\neq j}{\rm deg}_B(q_i)}\end{multline*}
			\item
		Let $\operator{f_1},\ldots,\operator {f}_{r-1} \in \La$. Then, for any $V \in \LL_{\xi,\zeta}^r$ self-adjoint and any $\tilde\xi,\tilde\zeta \geq 1$ with
		$\tilde\xi\le\xi$ and $\tilde\zeta\le\zeta$, we have

			\begin{equation} \label{boundQ}
				\|\overline{\operator{Q}}^{{V_\beta}}_{\operator{f}_1,\ldots,\operator{f}_{r-1}}\|_{\xi,\zeta} \leq \prod_{j=1}^{r-1} \|\operator{f}_j\|_{\tilde\xi,\tilde\zeta}\bigl\||\Pi {V_\beta}|\bigr\|_{\xi,\zeta,\tilde\xi,\tilde\zeta; 2} \,,
			\end{equation}
			with
			$$\bigl\||\Pi {V_\beta}|\bigr\|_{\xi,\zeta,\tilde\xi,\tilde\zeta;2} := \sum |\langle {V_\beta},q_1\otimes \cdots\otimes q_r\rangle| \sum_{j\neq \ell} \tilde \xi^{\sum_{i\neq \ell }{\rm deg}_U(q_i)}\tilde \zeta^{\sum_{i\neq \ell }{\rm deg}_{A,B}(q_i)}
 {\rm deg}_U(q_j)\xi^{{\rm deg}_U(q_{\ell})}\zeta^{{\rm deg}_{A,B}(q_\ell)} 
 \,.$$
 \item Let $\operator{f_1},\ldots,\operator {f}_{r} \in \La$, and for $V \in \LL^r_{\xi,\zeta}$ self-adjoint set
 \begin{multline}\label{defS2}\operator{S}^V_{\operator f_1,\ldots,\operator f_{r-2}} p:= \sum \langle V,q_1\otimes\cdots\otimes q_r\rangle\\
  \sum_{i=1}^d \sum_{j,k} \biggl[\biggl(\prod_{\ell \neq
 k,j} \operator f_{\ell -\one_{k\le \ell}-\one_{j\le \ell}} (q_\ell) \biggr)\, (1_{j<k} \D_i q_j \cdot \D_i p\otimes q_k+ 1_{k<j}
 q_k \otimes \D_i q_j \cdot \D_i p)\\
+\sum_{s\neq j,k} \biggl(\prod_{\ell\neq k,j,s} \operator f_{\ell -\one_{k\le \ell}-\one_{s\le \ell}-\one_{m\le \ell}}(q_\ell) \biggr)\,f_{r-2}( \D_i q_j \cdot \D_i p) \,q_s\otimes q_k\biggr].
\end{multline}
 Then, we have
	
			\begin{equation} \label{boundS}
				\left|\operator{f}_{r-1}\otimes \operator{f}_{r} ( \overline{\operator{S}}^V_{\operator{f}_1,\ldots,\operator{f}_{r-2}}(p  ))\right|
				 \leq \prod_{j=1}^{r} \|\operator{f}_j\|_{\tilde\xi,\tilde \zeta} \sum_{k=r-2}^r \frac{\|\operator{f}_k\|_{\xi,\zeta}}{\|\operator{f}_k\|_{\tilde\xi,\tilde\zeta}}
				  \bigl\||\Pi {V_\beta}|\bigr\|_{\xi,\zeta,\tilde\xi,\tilde \zeta;3}
				  \|p\|_{\xi,\zeta}, 
			\end{equation}
			
			where $$
			\||\Pi {V_\beta}|\bigr\|_{\xi,\zeta,\tilde\xi,\tilde \zeta;3}= r\bigl\||\Pi {V_\beta}|\bigr\|_{\xi,\zeta,\tilde\xi,\tilde\zeta }+r\bigl\||\Pi {V_\beta}|\bigr\|_{\xi,\zeta,\tilde\xi,\tilde\zeta;2}\,.
			$$			
\end{itemize}
\end{lem}

\begin{proof}
The proof of \eqref{boundT} is done by considering term by term the norm of $1\otimes \operator{f}\Delta_i p$. For instance, if $p$ has degree $d_i$ in $u_i$ and $u_i^*$,
and $d=\deg_U(p)$, we have
\begin{multline*}
\biggl\|\sum_{p=p_1u_ip_2}\sum_{p_2p_1 u_i=q_1u_iq_2 u_1} q_1u_i \operator{f}( q_2u_i)\biggr\|_{\xi,\zeta}
\le   \|f\|_{\tilde\xi,\tilde\zeta} \sum_{p=p_1u_ip_2}\sum_{p_2p_1 u_i=q_1u_iq_2 u_1}\|q_1 u_i\|_{\xi,\zeta}\|q_2u_i\|_{\tilde\xi,\tilde\zeta}\\
\le d_i  \|f\|_{\tilde\xi,\tilde\zeta} \sum_{p=0}^{d-1} \xi^{p}\,\tilde\xi^{d-p} \zeta^{\deg_{AB}(p)}
\le d_i  \|f\|_{\tilde\xi,\tilde\zeta} \|p\|_{\xi,\zeta} \frac{\tilde\xi}{\xi-\tilde\xi}
\end{multline*}
where we used that $\zeta\ge\tilde\zeta$ and $q_1$,$q_2$ have degree smaller than $d-1$. Proceeding for each term similarly (and noting a degree reduction
of each terms) yields the claim, after summing over $i$ and dividing by $d$.  More details are given  in \cite[Proposition 17]{GN} in  the case $\zeta=1$.

We prove next \eqref{boundP}.  Take $p$ a monomial in $(\AA\BB)^\perp$. Then, with $\epsilon,\epsilon_j=\pm 1$,
\begin{align*}
\|\overline{\operator{P}}^{{V_\beta}}_{\operator{f}_1,\ldots,\operator{f}_{r-1}}p\|_{\xi,\zeta}
&  =\biggl\| \frac{1}{\deg_U(p)}\sum_i \sum\langle {V_\beta},q_1\otimes \cdots\otimes q_r\rangle \sum_{j=1}^r\biggl( \prod_{k=1}^{j-1} \operator{f}_k(q_k)\biggr)\, \D_i q_j\cdot\D_i p\biggl(\prod_{k=j+1}^{r} \operator{f}_{k-1}(q_k)\biggr) \biggr\|_{\xi,\zeta}\\
&\le\frac{1}{\deg_U(p)}\sum_i \sum |\langle {V_\beta},q_1\otimes \cdots\otimes q_r\rangle| \sum_{j=1}^r \biggl(\prod_{k=1}^{j-1} |\operator{f}_k(q_k)|\biggr)\,\biggl(\prod_{k=j+1}^{r} |\operator{f}_{k+1}(q_k)| \biggr)\\
&\qquad\qquad\qquad\qquad\times \sum_{q_j=q_j^1u_i^{\epsilon^j} q_j^2}  \sum_{p=p^1u_i^{\epsilon} p^2}  \Bigl\|u_i^{-\one_{\epsilon^j=-1}} q_j^2q_j^1 u_i^{\one_{\epsilon^j=1}}
u_i^{-\one_{\epsilon=-1}} p^2p^1 u_i^{\one_{\epsilon=1}}\Bigr\|_{\xi,\zeta}\\
&\le \sum |\langle {V_\beta},q_1\otimes \cdots\otimes q_r\rangle| \sum_{j=1}^r\biggl( \prod_{k=1}^{r-1} \|\operator{f}_k\|_{\tilde\xi,\tilde \zeta}  \biggr)\,\tilde\xi^{\sum_{i\neq j}{\rm deg}_U(q_i)}\tilde\zeta^{\sum_{i\neq j}{\rm deg}_B(q_i)}\\
&\qquad\qquad\qquad\qquad\times
 {\rm deg}_U(q_j)\,\xi^{{\rm deg}_U(p)+{\rm deg}_U(q_j)} \zeta^{{\rm deg}_{A,B}(p)+{\rm deg}_{A,B}(q_j)} \\
\end{align*}
where we have used that $\xi,\zeta\ge 1$, that the degree of $u_i^{-\one_{\epsilon^j=-1}} q_j^2q_j^1 u_i^{\one_{\epsilon^j=1}}
u_i^{-\one_{\epsilon=-1}} p^2p^1 u_i^{\one_{\epsilon=1}}$ is at most ${\rm deg}_U(p)+{\rm deg}_U(q_j)$ in the $u_i$'s (and similarly in the $a_i$'s and $b_i$'s), and that the sum contained at most $\deg_U(p)\times {\rm deg}(q_j)$ terms.
We thus obtain \eqref{boundP}. 

To prove \eqref{boundQ} we note that  $\|\overline{\operator{Q}}^{{V_\beta}}_{\operator{f}_1,\ldots,\operator{f}_{r-1}}p\|_{\xi,\zeta}$
is equal to
\begin{align*}
&=\biggl\| \frac{1}{{\rm deg}_U(p)}\sum_i \sum\langle {V_\beta},q_1\otimes \cdots\otimes q_r\rangle \sum_{\ell\neq j}\biggl( \prod_{1\le k\le j-1\atop k\neq \ell
} \operator{f}_k(q_k)\biggr) \operator f_j(\D_i q_j\cdot\D_i p) q_\ell\biggl( \prod_{j+1\le k\le r\atop k\neq \ell } \operator{f}_{k-1}(q_k)\biggr) \biggr\|_{\xi,\zeta}\\
&\le
  \sum |\langle {V_\beta},q_1\otimes \cdots\otimes q_r\rangle| \sum_{j\neq \ell}\biggl( \prod_{k=1}^{r-1} \|\operator{f}_k\|_{\tilde\xi,\tilde\zeta}\biggr) \,\tilde \xi^{\sum_{i\neq \ell }{\rm deg}_U(q_i)+\deg_U(p)}\\
  &\qquad\qquad\qquad\qquad\times \tilde \zeta^{\sum_{i\neq \ell }{\rm deg}_B(q_i)+\deg_{A,B}(p)}
 {\rm deg}_U(q_j)\,\xi^{{\rm deg}_U(q_{\ell})}\zeta^{{\rm deg}_B(q_\ell)} \\
 &\le \bigl\||\Pi V_\beta|\bigr\|_{\xi,\zeta,\tilde\xi,\tilde\zeta;2}\|p\|_{\xi,\zeta},
\end{align*}
where we used in the last line that $\tilde\xi\le\xi$ and $\tilde\zeta\le \zeta$. The bound \eqref{boundS} is analogous and left to the reader.
\end{proof}

\begin{proof}[Proof of Lemma \ref{conv}] Following  the implicit function theorem, let us consider $F$ as a function from $X\times Y$ to $Y$, with
$X:=\La(\AA\BB)_{\zeta}\times \LL_{\xi,\zeta}^r$ 
and $Y:=\La(\AA\BB^\perp)_{\xi,\zeta}$.
(Here $\La(\AA\BB^\perp)$ is the set of linear functionals over $\AA\BB^\perp$. Even though $\AA\BB^\perp$ is not an algebra, this
is a well defined  Banach space once equipped with $\|\cdot\|_{\xi,\zeta}$.) 

Recall that $F$  has a unique solution $\tau^{0,\tau_1}_{10}$ on the subset of $X$ 
given by $\Ta(\AA\BB)\times\{0\}$, given by the law of free variables, as discussed above. To show that this unique solution extends to a neighborhood of
$\Ta(\AA\BB)\times\{0\}$, it is enough to check that $F$ is  differentiable along the variable $\tau\in Y$, and its derivative  is a Banach space isomorphism from  $\La(\AA\BB^\perp)_{\xi,\zeta}$ into $\La(\AA\BB^\perp)_{\xi,\zeta}$ at $(\tau_1,0)$.
But this is clear as for any $q\in \AA\BB^\perp$,
\begin{align*}
DF[\tau;\tau_1,V_\beta](\mu; 0)(q)&:=\lim_{\varepsilon\ra 0}\frac{1}{\varepsilon}\Bigl( F[\tau+\varepsilon\mu; \tau_1,V_\beta]-F[\tau;\tau_1,V_\beta]\Bigr)(q)\\
&= \mu\left( \Bigl(\operator{Id}+\Pi\Bigl[\overline{\operator{T}}_{\tau_{10}^{0,\tau_1} } \Bigr]\Bigr)q\right)\end{align*}
where $\operator{Id}+\Pi\overline{\operator{T}}_{\tau_{10}^{0,\tau_1} } $ is invertible, as a triangular operator. Hence, by the implicit function theorem
there exists a unique solution of $F(\tau;\tau_1,V_\beta)$ for $\|V_\beta\|_{\xi,\zeta}$ small enough and $\tau_1\in\Ta(\AA\BB)$.
However, for further use we shall prove again this result ``by hand''. 
For this, if $\tau$ and $\tau'$ are two solutions of
\eqref{deffix}
 we see that $\delta:=\tau-\tau'$ satisfies
\begin{equation}\label{eqdelta}
\delta\bigl((\operator{Id}+\Xi^V_{\tau,\tau_1})p\bigr)= \delta\otimes \delta\left(\overline{\Delta} p + \operator{R}^V_{\tau,\delta}p\right),\end{equation}
where
$$
			\Xi^V_{\tau,\tau_1}:=\Pi\Bigl[\overline{\operator{T}}_{\tau \Pi+\tau_1\Pi'} + \overline{\operator{P}}^{V_\beta}_{\tau \Pi+\tau_1\Pi'}+ \overline{\operator{Q}}^{V_\beta}_{\tau \Pi+\tau_1\Pi'}\Bigr]\,,$$
			and
$$\operator{R}^V_{\tau,\delta}:= - \int_0^1 \overline{\operator{S}}_{V,\tau'+s\delta}\,s \,ds,
\qquad \operator S_{V,\tau}(p):=\sum_{i=1}^d \operator{S}^i_{V,\tau}({\mathcal D}_ip)
$$
where $\operator{S}^i_{V,\tau}$ is defined in \eqref{defS}. 
Indeed, this follows by
the identity $\tau\otimes\tau-\tau'\otimes\tau'= \delta\otimes\tau+\tau\otimes\delta-\delta\otimes\delta$ and the expansion
\begin{align*}
\tau ( \overline{\operator{P}}^{V_\beta}_{\tau \Pi+\tau_1\Pi'}p)-\tau'( \overline{\operator{P}}^{V_\beta}_{\tau' \Pi+\tau_1\Pi'}p)&=\int_0^1 ds \,\frac{d}{ds} \left((\tau'+s\delta)(\overline{\operator{P}}^{V_\beta}_{(\tau'+s\delta) \Pi+\tau_1\Pi'}p)\right) \\
&=\delta \left(
\int_0^1 ds \left(\Pi\Bigl[ \overline{\operator{P}}^{V_\beta}_{(\tau'+s\delta) \Pi+\tau_1\Pi'}+ \overline{\operator{Q}}^{V_\beta}_{(\tau'+s\delta) \Pi+\tau_1\Pi'}\Bigr]p\right)\right)\\
&=\delta \left(\Pi\Bigl[ \overline{\operator{P}}^{V_\beta}_{\tau \Pi+\tau_1\Pi'}+ \overline{\operator{Q}}^{V_\beta}_{\tau \Pi+\tau_1\Pi'}\Bigr]p\right)+\delta\otimes \delta \left(
\int_0^1 ds\int_s^1 \,d\sigma\, \Pi\Bigl(\overline{\operator{S}}_{V,\tau'+\sigma\delta} p\Bigr)\right)\\
&=\delta \left(\Pi\Bigl[ \overline{\operator{P}}^{V_\beta}_{\tau \Pi+\tau_1\Pi'}+ \overline{\operator{Q}}^{V_\beta}_{\tau \Pi+\tau_1\Pi'}\Bigr]p\right)+\delta\otimes \delta \left(
\int_0^1 \,d\sigma \,\sigma\, \Pi\Bigl(\overline{\operator{S}}_{V,\tau'+\sigma\delta} p\Bigr)\right),
\end{align*}
which proves the desired formula noticing that $\delta=\delta\circ \Pi$.

We next claim that $\operator{Id}+\Xi^V_{\tau,\tau_1} $ is invertible and with bounded inverse in $((\AA\BB)^\perp,\|\cdot\|_{\xi,\zeta})$. 
We begin by noticing that  \eqref{boundT}, \eqref{boundP}, and \eqref{boundQ} imply the following: if $ \tau,\tau_1\in\Ta$, as $\tau\Pi+\tau_1\Pi'$ is a tracial state
which has $\|\cdot\|_{1,1}$ norm bounded by $1$, we have (by taking $\tilde\xi=\tilde\zeta=1$)
\begin{equation}
\label{eq:norm delta}\|\Xi^V_{\tau,\tau_1}\|_{\xi,\zeta}
\le    \frac{8}{(\xi-1)}+ \bigl\||\Pi {V_\beta}|\bigr\|_{\xi,\zeta,1,1} + \bigl\||\Pi {V_\beta}|\bigr\|_{\xi,\zeta,1,1; 2} =\delta_{\xi,\zeta}(V)
\end{equation}
(see \eqref{defdel}).
Therefore, since $\delta_{\xi,\zeta}(V)<1$ (by  \eqref{thecond}), it follows that $\operator{Id}+\Xi^V_{\tau,\tau_1}$ is invertible on $(\La\bigl((\AA\BB)^\perp\bigr),\|\cdot\|_{\xi,\zeta})$, with inverse bounded by $\left(1-\delta_{\xi,\zeta}(V)\right)^{-1}$.

By \eqref{boundT} and because $\|\delta\|_{1,1}\le \|\tau\|_{1,1}+\|\tau'\|_{1,1}\le 2$,   as well as $\|\tau'+s\delta\|_{1,1}\le 1$, 
$$\left| \delta\otimes \delta\left(\overline{\Delta} p\right)\right|=|\delta(\operator T_\delta p)|\le \frac{16}{(\xi-1)} \|\delta\|_{\xi,\zeta}\|p\|_{\xi,\zeta}$$
and we find similarly by \eqref{boundS} that for $\xi,\zeta\ge 1$, since $\|p\|_{1,1}\le \|p\|_{\xi,\zeta}$, 
$$|\delta \otimes\delta( \operator{R}^V_{\tau,\delta} (p))|\le \bigl\||\Pi {V_\beta}|\bigr\|_{\xi,\zeta,1,1;3} \|\delta\|_{\xi,\zeta} \|p\|_{\xi,\zeta}\,.$$
It follows from \eqref{eq:norm delta} and \eqref{eqdelta} that
$$\|\delta\|_{\xi,\zeta} \le \frac{\max\{2,r\} }{ 1-\delta_{\xi,\zeta}(V)} \delta_{\xi,\zeta}(V) \|\delta\|_{\xi,\zeta},$$
and recalling \eqref{thecond} we conclude that $\|\delta\|_{\xi,\zeta}=0$, that is $\tau=\tau'$ as desired.

We denote $\tau_{10}^{V,\tau_1}$ our unique solution. Notice that if $\tau_1$ is not necessarily a tracial state, but an
element of ${\LL}_{\xi,\zeta}$ which still satisfies $\|\tau_1\|_1\le 1$ 
and such that $\|\tau_1-\tau_1^0\|_{\zeta}\le \epsilon$ for some $\tau_1^0\in\Ta(\AA\BB)$ with $\epsilon$ small enough, then the very same argument as before shows that there exists a unique $\tau_{10}^{V,\tau_1}$ in a small neighborhood of $\tau_{10}^{V,\tau_1^0}$
solving \eqref{SD}.

By the implicit function theorem, since the function $F$ 
is smooth, the solution $ \tau_{10}^{V,\tau_1}$ is smooth both in $V$ and $\tau_1$.  For $\nu_1,\ldots,\nu_\ell\in \La_{\xi,\zeta}$, we denote by $D^\ell \tau_{0,1}^{V,\tau_1}$ the $\ell$-th derivative
of $\tau_{0,1}^{V,\tau_1}$ with respect to $\tau_1$, which is given by 
$$D^\ell\tau_{0,1}^{V,\tau_1}[\nu_1,\ldots,\nu_\ell]=\frac{d}{d\varepsilon_1}\ldots\frac{d}{d\varepsilon_\ell} |_{\varepsilon_1=0,\ldots,\varepsilon_\ell=0} \Bigl[\tau_{0,1}^{V,\tau_1+\sum_i\varepsilon_i \nu_i}\Bigr]\,,$$
and is defined inductively by the formula, valid for all $q\in (\AA\BB)^\perp$,
\begin{equation}\label{lkj}
D^1 \tau_{0,1}^{V,\tau_1}[\nu]\Bigl( (\operator{Id}+\Xi^V_{\tau_{01}^{V,\tau_1}} )q\Bigr)
=-\nu\left( \Pi'\Bigl[\overline{\operator{T}}_{\tau_{01}^{V,\tau_1} \Pi+\tau_1\Pi'} + \overline{\operator{P}}^{V_\beta}_{\tau_{01}^{V,\tau_1} \Pi+\tau_1\Pi'}+ \overline{\operator{Q}}^{V_\beta}_{\tau_{01}^{V,\tau_1} \Pi+\tau_1\Pi'}\Bigr]q\right)\,,\end{equation}
where we use the simplified notation $\Xi^V_{\tau_{01}^{V,\tau_1}}=\Xi^V_{\tau_{01}^{V,\tau_1},\tau_1}$.
Hence, if we denote in short $D_{I}\tau:=D^{|I|} \tau_{0,1}^{V,\tau^N_{AB}}[\nu_i,i\in I]$ and $K=\{1,\ldots,\ell\}$,
\begin{multline}
D_{K}\tau\left( \Bigl(\operator{Id}+\Xi^V_{\tau_{01}^{V,\tau_1} }\Bigr)q\right)=-\frac{1}{2}\sum_{I\cup J=K\atop I,J\neq \emptyset } D_I\tau\otimes D_J\tau(\overline{\Delta} p)\\
-\sum_{i=1}^\ell \nu_i\otimes D_{K\backslash\{i\}} \tau (\overline{\Delta} p)-\one_{\ell=2}\nu_i\otimes \nu_{K\backslash\{i\}} (\overline{\Delta} p)
-\sum_{\theta_i \in \{D_{J_i} \tau\}
\atop
\cup_{1\le i\le r}  J_i= K,J_1\neq \emptyset,K}\theta_1(\overline{\operator P}^{V}_{\theta_2,\ldots,\theta_r}q),\label{toto}
\end{multline}
where in the last term we sum over all choices of 
$\theta_i$
in the set $D_{J_i}\tau$, 
where $D_{J_i}\tau =\nu_{J_i}$  if $|J_i|=1$, 
and $D_{J_i}\tau=\tau$ if  $J_i=\emptyset$.
From this formula and the invertibility of $ \operator{Id}+\Xi^V_{\tau_{01}^{V,\tau_1} }$, we deduce by induction  that for all  $\xi$ satisfying 
\eqref{thecond} and for 
all $\ell\in \mathbb N$, there exists a finite constant $C_{\xi,\zeta,\ell}$ such that
  $$\left\|D^\ell\tau_{10}^{V,\tau_1}[\nu_1,\ldots,\nu_\ell]\right\|_{\xi,\zeta}\le C_{\xi,\zeta,\ell}\|\nu_1\|_{\zeta}\cdots\|\nu_\ell\|_{\zeta}.$$

Finally, we apply the above uniqueness result with $\tau_1:=\tau_{AB}^N$, that is, to the 
non-commutative distribution of $(A_1,\ldots,A_d, B_1,\ldots,B_m)$, see Proposition \ref{topexpprop}.
Indeed, by the discussion after Lemma \ref{conc},
any limit point of $N^{-1}\W_{1N}^V\in \La_{\xi,\zeta}$
satisfies the limiting Schwinger-Dyson equation,
so this lemma ensures that this limit is unique and that $N^{-1}\W_{1N}^V$
converge to $\tau_{10}^{V,\tau_{AB}^N}$,
which concludes the proof.
\end{proof}

In order to simplify the notation, we use $\tau_{10}$ to denote $\tau_{10}^{V,\tau_{AB}^N}$.
We next develop similar arguments to expand $\W_{1N}^V$ as a function of $N^{-1}$.
Let us first consider the first error term and rewrite  the first Schwinger-Dyson equation by taking $P=\D_ip$ in Corollary \ref{SD1}. Then, summing over $i$, we get
 $\delta_N:=\W_{1N}^V-N\tau_{10}$,
\begin{equation}\label{coc}
\delta_N\Bigl(  \bigl(\operator{Id}+\overline{\operator{T}}_{\tau_{10}} +\overline{\operator{P}}^{V_\beta}_{\tau_{10}} + \overline{\operator{Q}}^{V_\beta}_{\tau_{10}}\bigr)p\Bigr)=\frac{\one_{\beta=1}}{N}\W_{1N}^{V} (\tilde{\Delta} p)-\frac{1}{N}\W_{2N}^{V}(\overline{\Delta} p) +R_N(p),\end{equation}
where  
$$\tilde\Delta:= \sum_{i=1}^d \tilde m\circ \partial_i {\mathcal D}_i \operator{D}^{-1}$$
and
$R_N(p)$ contains the terms which are at least quadratic in $\delta_N$, or depending on cumulants of order 
greater or equal to $2$ :
\begin{align*}
&R_N(p):= -\delta_N(\overline{\operator{T}}_{N^{-1}\delta_N} p)\\
&\quad - \frac{ 1}{N^{r-1}}\sum_i  \sum_k \sum\langle {V_\beta},q_1\otimes\cdots\otimes q_r\rangle
\sum_{I\subset \{1,\ldots,r\}\backslash k \atop |I|\ge 1   } \delta_N( \D_iq_{k}\cdot\D_i \operator{D}^{-1} p)\biggl(\prod_{j\in I}\delta_N(q_j)\biggr)\biggl(\prod_{j\in (I\cup k)^c} \W_{1N}^V( q_j)\biggr)\\
&\quad -\frac{ 1}{N^{r-1}}\sum_i  \sum\langle {V_\beta},q_1\otimes\cdots\otimes q_r\rangle\\
&\qquad \times 
\sum_{I_1\cup I_2\cup\cdots I_k=\{1,\ldots,r\}, k\le r-1 }\W_{|I_1|N}^V( \D_iq_{i_1}\cdot\D_i  \operator{D}^{-1} p, q_j, j\in I_1\backslash \{i_1\})
 \prod_{\ell=2}^ k \W_{|I_\ell|N} (q_s, s\in  I_\ell),
 \end{align*}
 where in the above sum at least one set $I_j$ has at least two elements.
 
In order to control the right hand side of \eqref{coc}
we use the following estimate (compare with \cite[Proposition 18]{GN}):
\begin{lem}
\label{lem:laplacian}
For any  $\zeta\ge 1$ and $\xi_1> \xi_2$,  the
operator $\overline\Delta$
is a 
			bounded mapping from $((\AA\BB)^\perp,\|\cdot\|_{\xi_1,\zeta})$ into 
			$(\LL^{\otimes2},\|\cdot\|_{\xi_2,\zeta})$. Moreover
$\tilde \Delta $  is a 
			bounded mapping from $(\LL((\AA\BB)^\perp),\|\cdot\|_{\xi_1,\zeta})$ into 
			$(\LL,\|\cdot\|_{\xi_2,\zeta})$. 
\end{lem}
The proof of this result simply follows using \eqref{eqn:ReducedLaplacian}: using that formula
and noticing that there exists a constant $C_{\xi_1,\xi_2}>1$ such that $n\xi_2^n\le C_{\xi_1,\xi_2}\xi_1^n$
for all $n \geq 0$, one deduces that, for any monomial $p$,
$$\|\overline\Delta p\|_{\xi_2,\zeta}\le\deg_U(p) \xi_2^{\deg_U(p)}\zeta^{\deg_{A,B}(p)}\le C_{\xi_1,\xi_2} \xi_1^{\deg_U(p)}\zeta^{\deg_{A,B}(p)}=C_{\xi_1,\xi_2}\|p\|_{\xi_1,\zeta}.$$
The proof for $\tilde\Delta$ is similar. 

Next, we prove the following convergence result for $\delta_N$:
\begin{lem}\label{firstcor}
Assume there exist  $\xi_2<\xi_1$ and $\zeta \geq 1$, both for $\xi=\xi_1$ and for $\xi=\xi_2$,
$$ \delta_{\xi,\zeta}(V)<\frac{1}{1+\max\{2,r\}}\,.$$
 Then, for any $p\in \LL_{\xi_1,\zeta}$ we have 
$$\lim_{N\ra\infty} \delta_N(p)= \one_{\beta=1} \tau_{10}\left( \tilde{\Delta} \left(\operator{Id}+\overline{\operator{T}}_{\tau_{10}} +\overline{\operator{P}}^{V_\beta}_{\tau_{10}} + \overline{\operator{Q}}^{V_\beta}_{\tau_{10}}\right)^{-1}p \right)=:\tau_{11}(p)\,,$$
and $N\|\delta_N-\tau_{11}\|_{\xi_1,\zeta}$ is uniformly bounded in $N$.
\end{lem}
\begin{proof}
First notice  that for $\xi=\xi_1$ or $\xi=\xi_2$, our hypothesis ensures that
$$\Psi^{V_\beta}_\tau:=\operator{Id}+\overline{\operator{T}}_{\tau_{10}} +\overline{\operator{P}}^{V_\beta}_{\tau_{10}} + \overline{\operator{Q}}^{V_\beta}_{\tau_{10}}$$
 is invertible in $\LL_{{\xi,\zeta}}$ with norm smaller than $(1-\delta_{\xi,\zeta}(V))^{-1}$
 (see the proof of Lemma \ref{conv}).
 Therefore, it follows from \eqref{coc} that, for $p\in (\AA\BB)^\perp$, 
 \begin{equation}\label{pol}\delta_N(p)= \frac{1}{N} \W_{1N}^V\left( \tilde  \Delta( \Psi^{V_\beta}_{\tau_{10}})^{-1} p\right) - \frac{1}{N}\W_{2N}^V\left( \overline\Delta ( \Psi^{V_\beta}_{\tau_{10}})^{-1} p\right) + R_N\left(( \Psi^{V_\beta}_{\tau_{10}})^{-1} p\right)\,.\end{equation}
 We next bound each term separately. For the first one, we get
 \begin{multline*}
 \left|\frac{1}{N} \W_{1N}^V\left( \tilde  \Delta( \Psi^{V_\beta}_{\tau_{10}})^{-1} p\right)\right|\le \biggl\|\frac{1}{N} \W_{1N}^V\biggr\|_{\xi_2,\zeta}\| \tilde  \Delta( \Psi^{V_\beta}_{\tau_{10}})^{-1} p\|_{\xi_2,\zeta}
 \\
 \le \biggl\|\frac{1}{N} \W_{1N}^V\biggr\|_{\xi_2,\zeta}\| \tilde  \Delta\|_{\xi_2,\xi_1,\zeta}\|( \Psi^{V_\beta}_{\tau_{10}})^{-1} p\|_{\xi_1,\zeta}
 \le \biggl\|\frac{1}{N} \W_{1N}^V\biggr\|_{\xi_2,\zeta}\| \tilde  \Delta\|_{\xi_2,\xi_1,\zeta}\|( \Psi^{V_\beta}_{\tau_{10}})^{-1}\|_{\xi_1,\zeta}\| p\|_{\xi_1,\zeta}\,.\end{multline*}
A similar bound holds for the second term. For $R_N$, note first that \eqref{boundT} with $\tilde\xi= \xi_2$ yields
$$\left|\delta_N(\overline{\operator{T}}_{N^{-1}\delta_N} p)\right|\le 8 N^{-1}\frac{\xi_2}{\xi_1-\xi_2} \|\delta_N\|_{\xi_2,\zeta} \|\delta_N\|_{\xi_1,\zeta}\|p\|_{\xi_1,\zeta}$$
and noticing that similar bounds hold for the other terms in  $R_N$, we obtain
 \begin{multline*}
 \|\delta_N\|_{\xi_1,\zeta}
\le\biggl\| \frac{1}{N} \W_{1N}^V\biggr\|_{\xi_2,\zeta}\bigl\| \tilde {\Delta}\bigr\|_{\xi_2,\xi_1,\zeta} \left\| (\Psi^{V_\beta}_{\tau_{10}})^{-1} \right\|_{\xi_1,\zeta}
 +\biggl\|\frac{1}{N}\W_{2N}^V\biggr\|_{\xi_2,\zeta}\left\|\overline{\Delta}\right\|_{\xi_2,\xi_1,\zeta}
  \left\| (\Psi^{V_\beta}_{\tau_{10}})^{-1} \right\|_{\xi_1,\zeta} \\
 + C\biggl(1+ \left\| (\Psi^{V_\beta}_{\tau_{10}})^{-1} \right\|_{\xi_1,\zeta}\biggr)
 \frac{1}{N}\|\delta_N\|_{\xi_2,\zeta}
 \|\delta_N\|_{\xi_1,\zeta},\end{multline*}
 where we  bounded the last term using Lemma \ref{conc}.
 Since $N^{-1}\|\delta_N\|_{\xi_2,\zeta}\to 0$ (see Lemma \ref{conv}),
for $N$ sufficiently large we can reabsorb the last term and deduce that $\|\delta_N\|_{\xi_1,\zeta}$ is bounded.

 Moreover, this implies also that the  last term is of order $N^{-1}$.  In addition, the second one is of order $N^{-1}$ by Lemma \ref{conc}. Hence, 
 going back to \eqref{pol} we see that 
 the first term in the right hand side converges towards the desired limit by Lemma \ref{conv},
 provided $\tilde \Delta( \Psi^{V_\beta}_{\tau_{10}})^{-1} p\in \LL_{\xi_2,\zeta}$, which is true as soon as $p\in \LL_{\xi_1,\zeta}$ 
(see  Lemma \ref{lem:laplacian}).
 
 Finally, to prove the last statement, it is enough to notice that the above reasoning implies that $\|\delta_N\|_{\xi_3,\zeta}$
 is bounded for some $\xi_3\in (\xi_2,\xi_1)$ (notice that the assumption on $\delta_{\xi_3,\zeta}$ still holds for $\xi_3$ close enough to $\xi_2$ or $\xi_1$ by
 continuity of $\delta_{.,\zeta}$) so that the previous arguments (in particular the fact that $
\W_{2N}^V$ and $R_N$ are bounded) imply that there exists a finite constant $C$ such that
 $$N \|\delta_N-\tau_{11} \|_{\xi_1,\zeta}\le 
C\bigl\| \delta_N\bigr\|_{\xi_3,\zeta}\bigl\| \tilde {\Delta}\bigr\|_{\xi_3,\xi_1,\zeta} \left\| (\Psi^{V_\beta}_{\tau_{10}})^{-1} \right\|_{\xi_1,\zeta}
 +C$$
which concludes the proof. 
   \end{proof}
  The second order correction to $\W_{1N}^V$ depends on the limit of $\W_{2N}^V$ that we now derive by using the second Schwinger-Dyson equation.
  The latter is simply derived from the first Schwinger-Dyson equation (see Lemma \ref{lem:SDeq})
  by changing the potential $V$ into $V+tq\otimes \one^{r-1}$ and differentiating with respect to $t$ at $t=0$. This results into the equation, valid for all $p,q\in\LL$,
 \begin{multline*}
\mathbb E \biggl[(\Tr\,q-\mathbb E[\Tr \,q]) \left(\frac{1}{N}\Tr\otimes\frac{1}{N}\Tr(\partial_i p)+\frac{1+\one_{\beta=1}}{N} \Tr(\bigl(\D_{i,\frac{1}{N}\Tr} V)\,  p\bigr)\right)\biggr]\\
+\frac{1+\one_{\beta=1}}{N} \mathbb E\bigl[\Tr((\D_{i} q)  \,p)\bigr]
=\frac{1}{N} \mathbb E\biggl[(\Tr\,q-\mathbb E[\Tr \,q])  \Bigl(\frac{1}{N}\Tr(\tilde m\circ \partial_i p)\Bigr)\biggr]\,.
\end{multline*}
We next rearrange the above expression in terms of correlators $\W_{kN}^V$, $k=1,2$, replace $p$ by $\mathcal D_i p$, and sum over $i$, to deduce the second Schwinger-Dyson equation:
$$
\W_{2N}^V( q, p)=-\frac{1+\one_{\beta=1}}{N} \W_{1N}^V \left(\overline{\operator{P}}^q_{\tau_{10}} (\Psi_{\tau_{10}}^{V_\beta} )^{-1}p\right)+\hat R_N\left((\Psi_{\tau_{10}}^{V_\beta} )^{-1}p\right),$$
where $\hat R_N$ only depends on correlators  of  order greater than or equal to $3$, or on $\delta_N$ to a  power greater than or equal to $3$.
We can therefore see that $\hat R_N$ will be negligible provided $(\Psi_{\tau_{10}}^{V_\beta} )^{-1}p$ belongs to a space in which all the previous convergences hold.
This allows us to prove the following lemma:
\begin{lem}\label{secondcor} Let $\zeta\ge 1$.
Assume  there exist  $1<\xi_3<\xi_2<\xi_1$ such that, for $\xi=\xi_1,\xi_2,\xi_3$,
$$\delta_{\xi,\zeta}(V)<\frac{1}{1+\max\{2,r\}}\,.$$
 Then, for any $p,q\in  \LL_{\xi_1,\zeta}$ we have 
  $$\lim_{N\ra\infty} \W_{2N}^V(p,q)= -(1+\one_{\beta=1})\,\tau_{10} \Bigl(\overline{\operator{P}}^q_{\tau_{10}} (\Psi_{\tau_{10}}^{V_\beta} )^{-1}p\Bigr)=:\tau_{20}(p,q)\,,$$
and $N\|\W_{2N}^V-\tau_{20}\|_{\xi_1,\zeta} $ is uniformly bounded in $N$.
 \end{lem}
 
   We can finally derive the correction of order one for $\W_{1N}^V$ by going back to the first Schwinger equation. 
   Indeed if we denote $\delta^2_N:=N(\W_{1N}^V-N\tau_{10}-\tau_{11})$, the first Schwinger-Dsyon equation reads
$$
   \delta^2_N(\Psi_{\tau_{10}}^{V_\beta} p)= 1_{\beta=1}\delta_N( \tilde\Delta p)-[\W_{2N}^{V_\beta}+\delta_N\otimes \delta_N ]( \bar{ \operator{S}}^{V_\beta} p+ \overline{\Delta} p)+\tilde R_N(p),$$
   where $\tilde R_N(p)$ depends of correlators of order $3$ or higher, which are negligible by Lemma \ref{conc},
   and $ \operator{S}^{V}$ is defined in \eqref{defS2}. 
Then, arguing as previously, we infer the following result:
\begin{lem}\label{thirdcor}
  Assume  there exist  $1<\xi_4<\xi_3<\xi_2<\xi_1$ such that, for $\xi=\xi_1,\xi_2,\xi_3,\xi_4$,
$$\delta_{\xi,\zeta}(V)<\frac{1}{1+\max\{2,r\}}\,.$$
 Then 
 $$ \lim_{N\ra\infty}\delta^2_N(p)= \tau_{11}\left( \tilde \Delta(\Psi_{\tau_{10}}^{V_\beta})^{-1}p\right) -
  [\tau_{20}+\tau_{11}\otimes\tau_{11}  ]\left(\overline{\Delta}(\Psi_{\tau_{10}}^{V_\beta})^{-1}p+  \bar {\operator{S}}^{V_\beta} (\Psi_{\tau_{10}}^{V_\beta})^{-1}p\right)=:\tau_{12}(p)$$
 and $N\|\delta_N^2-\tau_{12}\|_{\xi_1,\zeta}$  is uniformly bounded in $N$. \end{lem}  
  
  This concludes the proof of Proposition \ref{corexpprop}.
We can now prove Proposition \ref{topexpprop} and Lemma \ref{lem:FreeEnergy}.

\subsection{Proof of Proposition \ref{topexpprop} and Lemma \ref{lem:FreeEnergy}.}\label{secFree}
We first show that the free energy is a function of the correlators, and then
that the correlators only depend on $\{L^N_i(x^\ell)\}_{\ell\ge 0, \,1\le i\le d}$ and $\tau_B^N$. Finally, we deduce the
large $N$ expansion of the free energy as well as its smoothness.
\subsubsection{The free energy in terms of the correlators}
\label{convfreeE}
Recalling the definition of free energy, \eqref{eq:QNV}, and \eqref{eq:W1N}, it holds
\begin{multline*}
F^{N,aV}_\beta(A_1,\ldots,A_d,B_1,\ldots, B_m):=\log I^{N,aV}_\beta\\
= \int_0^a \frac{d}{du} \log I^{N,uV}_\beta \,du=
N^2 \int_0^a  \int \frac{1}{N^r}\Tr^{\ot r}V\,  d\mathbb{Q}_\beta^{N,uV}\, du\\
=N^{2-r}\int_0^a (\W_{1N}^{uV})^{\ot r} (V) \, du +r(r-1)N^{2-r}\int_0^a \W_{2N}^{uV}\otimes(\W_{1N}^{uV})^{r-2}(V) \,du+\bar R_N\end{multline*}
where $\bar R_N$ has terms  either with two cumulants of order $2$, or a cumulant of order greater or equal to 3.
By Lemma \ref{conc} (note that it applies uniformly in $u\in [-a_0,a_0]$, for some $a_0$ universally small), this latter term is at most of order $1/{N}$, and is therefore negligible.
Moreover, using Corollary \ref{conv} and Lemmas \ref{secondcor} and \ref{thirdcor},
we find that

$$F^{N,aV}_\beta(A_1,\ldots,A_d,B_1,\ldots, B_m
)=N^2\int_0^a f_0^u du+N\int_0^a f_1^u du +\int_0^a f_2^udu +O\biggl(\frac{1}{N}\biggr)$$
with
\begin{align}
f_0^u&:=(\tau_{10}^{uV})^{\ot r} (V),\nonumber\\
f_1^u&:=r\, \tau_{11}^{uV}\otimes(\tau_{10}^{uV})^{\ot r-1} (V),\label{tol}\\
f_2^u&:=r(r-1)\bigl[(\tau_{11}^{uV})^{\ot 2}+\tau_{20}^{uV}\bigr]\otimes(\tau_{10}^{uV})^{\ot r-2} (V),\nonumber
\end{align}
where we have used that $V$ is symmetric and such that $\|V\|_{\xi_1,\zeta}$ is finite for $\xi_1$ big enough, so that $\delta_{\xi_1,\zeta}(uV)<(1+\max\{2,r\})^{-1}$ provided $u\in [-a_0,a_0]$ with $a_0$ sufficiently small.
 In particular this implies that, for $a_0$ small enough and any $1<\xi_4<\xi_3<\xi_2<\xi_1$,
$\delta_{\xi_i,\zeta}(uV)<(1+\max\{2,r\})^{-1}$ for all $u\in [-a_0,a_0]$, so that the previous lemmas apply.
Hence, we deduce the following:
\begin{lem}\label{totolem}
Let $\|V\|_{\xi_1,\zeta_1}$ be finite for some $\xi_1$  large enough and $\zeta_1\ge 1$. Then there exists $a_0>0$ so that, for $a\in [-a_0,a_0]$,
uniformly on Hermitian matrices $\{A_i\}_{1\le i\le d}$ and $\{B_i\}_{1\le i\le m}$ whose operator norm is bounded by
$1$, we have 
$$F^{N,aV}_\beta(A_1,\ldots,A_d,B_1,\ldots, B_m
)=\sum_{l=0}^2 N^{2-l}F^a_l +O\biggl(\frac{1}{N}\biggr)$$
with $F^a_l=\int_0^a f^u_l du$ and  $f^u_l$ given by \eqref{tol}.
\end{lem}

\subsubsection{The correlators as  functions of $\{L^N_i\}_{1\le i\le d}$ and $\tau^N_B$}

Let us define the space 
$$\mathcal P:=\Bigl\{Q(u_{1}a_{1}u_{1}^{-1},\ldots u_da_du_d^{-1}, b_1,\ldots,b_k):\, Q\in \mathbb C\langle x_1,\ldots,x_d, b_1,\ldots, b_m\rangle\Bigr\}.$$
As the functions $F^a_l$ only depend  on the restriction to $\mathcal P$
of $\tau_{10}^{uV}$, $\tau_{11}^{uV}$, $\tau_{12}^{uV}$, and $\tau_{20}^{uV}$ 
for $u \in [-a,a]$, we shall first prove that the latter only depend on 
$$M_{A,B}:=\biggl\{\frac{1}{N}\sum_{1\le j\le N} (a^i_j)^\ell:\, \ell\ge 0,\,1\le i\le d\biggr\}\cup\bigl\{\tau_B^N\bigr\}\,.$$

$\bullet${\it \, $\tau_{01}^{aV}|_{\mathcal P}$ depends only on $M_{A,B}$.}
We start by showing that $\tau_{01}^{aV}$ can be defined inductively, as is the case when $V=0$, since it depends analytically
on the potential $V$ in the following sense. \begin{lem}\label{expina}
Let $p\in \LL$ and $V$ be a potential such that, for some $\xi>1$ and $\zeta\ge 1$,
$$\delta_{\xi,\zeta}(V)<\frac{1}{1+\max\{2,r\}}\,.$$
Then, for all $a\in [-1,1]$, the solution $\tau_{10}^{aV}$ of
\begin{equation}\label{eq11}
		\tau\otimes \tau(\partial_i p) +a(1+\one_{\beta=1})\tau(\D_{i,\tau} V p)= 0\,.\end{equation} is uniquely defined. Moreover 
		we have the decomposition
		$$\tau^{aV}_{10}=\sum_{n\ge 0} a^n \tau^V_n$$
		with $\tau^V_n\in \La_{\xi,\zeta}$ satisfying $\|\tau^V_n\|_{\xi,\zeta}\le C_n D^n$, where $\{C_n\}_{n \geq 0}$ denote the Catalan numbers
		and $D$ is a positive constant.
		\end{lem}
		\begin{proof} This result can be seen to be a consequence of the implicit function theorem. However we will
		need soon additional informations on the $\tau^V_n$, and therefore give a proof ``by hand''.
		
		By
uniqueness of solutions it is enough to show that there exists a solution to \eqref{eq11},
or more precisely of \eqref{eq1b}, which is analytic in $a$. Let us therefore look for such a solution and write $\tau^{aV}(p):=\sum_{n\ge 0} a^n \tau_n^V(p)$. We then find that $\tau^{aV}$ satisfies \eqref{eq1b}
if and only if 
\begin{multline}
\label{tyr}
	\tau_n^V(p)+\sum_{k=0,n} \tau_k^V\otimes\tau_{n-k}^V ( \Pi\overline \Delta p) 
	=-\sum_{0<k<n} \tau_k^V\otimes\tau_{n-k}^V ( \Pi\overline \Delta p) \\
	\qquad -\sum \langle V, q_1\otimes \cdots\otimes q_r\rangle \sum_{i=1}^d\sum_{\ell=1}^r\sum_{\sum k_i=n-1}
	\biggl(\prod_{j\neq \ell}\tau^V_{k_j}(q_j)\biggr)\,\tau^V_{k_\ell}(\mathcal D_i q_\ell\cdot\mathcal D_i \operator{D}^{-1}p)
\end{multline}

As $\overline \Delta$ splits monomials $p$ into simple tensors $q_1 \otimes q_2$ 
each of whose factors has degree strictly smaller than that of $p$, we see that there exists a unique solution to this equation.
Moreover, we  prove by induction that there exists  finite constant $\xi,D>0$ such that, if $C_n$ denote the Catalan numbers,
then $$\|\tau_n^V\|_{\xi,\zeta} \le C_n D^n$$
Indeed, for $n=0$, we simply have the law of free variables bounded by $1$, so that the result is clear. 
 Using the inductive hypothesis until $n-1$ to bound the right hand side in \eqref{tyr}, and  \eqref{boundT} to bound
the second term in the left hand side of \eqref{tyr},
we deduce that 
\begin{multline*}
\bigl(1-  \delta_{\xi,\zeta}(V)\bigr) \|\tau_n^V\|_\xi\le  \frac{8}{ (\xi-1)}D^n
\sum_{0< k< n} C_kC_{n-k} \\
+D^{n-1} \sum |\langle V, q_1\otimes \cdots\otimes q_r\rangle| \Bigl(\sum \deg{q_i}\Bigr) \zeta^{\sum_i \deg_{A,B}(q_i)} (\xi )^{\sum \deg_U{q_i}} \sum_{\sum k_i\le n-1}\prod C_{k_i}
\end{multline*}
Using that $\sum_{0\le k\le n} C_k C_{n-k}=C_{n+1}\le 4 C_{n}$, we find recursively  $$ \sum_{\sum k_i\le n-1}\prod_{1\le i\le r} C_{k_i}\le C_{n+r-1}\le 4^{r-1} C_n.$$
Thus
we can bound the last term by $ 4^{r-1} C_n D^{n-1} \||V|\|_{\xi}$,
which implies that $$ \|\tau_n^V\|_\xi\le C_n D^n$$
provided $D$ is chosen sufficiently large.
Since $C_{n} \leq 4^n$,
this implies that $\tau^{aV}=\sum_{n\ge 0} a^n \tau_n^V$ is absolutely converging provided $|a|<1/(4D)$ and it satisfies \eqref{eq11}, so $\tau^{aV}=\tau_{01}^{aV}$ as desired.
\end{proof}

  We finally show that $\tau_n^V|_{\mathcal P}$ only depends on $M_{A,B}$. Again, we can 
argue by induction. As already mentioned, this is clear when $n=0$ 
as $\tau_0^V$ is the law of free variables. 
Also, if $p \in \mathcal P$ and $ {\rm deg}(p)=0$ then $p$ depends only on $b_1,\ldots,b_k$, and therefore $\tau_n^V$ only depends on $\tau_B^N$ for all $n \geq 0$.
Thus, by the inductive hypothesis, we can assume that  the result is 
true for $\tau_k^V(p)$ when $k\le n-1$ and $p\in \mathcal P$, and for $\tau_n^V(p)$ when $p\in\mathcal P$ and ${\rm deg}(p)\le \ell$.

To show that this property propagates we shall use the fact that \eqref{tyr} can be seen as an induction
relation where all monomials belong to $\mathcal P$.  To this end, first note that  $\{\tau_n^V\}_{n\ge 0}$ are tracial, that is
$$\tau_n^V(pq)=\tau_n^V(qp)\qquad\forall \,p,q\in \mathcal P\,.$$
Indeed this property is clear as it is satisfied by $\tau^{aV}$, and $\{\tau_n^V\}_{n \geq 0}$ are derivatives of $\tau^{aV}$
with respect to $a$. 

Next, observe that  $\operator{D}^{-1}$ keeps $\mathcal P$ stable. Moreover,  if $p=Q\left(\{u_{i} a_{i}u_{i}^{-1}\}_{1\le i\le m}\right)$ where $Q$ is a monomial, then
$${\mathcal D}_i p=\sum_{Q=q_1 x_i q_2}\bigl(a_{i}u_i^{-1} q_2q_1 u_i-u_i^{-1} q_2q_1u_ia_i\bigr),$$
so that, up to cyclic symmetry, ${\mathcal D}_i p\cdot {\mathcal D}_i q\in \mathcal P$ for each $i$
and $q\subset \mathcal P$. (Here and in the sequel, cyclic symmetry 
is just the action of exchanging $pq$ into $qp$.)  We also show that  $\overline \Delta$
maps $\mathcal P$ into $\mathcal P\otimes \mathcal P$ up to cyclic symmetry.
Indeed, it follows from \eqref{eqn:ReducedLaplacian} that,
	for $p\in\mathcal P$,

\begin{equation*}
\begin{split}
				\overline{\Delta}_ip &= \sum_{p=p_1u_ia_iu_i^{-1}p_2} \bigg{(} \sum_{a_iu_i^{-1}p_2p_1u_i=a_iu_i^{-1}q_1u_ia_iu_i^{-1}
				q_2u_i} a_iu_i^{-1}q_1u_i \otimes a_iu_i^{-1}q_2u_i
			\\
			&\qquad\qquad\qquad\qquad- \sum_{a_iu_i^{-1} p_2p_1u_i=a_iu_i^{-1}q_1u_ia_iu_i^{-1}q_2u_i} a_iu_i^{-1}q_1u_ia_i \otimes q_2 -a_i\otimes p_2p_1  -p_2p_1\otimes a_i \\
				&\qquad\qquad\qquad \qquad- \sum_{u_i^{-1}p_2p_1u_ia_i=u_i^{-1}q_1u_ia_iu_i^{-1} q_2u_ia_i} q_1  \otimes a_iu_i^{-1} q_2u_ia_i\\
				&
				\qquad\qquad\qquad\qquad+  \sum_{u_i^{-1}p_2p_1u_ia_i=u_i^{-1}q_1u_ia_iu_i^{-1} q_2u_ia_i} u_i^{-1}q_1 u_ia_i \otimes u_i^{-1}q_2u_ia_i \bigg{)},
\end{split}
\end{equation*}
so that, up to cyclic symmetry, $	\overline{\Delta}_ip\in \mathcal P\otimes\mathcal P$ for all $i\in \{1,\ldots,m\}$ and $p\in\mathcal P$.

Hence, by induction we see that $\tau_n^V$ restricted to $\mathcal P$ only depends on the restriction of $\{\tau_k^V\}_{k\le n-1}$ to $\mathcal P$,
therefore to the restriction of $\tau^V_0$ to $\mathcal P$.  Since we have already seen that
$\tau_0^V|_\mathcal P$  only depends on  $M_{A,B}$, the conclusion follows.\\

 $\bullet$ {\it \,$\tau_{11}^{aV}$ depends only on $M_{A,B}$.} 
 A direct inspection shows that $\tilde{\Delta}$ maps $\mathcal P$ into $\mathcal P$ up to cyclic symmetry. 
 Indeed, $\tilde{\Delta}=\sum_i\tilde{\Delta}_i$ with 
 \begin{equation*}
\begin{split}
				\tilde{\Delta}_ip &= \sum_{p=p_1u_ia_iu_i^{-1}p_2} \bigg{(} \sum_{a_iu_i^{-1}p_2p_1u_i=a_iu_i^{-1}q_1u_ia_iu_i^{-1}
				q_2u_i} u_i^{-1}q_2^* u_i a_i^2u_i^{-1}q_1u_i 
			\\
			&\qquad\qquad\qquad- \sum_{a_iu_i^{-1} p_2p_1u_i=a_iu_i^{-1}q_1u_ia_iu_i^{-1}q_2u_i} u_i^{-1} q_2^* u_i a_iu_i^{-1}q_1u_ia_i  -u_i^{-1}p_1^*p_2^* u_i a_i  -a_iu_i^{-1}p_2p_1u_i \\
				&\qquad\qquad\qquad - \sum_{u_i^{-1}p_2p_1u_ia_i=u_i^{-1}q_1u_ia_iu_i^{-1} q_2u_ia_i} a_iu_i^{-1} q_2^*u_i a_iu_i^{-1} q_1 u_i \\
				&\qquad\qquad\qquad+  \sum_{u_i^{-1}p_2p_1u_ia_i=u_i^{-1}q_1u_ia_iu_i^{-1} q_2u_ia_i} a_iu_i^{-1}q_2^*q_1 u_ia_i  \bigg{)}\,.
\end{split}
\end{equation*}
 Moreover, the previous considerations showed that $\Psi^{aV}_{\tau_{10}}$ maps $\mathcal P$ into $\mathcal P$ for $a$ small, therefore
 $$\tau_{11}^{aV}(p)= \one_{\beta=1} \tau_{10}^{aV}\left( \tilde {\Delta} (\Psi^{aV}_{\tau_{10}})^{-1}(p)\right)$$
 only depends on $\tau_{10}^{aV}|_{\mathcal P}$. 
Since we just checked that the latter only depends on $M_{A,B}$, this proves the result. \\

$\bullet${ \it $\tau_{20}^{aV}$ 
 depends only on $M_{A,B}$.}
 By Lemma \ref{secondcor}
 $$\tau_{20}^{aV}( \Psi_{\tau_{10}}^{aV} p,q)=-(1+\one_{\beta=1})\, \tau_{10}^{aV} (\overline{\operator{P}}^q_{\tau_{10}^{aV}} p),$$
 and recalling that $\tau_{10}^{aV}$ expands in a 
 convergent series in $a$,  we see that so does $\tau_{20}^{aV}$. We only need to check that the operators which appear in the
 equation defining $\tau_{20}^{aV}$ keeps $\mathcal P$ stable. But we have already seen that both operators $\overline{\Delta}$ and $\operator{P}^V$ keeps $\mathcal P$ stable, 
hence $\tau_{20}^{aV}(p,q)$ only depends on $M_{A,B}$ and it is in fact a convergent series
 in such elements.\\

$\bullet${ \it $\tau_{12}^{aV}$ 
 depends only on $M_{A,B}$.}
 By Lemma \ref{thirdcor}
$$\tau_{12}^{aV}(\Psi_{\tau_{10}}^{aV}p)=
\tau_{11}^{aV}(\tilde\Delta p)-[\tau_{20}^{aV}+\tau^{aV}_{11}\otimes\tau^{aV}_{11}]\bigl(  \overline{\Delta}p
 +\bar{\operator{S}}^{aV_\beta} p\bigr)\,,$$
from which we see that $\tau_{12}^{aV}(p)$ is a convergent series in $a$
(recall that we already proved that $\tau_{10}^{aV}(p),\tau_{11}^{aV}(p)$
and $\tau_{20}^{aV}(p)$ are convergent series in $a$).
 So the main point is to prove that,  up to cyclic symmetry,
 $\overline{\Delta}p
 + \overline{\operator{S}}^{aV_\beta} p\in \mathcal P\otimes \mathcal P$ whenever $p\in\mathcal P$.
 
 We already proved that this is the case for $\overline{\Delta}p$, so we focus on $\bar{\operator{S}}^{aV_\beta} p$.
We notice that it is the sum of two parts: one is linear over tensors of two monomials appearing in the decomposition of $aV$, and as $aV\in \mathcal P^{\ot r}$
 this part clearly belongs to $\mathcal P^{\otimes 2}$; the other part is linear over tensors of one monomial appearing in the decomposition of $aV$ (which therefore belongs to $\mathcal P$)
  and $\mathcal D_i p\cdot\mathcal D_i q_j$
with $q_j$ appearing in the decomposition of $aV$ (which we have seen belongs to $\mathcal P$ up to cyclic symmetry).
Hence also this second part satisfies the desired property, which concludes the proof.

\subsubsection{Smoothness of the functions $F_2,F_1,F_0$}
By Lemma \ref{totolem} and the discussion in the previous subsection, we know that 
$$F^{N,aV}_\beta=\sum_{l=0}^2  N^{2-l}F_{l}^a(L^N_1,\ldots,L_d^N ,\tau^N_B)+O\biggl(\frac1N\biggr),$$
where the functionals $F_0^a,F_1^a,F_2^a$ depend on $\{L^N_i\}_{1\le i\le d}$ and on $\tau^N_B$ through the asymptotic correlators $\{\tau_{1g}^{uV}\}_{0\le g\le 2}$ and $\tau_{20}^{uV}$.
We finally prove that they are smooth functions of these measures.

Recall the notation \eqref{eq:norm zeta}.
We show that:
\begin{lem}\label{lem:smooth}There exists $\xi_0>1$  large enough such that the following holds:
let $V$ have finite $\|\cdot\|_{\xi,\zeta}$ norm for some $\xi>\xi_0$ and $\zeta\ge 1$.
 Then there exists $a_0>0$ such that, for all $a\in [-a_0,a_0]$,
$F_l^a$ is Fr\'echet differentiable $\ell$-times for all $\ell \in \mathbb N$, and
if $\nu_j=(\nu_1^j,\ldots,\nu_d^j,\tau_j)\in P([-1,1])^d\times \Ta(\BB)$,
 we have
  $$\left|D^\ell F^a_l(L^N_1,\ldots,L^N_d,\tau^N_B)[\nu_1,\ldots, \nu_\ell]\right|\le C_{\ell}\,|a|\,\|\nu_1\|_\zeta\cdots\|\nu_\ell\|_\zeta.$$
Moreover, the derivative $D_k F^a_0(L^N_1,\ldots,L^N_d,\tau^N_B)=DF^a_0(L^N_1,\ldots,L^N_d,\tau^N_B)[0,\ldots,0,\delta_x,0\ldots,0]$ of $F$ in the direction of the measure
$L_k^N$ is a function on the real line with finite $\|\cdot\|_\zeta$ norm for any $k\in \{1,\ldots,d\}$.  As a consequence, it is of class $C^\infty$ in an open neighborhood of $[-1,1]$.
\end{lem}
\begin{proof} 
First, fix $\xi_0$ sufficiently large so that all previous results apply.
By the previous section it is enough to show that $\{\tau_{1g}^{uV}\}_{0\le g\le 2}$ and $\tau_{20}^{uV}$ depend smoothly on  $(\{L^N_i\}_{1\le i\le d},\tau^N_B)$,
uniformly with respect to $u  \in [-a,a]$. 
Indeed, by \eqref{tol}, $F^a_0$ is the integral of $(\tau^{uV}_{10})^{\ot r}(V)$ over $u\in [0,a]$. 
We have seen in Lemma \ref{conv} that $\tau^N_{AB}\mapsto \tau_{01}^{uV,\tau^N_{AB}}$ is $\ell$-times Fr\'echet
differentiable. Moreover, we have also seen that, once restricted to $\mathcal P$,
 it depends only on $\{L^N_i\}_{1\le i\le d}$ and $\tau^N_B$, and not the full distribution $\tau^N_{AB}$.
As a consequence, the  smoothness of  $\tau_{01}^{uV,\tau^N_{AB}}$ as a function of  $\tau^N_{AB}$ becomes a smoothness as a function of the probability measures $\{L^N_i\}_{1\le i\le d}$
 and $\tau^N_B$. The fact that $DF^a_0$ is $C^\infty$ is a direct consequence of formulas \eqref{lkj} and \eqref{toto}.
 For instance, if we denote by  $D_k$  the derivative along $L^N_k$,
    and $\Pi_k'$ is the projection onto the algebra generated by $\{a_k\}$, for any $p\in \LL_{\xi,\zeta}\cap \mathcal P$ we have
    \begin{equation}\label{deri}
 D_k \tau_{0,1}^{V,\tau_1}[p]= -\Pi_k'\Bigl[\overline{\operator{T}}_{\tau \Pi+\tau_1\Pi'} + \overline{\operator{P}}^{V_\beta}_{\tau \Pi+\tau_1\Pi'}+ \overline{\operator{Q}}^{V_\beta}_{\tau \Pi+\tau_1\Pi'}\Bigr]\bigl(\operator{Id}+\Xi^V_{\tau_{01}^{V,\tau_1}} \bigr)^{-1}p\in \mathcal P,\end{equation}
where we use the fact (see Lemma \ref{secFree}) that
$$\Bigl[\overline{\operator{T}}_{\tau \Pi+\tau_1\Pi'} + \overline{\operator{P}}^{V_\beta}_{\tau \Pi+\tau_1\Pi'}+ \overline{\operator{Q}}^{V_\beta}_{\tau \Pi+\tau_1\Pi'}\Bigr]\bigl(\operator{Id}+\Xi^V_{\tau_{01}^{V,\tau_1}} \bigr)^{-1}(\mathcal P)\subset \mathcal P$$
so that once we project it on $\AA\BB$ we get only polynomials either in the $a_i$ or in the $b_i$'s,
and hence differentiating in the direction of $L^N_k$
we only keep those in $a_k$.

 The same argument  holds for $F_1^u$ and $F_0^u$, since also $\tau_{10}^{uV}$, $\tau_{11}^{uV}$, and $\tau_{20}^{uV}$
 are smooth and only depend on $\{L^N_i\}_{1\le i\le d}$ and $\tau^N_B$.
\end{proof}

\section{ Law of polynomials of random matrices}\label{lawofpol}
Let us consider the equation
$$Y_i=X_i+a\,F_i(X_1,\ldots, X_d, B_1,\ldots, B_m)$$
 with $X_1,\ldots, X_d,$ (resp. $B_1,\ldots, B_m$) self-adjoint operators with norm bounded by $\xi$ (resp. $\zeta$)  and $F_i$ 
smooth functions (eventually polynomial functions) on such operators. We assume that $F_i$ are self-adjoint and that
$F_i=\sum \beta_i^q q$,  where the sum is over monomials in $X_i$'s and $B_i$'  with total degree ${\rm deg}_X(q)$ (resp. ${\rm deg}_B(q)$) in $X_1,\ldots,X_d$ (resp. in $B_1,\ldots, B_m$).
We also assume that for $\zeta\ge 1$ and $\xi$ large enough
 $$\|F_i\|_{\xi,\zeta}:=\sum |\beta_i^q|\,\xi^{{\rm deg}_X(q)} \zeta^{{\rm deg}_B(q)} <\infty\,.$$ 
 By the implicit function theorem, see \cite[Corollary 2.4]{GS}, for any fixed $\xi,\zeta$ there exist $A<A'<\xi$ 
 such that for
$a$ small enough (e.g., so that $A+|a| \|F_i\|_{\xi,\zeta}\le A'$) there exist analytic functions $G_i$, with $\|G_i\|_{A,\zeta}=O(|a|)$,
satisfying
$$
X_i=Y_i+G_i(Y_1,\ldots,Y_d, B_1,\ldots, B_m)\,,
$$
for all operators $Y_i$  whose norm is bounded by $A$.

To be precise, notice that \cite{GS} only consider the case where the $B_i$'s are constant, but the proof extends readily to the
case where some additional fixed matrices $B_i$ are present, as it is based on a fixed point argument showing
that the sequence
$$ X_i^0=Y_i,\qquad X^{n+1}_i=Y_i-a\,F_i( X^n_1,\ldots, X^n_d, B_1,\ldots, B_m)$$ is Cauchy for $\|\cdot\|_{A,\zeta}$ provided $a$ is small enough. Since the closure $\mathbb C\langle x_1,\ldots,x_d; b_1,\ldots,b_m\rangle_{A,\zeta}$
 of the  space of polynomials  under $\|\cdot\|_{A,\zeta}$  is complete,
it follows that the sequence $\{X_i^n\}_{n \in \mathbb N}$ converges in this space for all $1\le i\le d$.  This construction also shows that there exist functions $G_i\in \overline{\mathbb C\langle x_1,\ldots,x_d; b_1,\ldots,b_m\rangle}^{\|\cdot\|_{A,\zeta}}$
satisfying the desired properties.

We next consider the law ${\mathbb P}^N_Y$ of the  random matrices
$$Y^{N}_i=X^{N}_i+a\,F_i(X^{N}_1,\ldots, X^{N}_d,B^{N}_1,\ldots,B^{N}_m)$$
for $d$ independent GUE matrices $X^{N}_1,\ldots, X^{N}_d$ and $m$ deterministic matrices $B^{N}_1,\ldots, B^{N}_m$. Our goal in this section
is to show that the law of $Y^{N}_1,\ldots, Y^{N}_d$ satisfies our previous hypotheses.

First, notice that by Lemma \ref{confinement} applied to the current situation where the equilibrium density is the semicircle law, see \eqref{eq:semicircle}, the matrices $X^{N}_i$ have norms bounded by $3$ with probability greater than $1-e^{-cN}$. Hence, if we fix $\xi=4$ 
and $F_i\in \mathbb C\langle x_1,\ldots,x_d; b_1,\ldots,b_m\rangle_{\xi,\zeta}$ we see that, with probability greater than $1-e^{-cN}$, for $a$ small enough we have
$$X^{N}_i=Y^{N}_i+G_i(Y^{N}_1,\ldots,Y^{N}_d, B^{N}_1,\ldots,B^{N}_m)\,,$$
for some $G_i\in \overline{\mathbb C\langle x_1,\ldots,x_d; b_1,\ldots,b_m\rangle}^{\|\cdot\|_{A,\zeta}}$ with $3\le A<A'<\xi$. 

Therefore, up to an error of order $e^{-cN}$ in the total variation norm,  we have
\begin{multline*}
{\mathbb P}^N_Y(dY_1^N,\ldots,dY_d^N)=\frac{1}{Z_N} e^{-N\sum_{i=1}^d \Tr \left(Y_i^N+G_i(Y_1^N,\ldots,Y_d^N,B_1^N,\ldots, B_m^N)\right)^2}\\
\times \mbox{Jac} \,G\bigl(Y^{N}_1,\ldots,Y^{N}_d, B^{N}_1,\ldots, B^{N}_m\bigr)  \prod_{i} dY_i^N
\end{multline*}
where $ \mbox{Jac} \,G(Y^{N}_1,\ldots,Y^{N}_d, B^{N}_1,\ldots, B^{N}_m)$ 
denotes the Jacobian of the change of variable $X_i= Y_i+G_i(Y_1,\ldots,Y_d, B_1,\ldots,B_m)$.
It turns out that in the case $\beta=2$
$$\log  \mbox{Jac} G(Y^{N}_1,\ldots,Y_d^{N}, B^{N}_1,\ldots, B^{N}_m) =\tr_d\bigl( \Tr\otimes \Tr (\log (\operator{Id}+\mathcal J  G))\bigr)$$
where $\Tr$ is the trace over $N\times N$ matrices, $\tr_d$ is the trace over $d\times d$ matrices, and 
$$(\mathcal J G)_{ij,k\ell; t,s}=\partial_{Y^N_t(k\ell)} G_s(ij) =
(\hat\partial_t G_s)_{ik,\ell j}\,, \qquad i,j,k,\ell \in \{1,\ldots,N\},\,
s,t \in \{1,\ldots,d\},$$
where $\hat\partial_t$ denotes the non-commutative derivative over polynomial of self-adjoint variables defined as
$$\hat\partial_tp:=\sum_{p=q_1Y_t q_2} q_1\otimes q_2\,.$$
Indeed, the above formula follows from the fact that $\hat\partial_t p$ lives in the tensor product space (in other words, on
the algebra of left multiplication tensored with the right multiplication) and
$$\partial_{Y_t^N({k\ell})} G_s({ij})=(\hat \partial_t G_s\sharp \Delta_{k\ell})(ij)= (\hat\partial_t G_s)_{ik,\ell j},$$
where $\Delta_{k\ell}$ is the matrix with null entries except at position $\ell k$ where there is a one (here $A\otimes B\sharp C=ACB$).

As $G$ is small for $a$ small enough (at least when restricted to matrices with universally bounded operator norm), the singularity of the logarithm
is away from our support of integration and  we
deduce that  the law of $Y^{N}_1,\ldots, Y^{N}_d$ can be approximated in the total variation distance by 
$$\frac{1}{Z_N}e^{ N\Tr \,F_1(Y^{N}_1,\ldots,Y^{N}_d,B^{N}_1,\ldots,B^{N}_m)+ \Tr \otimes \Tr \,F_2(Y^{N}_1,\ldots,Y^{N}_d,B^{N}_1,\ldots,B^{N}_m)} \prod dY^{N}_i$$
for two smooth functions $F_1$ and $F_2$, belonging respectively to the closure of $\mathbb C\langle x_1,\ldots,x_d,b_1,\ldots, b_m\rangle$ and $\mathbb C\langle x_1,\ldots,x_d, b_1,\ldots, b_m\rangle^{\otimes 2}$ 
with respect to the norms $\|\cdot\|_{\xi,\zeta}$, where 
$$\|F_2\|_{\xi,\zeta}:=\sum_{q_1,q_2} |\langle F,q_1\otimes q_2\rangle|\|q_1\|_{\xi,\zeta}\|q_2\|_{\xi,\zeta}$$
whenever $F=\sum_{q_1,q_2}  \langle F,q_1\otimes q_2\rangle q_1\otimes q_2$ and the sum runs over monomials.
This proves the result when $\beta=2$.

Next, we consider the random matrices
$$Y^{N}_i=X^{N}_i+a\,F_i(X^{N}_1,\ldots, X^{N}_d,B^{N}_1,\ldots,B^{N}_k)$$
for $d$ independent GOE matrices $(X^{N}_1,\ldots, X^{N}_d)$ and $m$ deterministic symmetric matrices $B^{N}_1,\ldots, B^{N}_m$.
The Jacobian is slightly changed and reads
$$(\mathcal J G)_{ij,k\ell; t,s}=(\hat\partial_t G_s)_{ik,\ell j}+(\hat\partial_t G_s)_{i\ell,kj} \qquad i,j,k,\ell \in \{1,\ldots, d\},$$
where the second term comes from the fact that $\partial_{X_{\ell k }} X_{\ell k}$ does not vanish (as in the complex case) but is equal to one.
Notice that can write the second term as $\Sigma(\hat\partial_t G_s)$, where $\Sigma$ acts on basic tensor product by 
$$\Sigma (A\otimes B)_{ik,\ell j}:=A_{i\ell}B_{kj}.$$
Considering the logarithm of the determinant of $(I+ \mathcal J G)$, we see that it expands in moments of $\mathcal J G$ as 
\begin{align*}
\log\det(I+ \mathcal J G)&=\Tr_d\Tr\otimes\Tr \log(I+ \mathcal J G)=\sum_{n\ge 1} \frac{(-1)^{n+1} }{n} \Tr_d\Tr\otimes\Tr ( \mathcal J G)^n\\
&=\sum_{n\ge 1} \frac{(-1)^{n+1} }{n} \Tr_d\Tr\otimes\Tr ( \nabla G+\Sigma(\nabla G) )^n\end{align*}
with $\nabla G_{ij,k\ell; t,s}=(\hat\partial_t G_s)_{ik,\ell j}$. When expanding the above moments, 
it turns out that the moments with an odd number of $\Sigma$ result
 into the trace of a single polynomial, whereas even numbers result
with tensor products of two traces.  For instance, when $n=1$,
$$\Tr_d\Tr\otimes\Tr ( \Sigma(\nabla G) )=\sum_t \sum_{i,j} (\hat\partial_t G_t)_{ij,ji}=\sum_t \Tr(m(\hat\partial_t G_t))$$
whereas $\Tr_d\Tr\otimes\Tr ((\nabla G) )=\sum_t \sum_{i,j} (\hat\partial_t G_t)_{ii,jj}$. 
Hence, also in this case there exist
convergent series $F_1,F_0$ such that 
\begin{align*}
\log  \mbox{Jac} G(Y_1,\ldots,Y_d, B_1,\ldots, B_m)&=\Tr\otimes \Tr\, F_0+\Tr\, F_1\\
&=\Tr\otimes \Tr\,\Bigl(F_0+\frac{1}{2N} (F_1\otimes \operator{Id}+\operator{Id}\otimes F_1)\Bigr),
\end{align*}
and we conclude as before.

\section{Appendix: Concentration Lemma}\label{secconc}
In this section we prove Lemma \ref{conc}. As already mentioned, it follows from standard results on concentration of measure.

Indeed, thanks to Gromov, it is well known that the groups
$$SU(N):=\{U\in U(N):\det(U)=1\},\qquad SO(N):=O(N)\cap SU(N)$$
can be seen as submanifolds of the set of $N\times N$ matrices that have a Ricci curvature bounded below by $\beta(N+2)/4-1$,
see e.g. \cite[Theorem 4.4.27]{AGZ} and \cite[Corollary 4.4.31]{AGZ}. In particular, this implies concentration of measure under the Haar measures on these
groups. To lift this result to $\mathbb Q_{\beta,N}^V$, let us first notice that, by definition, the potential
$V$ is balanced, in the sense that 
it is invariant under the maps $U_j\mapsto U_j e^{i\theta_j}$ for any $\theta_j\in [0,2\pi)$, being a sum of words each one
containing the same number of letters $U_i$ and $U_i^*$.
Recalling that $\mathbb Q_{\beta,N}^V$ is a measure on $O(N)$ (resp. $U(N)$) when $\beta=1$ (resp. $\beta=2$), it follows that, for any balanced polynomial $P$,
$$\mathbb Q_{\beta,N}^V\Bigl(\bigl|\tr(P)-\mathbb Q_{\beta,N}^V\bigl(\tr(P)\bigr)\bigr|\ge\delta\Bigr)=
\tilde{\mathbb Q}_{\beta,N}^V\Bigl(\bigl|\tr(P)-\tilde{\mathbb Q}_{\beta,N}^V\bigl(\tr(P)\bigr)\bigr|\ge\delta\Bigr),$$
where $\tilde{\mathbb Q}_{\beta,N}^V$ is the restriction of $\mathbb Q_{\beta,N}^V$ to $SO(N)$ (resp. $SU(N)$) when $\beta=1$ (resp. $\beta=2$).

On the other hand, if $P$ is a word which is not balanced and we write $U_j$ has $U_j=e^{i\theta_j} \tilde U_j$ with $\tilde U_j$ in $SU(N)$, then $\tr \,P(U)=e^{i\theta} \tr\, P(\tilde U)$
for some $\theta$ which is a linear combination of the $\theta_j$. As $\theta_j$ follows the uniform measure on $[0,2\pi]$, we deduce that $\mathbb Q_{\beta,N}^V\bigl( \tr (P)\bigr)=0$. Hence,
if $P$ is not balanced,
$$\mathbb Q_{\beta,N}^V\Bigl(\bigl|\tr(P)-\mathbb Q_{\beta,N}^V\bigl(\tr(P)\bigr)\bigr|\ge\delta\Bigr)=
\tilde{\mathbb Q}_{\beta,N}^V\Bigl(\bigl|\tr(P)\bigr|\ge\delta\Bigr),$$
Therefore in both cases we can use concentration inequalities on the special groups.

We then notice that $N^{1-r}\Tr^{\ot r} V$ has a bounded Hessian,
going to zero when $\|V\|_{\xi,\zeta}$ goes to zero. Hence, we can use Bakry-Emery criterion to conclude that, for any $\xi>1$, if $\|V\|_{\xi,\zeta}$ is small enough then
\begin{equation}\label{conc2}
\mathbb Q_{\beta,N}^V\Bigl(\bigl|\tr(P)-\mathbb Q_{\beta,N}^V\bigl(\tr(P)\bigr)\bigr|\ge\delta\Bigr)\le 2e^{-\frac{\beta}{8\|P\|_{\La}^2}\delta^2 }\,,\end{equation}
where $\|P\|_{\La} $ is the Lipschitz constant of $\tr P$, which can be bounded as $$\|P\|^2_{\La}\le
\sup_{u_j,u_j^*,a_j} \sum_{i=1}^d \tau\bigl(|\mathcal D_iP|^2(u_j,u_j^*,a_j)\bigr)$$
where the supremum is taken over all unitary operators $u_i$, all operators $a_i$ with norm bounded by $1$, and all tracial states $\tau$.
Note that
if $P$ is a word we simply have $\|P\|_{\La}\le \mbox{deg}_U(p)$, and more in general
$$\|P\|_{\La}\le \sum |\langle P, q\rangle|\,{\rm deg}_U(q)\le C_\xi\|P\|_{\xi,1},$$
where $C_\xi$ is a finite constant so that $s\le C_\xi\, \xi^s$ for all $s\in\mathbb N $. 
Therefore, thanks to \eqref{conc2} we deduce that, for any monomials $q_1,\ldots, q_k$,
\begin{equation}\label{conc3}
\left|\mathbb Q_{\beta,N}^V\biggl(\prod_{\ell=1}^k \left(\tr(q_\ell)-\mathbb Q_{\beta,N}^V(\tr(q_\ell))\right)\biggr)\right|\le C_k \prod_{\ell=1}^k {\rm deg}_U(q_\ell)\,.
\end{equation}
As correlators can be decomposed as the sum of products of such moments, it follows that for any words $q_1,\ldots,q_k$ and any $\xi>1$
$$\left|\W_{kN}^V(q_1,\ldots,q_k)\right|\le C_k \prod_{\ell=1}^k {\rm deg}_U(q_\ell)\le C_k(C_\xi)^k \prod_{\ell=1}^k \|q_\ell\|_{\xi},
$$
which concludes the proof of Lemma \ref{conc}.

\bibliographystyle{amsalpha}
\bibliography{tran2}

\end{document}